\documentclass[a4paper, 12pt]{article}

\usepackage{LatexDefinitions}

\usepackage{extarrows}
\usepackage[all]{xy}
\usepackage{bbold}
\usepackage[font=small,labelfont=bf]{caption}

\usepackage{xcolor} % Required for specifying colors by name
\definecolor{ocre}{RGB}{243,102,25} 
\definecolor{royalblue}{RGB}{0,78,156}

\usepackage{subcaption}

\newcommand*{\Cont}{\mathcal C}% continuous functions
\newcommand*{\Contc}{\Cont_\textup c} % compactly
\newcommand{\bnorm}[1]{\left\lVert#1\right\rVert} % LDsty
\newcommand{\norm}[1]{\lVert#1\rVert} % 
\newcommand{\inv}{^{-1}}
\newcommand{\pair}[2]{\langle#1,#2\rangle}
\newcommand{\bpair}[2]{\left\langle#1,#2\right\rangle}
\newcommand{\setgiven}[2]{\left\{#1\,\middle|\,#2\right\}}
\newcommand{\dd}[2]{\frac{\mathrm{d} #1}{\mathrm{d} #2}}
\newcommand{\textweakstar}{weak\nobreakdash-\(\ast\)}

\newcommand{\OT}{\mathrm{OT}}
\newcommand{\g}{\mathbf{g}}
\newcommand{\tg}{\Tilde{\mathbf{g}}}
\newcommand{\Transpose}{\mathrm{T}}
\newcommand{\AC}{\mathrm{AC}}

\newcommand*{\calM}{\mathcal M}

\newcommand*{\calA}{\mathcal A}
\newcommand*{\calB}{\mathcal B}
\newcommand*{\calS}{\mathcal S}
\newcommand*{\calX}{\mathcal X}

\newcommand*{\Hil}{\mathcal H}

\newcommand*{\prm}{\mathcal P}

\newcommand{\mudot}{\dot{\mu}}
\newcommand{\nudot}{\dot{\nu}}

\newcommand{\alphadot}{\dot{\alpha}}
\newcommand{\betadot}{\dot{\beta}}

\newcommand{\fmumu}{{f_{\mu,\mu}}}
\newcommand{\gmumu}{{g_{\mu,\mu}}}
\newcommand{\fmunu}{{f_{\mu,\nu}}}
\newcommand{\gmunu}{{g_{\mu,\nu}}}

\newcommand{\gnumu}{{g_{\nu,\mu}}}

\newcommand*{\Pall}{P}

\newcommand*{\A}{A}
\newcommand*{\ef}{a}
\newcommand*{\B}{B}
\newcommand{\oh}{{o}}
\newcommand{\Oh}{\mathcal{O}}

\newcommand*{\dS}{\mathsf{d}_S}
\newcommand*{\vdot}{\cdot}

\title{The Riemannian geometry of Sinkhorn divergences}
\author{Hugo Lavenant\thanks{Bocconi University, Department of Decision Sciences and BIDSA, Via Sarfatti 25, 20136 Milan, Italy.}, Jonas Luckhardt\thanks{Georg-August-Universität Göttingen, Institute for Computer Science, Goldschmidstraße 7, 37077, Göttingen, Germany.}, Gilles Mordant\thanks{Georg-August-Universität Göttingen, Institute for Mathematical Stochastics,  Goldschmidstraße 7, 37077, Göttingen, Germany; Now at Yale University, Department of Mathematics, 219 Prospect Street, CT 06511, New Haven, USA.},\\ Bernhard Schmitzer\addtocounter{footnote}{-2}\addtocounter{Hfootnote}{-2}\footnotemark\addtocounter{footnote}{2}\addtocounter{Hfootnote}{2}, Luca Tamanini\thanks{Universit\`a Cattolica del Sacro Cuore, Dipartimento di Matematica e Fisica
"Niccolo Tartaglia", Via della Garzetta 48, I-25133 Brescia, Italy.\\
\emph{E-mail addresses}: hugo.lavenant@unibocconi.it, jonas.luckhardt@uni-goettingen.de, gilles.mordant@uni-goettingen.de / gilles.mordant@yale.edu, schmitzer@cs.uni-goettingen.de, luca.tamanini@unicatt.it}}
\date{\today}

\begin{document}

\maketitle

\begin{abstract}
We propose a new metric between probability measures on a compact metric space that mirrors the Riemannian manifold-like structure of quadratic optimal transport but includes entropic regularization. Its metric tensor is given by the Hessian of the Sinkhorn divergence, a debiased variant of entropic optimal transport. We precisely identify the tangent space it induces, which turns out to be related to a Reproducing Kernel Hilbert Space (RKHS). As usual in Riemannian geometry, the distance is built by looking for shortest paths. We prove that our distance is geodesic, metrizes the \textweakstar~topology, and is equivalent to a RKHS norm. Still it retains the geometric flavor of optimal transport: as a paradigmatic example, translations are geodesics for the quadratic cost on $\R^d$.
We also show two negative results on the Sinkhorn divergence that may be of independent interest: that it is not jointly convex, and that its square root is not a distance because it fails to satisfy the triangle inequality.
\end{abstract}

\newpage
\tableofcontents

\section{Introduction}

Optimal transport is a mathematical optimization problem which allows one to lift a metric $d$ on a space $X$ to the space $\prm(X)$ of probability measures over $X$, interpreted as the average distance that infinitesimal units of mass have to travel when efficiently transforming one measure into the other. When the base space $X$ is a Riemannian manifold, then the space of probability measures can be formally interpreted as a Riemannian manifold as well. 

In recent years there has been an increasing interest in entropic regularization of optimal transport, which adds some stochasticity to the coupling of probability measures. This variant is easier to compute numerically, has better statistical properties, and is more regular as a function of its input measures and ground cost. Its debiased version defines a smooth positive definite loss function on $\prm(X)$. However, it does not define a distance, or even a Riemannian geometry, on $\prm(X)$. 

\medskip

\begin{center}
\fbox{\begin{minipage}{0.9\textwidth}
In this work we aim at combining both worlds: having a Riemannian metric on $\prm(X)$ which lifts the geometry of $X$ in a similar way to what optimal transport does, while retaining the smoothness of the Sinkhorn divergence. 
\end{minipage}}
\end{center} 

\medskip

\noindent To that end, we consider the second-order expansion of the Sinkhorn divergence on the diagonal: this yields a positive definite quadratic form on a meaningful tangent space to the space of probability measures, in other words a Riemannian metric tensor. This metric tensor induces a Riemannian distance on $\prm(X)$ via minimization over lengths of curves, which is aware of the geometry of $X$ while being and thus ``smoother'' than the optimal transport geometry in the sense that the metric tensor is jointly continuous in its inputs. Investigation of the additional properties of our distance, in particular regarding computational and statistical aspects, is postponed to future works. In the rest of this introduction we recall what is meant by the ``geometry of optimal transport'' before summarizing the main results of the present work.    
We emphasize that geodesics for our distance are different from Schrödinger bridges: a more detailed discussion is given on page~\pageref{paragraph:difference_bridges}.

\paragraph{Optimal transport and its geometry.}

On a (compact) space $X$, the value of the optimal transport problem with ground cost $c : X \times X \to [0,\infty)$ between two probability measures $\mu, \nu \in \prm(X)$ will be denoted by $\OT_0(\mu, \nu)$. It is defined as~\cite{Villani2009,Santambrogio2015,AGS2008} 
\begin{equation} \label{eq:intro_def_OT}
    \OT_0(\mu, \nu) \coloneqq \inf_{\pi \in \Pi(\mu,\nu)} \iint_{X \times X} c(x,y) \diff \pi(x,y),
\end{equation}
where \(\Pi(\mu,\nu) = \setgiven{\pi \in \prm(X \times X)}{(\proj_1)_\# \pi = \mu, ~ (\proj_2)_\# \pi = \nu}\) is the set of couplings having marginals $\mu, \nu$ and $(P_i)_\# \pi$ is the push-forward of the measure $\pi$ by the projection $P_i$ on the $i$-th factor. If $c = d^p$ is a power of a distance $d$ on $X$ for some \(p \geq 1\), then $(\mu, \nu) \mapsto {\OT_0(\mu, \nu)}^\frac{1}{p}$ defines a distance on $\prm(X)$, the so-called $p$-Wasserstein or Monge-Kantorovich distance~\cite[Sec. 7]{AGS2008}.

To explain the \emph{geometry} it induces, we restrict to $X$ being a convex subset of a Euclidean space $\R^d$ and $c(x,y) = \| x-y \|^2$, though everything could be extended to Riemannian manifolds. Historically this geometric structure was introduced in~\cite[Sec.~4]{Otto2001} as the quotient of the space of maps by measure preserving maps with applications to dissipative evolution equations, and can be traced back to hydrodynamics~\cite{Arnold1966,Brenier1989}. In this presentation we will rather focus on the ``metric tensor'' that optimal transport defines. If $(\mu_t)_t$ is a curve of probability measures obtained by following the flow of a potential velocity field $v = \nabla \psi$, that is, $\mu_t = \Psi(t,\cdot)_\# \mu_0$ is the pushforward of $\mu_0$ by the flow $\Psi(t,\cdot)$ of the ODE $\dot{x}_t = \nabla \psi(x_t)$, then by~\cite[Prop.~8.5.6]{AGS2008}
\begin{equation} \label{eq:intro_Hessian_OT}
    \lim_{t \to 0} \frac{\OT_0(\mu_0,\mu_t)}{t^2} = \int_X |\nabla \psi(x)|^2 \diff \mu_0 (x).
\end{equation}  
That is, on small distances, the squared Wasserstein distance is quadratic in $v = \nabla \psi$ which parametrizes how the curve moves on $\prm(X)$. As the curve $(\mu_t)_t$ satisfies the continuity equation $\mudot_t = - \ddiv(\mu_t \nabla \psi)$, we can think of the first-order distribution $\mudot_t$ as the tangent vector whose squared length is given by the right-hand side of~\eqref{eq:intro_Hessian_OT}. That is, for $\mu \in \prm(X)$ the formal metric tensor $\g^0$ is
\begin{equation} \label{eq:intro_gmu_OT}
    \g^0_\mu(\mudot, \mudot) = \int_X |\nabla \psi(x)|^2 \diff \mu (x), \qquad \mudot = - \ddiv (\mu \nabla \psi).
\end{equation}
That $\prm(X)$ has such a Riemannian structure means that the squared distance $\OT_0$ can be recovered by minimizing the energy of paths joining $\mu$ to $\nu$: 
\begin{equation} \label{eq:Benamou_Brenier}
    \OT_0(\mu,\nu) = \inf_{(\mu_t)_{t \in [0,1]}} \int_0^1 \g^0_{\mu_t}(\mudot_t, \mudot_t) \diff t
\end{equation}
where the infimum is taken over sufficiently regular curves $(\mu_t)_{t \in [0,1]}$ of measures joining $\mu$ to $\nu$: this is the dynamical formulation of optimal transport, the so-called Benamou-Brenier formula~\cite{BB2000},~\cite[Sec. 6.1]{Santambrogio2015}.

In summary: for $\mu \in \prm(X)$ a ``point'' in the manifold, the tangent vectors at $\mu$ read $\mudot = - \ddiv (\mu \nabla \psi)$ for $\psi : X \to \R$ , and $\g^0_\mu(\mudot, \mudot)$ defined in~\eqref{eq:intro_gmu_OT} is the squared length of the tangent vector \(\mudot\). The proper tangent space is actually the completion of \(\setgiven{-\ddiv(\mu \nabla \psi)}{\psi \in \Cont^\infty(X)}\) with respect to $\g^0_\mu$. It is canonically identified with vector fields $v \in L^2(X,\mu;\R^d)$ which can be approximated in $L^2(X,\mu;\R^d)$ by smooth gradients. Even though this does not define an infinite-dimensional Riemannian manifold in the sense of~\cite{Lang1999}, a precise meaning can still be given thanks to the tools of analysis on non-smooth spaces~\cite[Sec. 8]{AGS2008}.

Once viewing the problem through this lens, we can import many concepts of Riemannian geometry into $\prm(X)$. Geodesics are defined as the minimizers of the right-hand side of~\eqref{eq:Benamou_Brenier} and correspond to McCann's displacement interpolation~\cite{McCann1997}. For two probability measures $\mu, \nu \in \prm(X)$, the Riemannian logarithmic map $\log_\mu(\nu)$ is the initial velocity of the geodesic from $\mu$ to $\nu$. Conversely, given a tangent vector $\mudot = - \ddiv(\mu \nabla \psi)$, the exponential map is $\exp_\mu(\mudot) = (\id + \nabla \psi)_\# \mu$, but the extension to the whole tangent space and the injectivity radius can be tricky to characterize. With logarithmic and exponential maps, one can for instance define a linearized version of optimal transport~\cite{Wang2013,Sarrazin2024} and subsequently perform principal component analysis in tangent space~\cite{Bigot2013}. Several extensions of optimal transport rely on its Riemannian structure, such as barycenters~\cite{AguehCarlier}, gradient flows~\cite{AGS2008}, harmonic maps~\cite{Lavenant2019}, parallel transport~\cite{Gigli2012}, to cite a few examples. As the Riemannian interpretation is mostly formal, giving a rigorous meaning to these concepts usually requires some (hard) work of analysis on non-smooth spaces.

\paragraph{The Sinkhorn divergence.}

For a regularization parameter $\varepsilon > 0$, entropic optimal transport~\cite{Cuturi2013,PeyreCuturi} is given by the minimization problem
\begin{equation} \label{eq:OTeps}
    \OT_\varepsilon(\mu, \nu) \coloneqq \inf_{\pi \in \Pi(\mu, \nu)} \left\{ \iint_{X \times X} c \diff \pi + \varepsilon \KL(\pi \, | \, \mu \otimes \nu) \right\}, 
\end{equation}
where \(\KL\) is the Kullback--Leibler divergence
\begin{equation*}
    \KL(\pi \, | \, \rho) \coloneqq 
    \begin{cases}
        \displaystyle{\int_X h \left( \dd{\pi}{\rho} \right) \diff \rho} &\tn{if } \pi \ll \rho,\, \rho \geq 0,\\
        + \infty &\tn{else}
    \end{cases}
\end{equation*}
with \(h(s) \coloneqq s \log(s) - s + 1\) for \(s > 0\), \(h(0) \coloneqq 1\), \(h(s) \coloneqq + \infty\) for \(s < 0\). As $\varepsilon \to 0$, $\OT_\varepsilon$ converges to $\OT_0$~\cite{Carlier2017,Leonard2012,mikami2004}. Compared to the original transport problem it can be computed faster~\cite{PeyreCuturi} and the statistical complexity suffers much less from the curse of dimensionality~\cite{Genevay:Samples,MenaWeed}. Moreover, as a function of $\mu$ and $\nu$, the cost \(\OT_\varepsilon\) is more regular and its gradients can be computed efficiently~\cite{Genevay:Loss}.

However, even if $c=d^2$ for $d$ a distance on $X$, the function $\sqrt{\OT_\varepsilon}$ is not a distance, as generically $\OT_\varepsilon(\mu, \mu) > 0$ and \(\argmin_{\nu \in \prm(X)} \OT_\varepsilon(\mu,\nu) \neq \mu\). For this reason~\cite{Genevay:Loss} introduced the Sinkhorn divergence, or debiased entropic optimal transport:
\begin{equation} \label{eq:intro:def_sinkhorn_div}
    S_\varepsilon(\mu,\nu) \coloneqq \OT_\varepsilon(\mu, \nu) - \frac{1}{2} \OT_\varepsilon(\mu, \mu) - \frac{1}{2} \OT_\varepsilon(\nu, \nu).
\end{equation}
By construction it satisfies $S_\varepsilon(\mu,\mu) =0$, in fact more generally $S_\varepsilon(\mu,\nu) \geq 0$ with equality if and only if $\mu=\nu$, and it also converges to $\OT_0$ as $\varepsilon \to 0$. In~\cite{Feydy2018} the authors prove the following theorem which justifies its use in machine learning as a loss function:

\begin{theorem}[{\cite[Thm.~1]{Feydy2018}}]
    \label{theorem:FeydyEtAl}
    When \(k_c=\exp(-c/\varepsilon)\) is a positive definite, universal, jointly Lipschitz continuous kernel over a compact space \(X\), then \(S_\varepsilon\) defines a symmetric, positive definite and smooth loss function that is convex in each of its input variables. It also metrizes the convergence in law.
\end{theorem}

\noindent For the definition of  positive definiteness and  universality of the kernel see Section~\ref{section:sub:RKHS_introduction}. The assumptions on the kernel are satisfied in particular for the quadratic cost on a Euclidean domain. Nevertheless, $\sqrt{S_\varepsilon}$ is not a distance, because it does not satisfy the triangle inequality. Although this is well known in the community, we have not been able to find an explicit statement of this property and consequently provide an elementary proof in Section~\ref{subsection:Seps_triangle} of this article.

\paragraph{Outline of the contributions.} 

As mentioned earlier we seek to build a Riemannian distance out of the loss function $S_\varepsilon$. We restrict to $X$ being a compact metric space with a continuous cost function \(c \in \Cont(X \times X)\). Similarly to~\cite{Feydy2018}, we assume that it induces a positive definite universal kernel \(k_c \assign \exp(-c/\varepsilon)\) on \(X \times X\) (see~\cite{Micchelli2006, Sriperumbudur2011}). 

Throughout this article we consider some fixed $\varepsilon > 0$ and we will briefly comment the limits $\varepsilon \to 0$ and $\varepsilon \to \infty$ in Section~\ref{sec:sub:epsilon_limits}.

\medskip

\emph{The Hessian of the Sinkhorn divergence.}
To construct a Riemannian structure based on the Sinkhorn divergence, we follow an approach from information geometry. For each point $\mu \in \prm(X)$ we build a metric tensor, that is, a positive definite quadratic form $\g_\mu$ on some appropriate tangent space to $\prm(X)$ at $\mu$, by considering the Hessian of $S_\varepsilon(\mu,\nu)$ on the diagonal $\mu=\nu$.
Applying this strategy to the Kullback--Leibler divergence in place of the Sinkhorn divergence, as in the pioneering works of Hotelling~\cite{Hotelling1930} and Rao~\cite{Rao1945}, leads to the Fisher information metric.
In Theorem~\ref{theorem:hessian_sinkhorn} we derive the formula for the Hessian of $S_\varepsilon$ along paths \((\mu_t)_t\) in \(\prm(X)\) following a $\Cont^m$-perturbation, meaning that $\mudot_t$ exists and is continuous as a linear form on $\Cont^m(X)$ (see Definition~\ref{def:Cm_perturbations}). 
The case of a ``horizontal'' perturbation \(\mu_t = \Psi(t,\vdot)_\# \mu_0\) corresponds to $m = 1$, while a ``vertical perturbation'' $\mu_t = \mu_0 + t \nu$ with $\nu \in \calM_0(X)$ corresponds to $m=0$. 
It requires the cost function to be of class $\Cont^m$, with the setting $m=0$ being the one which encompasses continuous cost functions on compact metric space.
Writing \((f_{\mu_t,\mu_s}, g_{\mu_t,\mu_s})\) for the entropic dual variables between \(\mu_t\) and \(\mu_s\) (i.e.\ the entropic analogues of Kantorovich potentials, see Section~\ref{section:preliminaries}), the formula reads
\begin{equation} \label{eq:intro:hessian_sinkhorn}
    \lim_{t \to 0} \frac{S_\varepsilon(\mu, \mu_t)}{t^2} = \frac{1}{2} \bpair{\mudot}{\frac{\partial f_{\mu,\mu_s}}{\partial s} \Big|_{s=0}} = \frac{\varepsilon}{2} \bpair{ \mudot }{ (\id - K_\mu^2)^{-1} H_\mu \left[ \mudot \right] }.
\end{equation}
Here \(K_\mu\) and \(H_\mu\), introduced in Definition~\ref{def:kernel_operators}, are integral operators induced by the self- transport kernel 
\begin{equation*}
    k_\mu \assign \exp \left( \frac{1}{\varepsilon} (\fmumu \oplus \fmumu - c) \right).
\end{equation*}
The result is not completely new and can be related to works studying the central limit theorem for the Sinkhorn divergence~\cite{Goldfeld2022,GonzalezSanz2022}, where the ``perturbations'' $\mu_t$ are estimations of the measure $\mu$ by i.i.d.\ samples.

\emph{The metric tensor and its canonical tangent space.}
We \emph{define} the metric tensor $\g_\mu$ as the right-hand side of~\eqref{eq:intro:hessian_sinkhorn}, that is, at least for a signed measure $\mudot$ which is ``balanced'' in the sense $\pair{\mudot}{\ones_X} = 0$,
\begin{equation*}
    \g_\mu(\mudot, \mudot) =  \frac{\varepsilon}{2} \bpair{ \mudot }{ (\id - K_\mu^2)^{-1} H_\mu \left[ \mudot \right] }.
\end{equation*}
We then want to understand for which class of ``tangent vectors'' $\mudot$ this can be extended. This leads us to the most technical part of the paper in Section~\ref{section:FA_regularity}.  Whereas we already saw that distributions of order $m$ should be allowed if the cost function is $\Cont^m$, it turns out that the right regularity condition of the perturbation is expressed in terms of the Reproducing Kernel Hilbert Space (RKHS) \(\Hil_\mu\) induced by the self-transport kernel \(k_\mu\). Specifically, $\g_\mu$ defines a positive definite quadratic form on the \emph{dual} $\Hil_\mu^*$ of the RKHS $\Hil_\mu$, restricted to functionals $\mudot$ such that $\pair{\mudot}{\ones_X} = 0$. On this subspace of $\Hil_\mu^*$ the norm induced by the metric tensor $\g_\mu$ is equivalent to the operator norm. For a smooth cost function $c$ this space contains all $\Cont^m$-perturbations, including the ``vertical'' and ``horizontal'' ones. In contrast, the tangent space for the optimal transport metric tensor~\eqref{eq:intro_gmu_OT} only allows for horizontal perturbations.

\emph{A useful change of variables.}
In order to build a distance from the metric tensor $\g_\mu$ and to prove the existence of geodesics we need some uniform coercivity. To that end we exploit the change of variables $\mu \leftrightarrow \beta \assign \exp(- \fmumu/\varepsilon)$ which maps $\prm(X)$ into a subset of the unit sphere of the Reproducing Kernel Hilbert Spaces $\Hil_c$ induced by \(k_c \assign \exp(-c/\varepsilon)\) (Theorem~\ref{theorem:alpha_beta_homeo}). This change of variables is not completely new and can be seen implicitly in~\cite[Prop.~3]{Feydy2018} when the authors derive lower bounds on the Sinkhorn divergence. For a perturbation $\betadot$ of the variable $\beta = \exp(- \fmumu/\varepsilon)$ and its associated perturbation $\mudot$ of the measure $\mu$, Theorem~\ref{theorem:metric_tensor_in_betadot} expresses \(\g_\mu\) as
\begin{equation} \label{eq:intro_tgmu_expression}
    \g_\mu(\mudot,\mudot) = \tg_{\mu}(\betadot,\betadot) \assign \frac{\varepsilon}{2} \left( \norm{\betadot}_{\Hil_c}^2 + 2 \bpair{\exp(\fmumu/\varepsilon) \cdot \betadot }{ (\id - K_{\mu})\inv [\exp(\fmumu/\varepsilon) \cdot \betadot] }_{\mu} \right).
\end{equation}   
Expressed in this space $\Hil_c$, independent of $\mu$, the metric tensor $\tg_{\mu}(\betadot,\betadot)$ is jointly continuous in $\mu$ and $\betadot$ (Proposition~\ref{prop:tgmu_joint_cont}) and equivalent to $\| \betadot \|^2_{\Hil_c}$ the squared norm on $\Hil_c$ uniformly in $\mu$ (Proposition~\ref{prop:tgmu_Hcalphadot_equiv}). This is in stark contrast with the metric tensor $\g^0_\mu$ of the optimal transport problem~\eqref{eq:intro_gmu_OT}, which depends in a non-smooth way on $\mu$ and which does not compare well to any other functional space uniformly on $\prm(X)$.

\emph{Scaling limit.} In Section~\ref{sec:sub:epsilon_limits} we verify, though without rigorous derivations, that the scaling limits of our metric tensor when $\varepsilon \to 0$ and $\varepsilon \to \infty$ indeed correspond to what is expected, respectively the metric tensor of Optimal Transport and the one of a degenerate MMD distance.

\emph{Building a new distance.}
Once we understand well our metric tensor $\g_\mu$, we build our new metric on the space of probability measures in Section~\ref{section:path_metric}. While in the classical case the ``dynamic'' formula~\eqref{eq:Benamou_Brenier} is a consequence of the ``static'' definition~\eqref{eq:intro_def_OT}, here we take it as definition of the distance. Specifically, in Definition~\ref{def:distance} we set
\begin{equation} \label{eq:intro:def_dS}
    \dS(\mu_0,\mu_1) \assign \left( \inf \int_0^1 \g_{\mu_t}(\mudot_t,\mudot_t) \diff t \right)^\frac{1}{2},
\end{equation}
where the infimum is taken over a suitable set of admissible paths \((\mu_t)_{t \in [0,1]}\) between \(\mu_0\) and \(\mu_1\). Theorem~\ref{theorem:dS_metric} shows that \(\dS\) is a metric which metrizes the \textweakstar~topology on \(\prm(X)\). We actually prove more: \(\dS(\mu_0,\mu_1)\) is uniformly equivalent to \(\norm{\beta_0 - \beta_1}_{\Hil_c}\). This is to be compared to optimal transport, where $\sqrt{\OT_0}$ can be compared to an $H^{-1}$ norm as in~\cite{Peyre:Comparison} and~\cite[Sec. 5.5.2]{Santambrogio2015}, but these results require some regularity assumptions on the measures involved and are not valid uniformly on $\prm(X)$. Estimates can also be made with linearized optimal transport~\cite{DelalandeMerigot}, though this only leads to a bi-Hölder equivalence. The price to pay is that our constants depend on $\varepsilon$ and blow up as $\varepsilon \to 0$.

\emph{Existence of geodesics.}
In Theorem~\ref{theorem:dE_minimizers_existence} we prove existence of geodesics, that is, that the infimum in~\eqref{eq:intro:def_dS} is attained. In most of the problems which are variations of the Benamou-Brenier formula, such as unbalanced optimal transport~\cite{Chizat:UOT}, variational mean field games~\cite{Benamou:MFG} or harmonic maps~\cite{Lavenant2019}, the existence of solutions to problems of the type~\eqref{eq:intro:def_dS} uses the property that $\g^0_\mu(\mudot, \mudot)$ (or its suitable counterpart) is jointly convex in $\mu$ and $\mudot$. Thus tools of convex analysis can be used and lower semicontinuity is easy to establish. In our case, however, neither the Sinkhorn divergence $S_\varepsilon(\mu, \nu)$ nor the associated metric tensor $\g_\mu(\mudot,\mudot)$ are jointly convex: we provide counterexamples that we defer to Section~\ref{section:sub:nonconvexity}. To prove that the infimum in~\eqref{eq:intro:def_dS} is attained we still rely on the direct method of calculus of variations, but we have to work a little bit harder: coercivity comes from the change of variables $\mu \leftrightarrow \beta$ studied in Section~\ref{section:FA_regularity}, and lower semicontinuity is proved thanks to the smoothness of the metric tensor in its inputs.

\emph{Absolutely continuous paths.}
The class of admissible paths in the minimization problem~\eqref{eq:intro:def_dS} defining the metric \(\dS\) is initially defined via Sobolev regularity of the corresponding path \( (\beta_t)_t\) in \(\Hil_c\). These paths end up being precisely the 2-absolutely continuous paths in \((\prm(X),\dS)\). We give a characterization of admissibility in terms of the path \((\mu_t)_t\) itself and some regularity of the potentials in Proposition~\ref{prop:admissible_paths}. 
If $c$ is smooth we show in Corollary~\ref{cor:Cm_admissible} that \(\Cont^m\)-perturbations, for which the metric tensor coincides with the Hessian of the Sinkhorn divergence, are admissible paths. We further show that 
\begin{equation*}
    \dS(\mu_0, \mu_1) \leq C \cdot W_1(\mu_0, \mu_1), \quad \text{and} \quad \dS(\mu_0, \mu_1) \leq C \cdot \mathsf{SHK} (\mu_0, \mu_1)
\end{equation*}
for some \(C > 0 \) depending on $\varepsilon$, where $W_1$ denotes the $1$-Wasserstein distance and \(\mathsf{SHK}\) the spherical Hellinger--Kantorovich distance (see~\cite{LaMi2017}). In particular, 2-absolutely continuous paths in the \(p\)-Wasserstein metric for any \(p \geq 1\) are admissible paths for \(\dS\).

\emph{The quadratic case.}
Our theory applies to \emph{any} compact metric space with \emph{any} continuous cost function as soon as $k_c = \exp(-c / \varepsilon)$ defines a positive definite universal kernel. The special case of the squared Euclidean distance as the cost function on a subset of \(\R^d\) is studied in Section~\ref{section:sub:mean_estimate}. For measures $\mu_0$, $\mu_1$ it has been known for a while (see e.g.~\cite[Lemma 8.8]{BickelFreedman} or the proof of Theorem 2.1 in~\cite{Gelbrich1990}) that the following decomposition holds: 
\begin{equation*}
    \OT_0(\mu_0,\mu_1) = \norm{m_1 - m_0}^2 + \OT_0(\bar{\mu}_0,\bar{\mu}_1),
\end{equation*}
where $m_i$ denotes the mean of the measure $\mu_i$ and $\bar{\mu}_i = (\id- m_i)_\# \mu_i$ is the centered version of $\mu_i$, for $i=0,1$. In other words, the transport of the means and the transport of the centered parts decouple. We show that this extends to Sinkhorn divergences and that it can be read as an orthogonal decomposition of the tangent space for the geometry of $\g_\mu$ in Proposition~\ref{prop:gmu_mean}. We then integrate this inequality and show in Theorem~\ref{theorem:Pythagoras_distance} that our novel distance satisfies
\begin{equation*}
    \dS(\mu_0,\mu_1)^2 = \norm{m_1 - m_0}^2 + \dS(\bar{\mu}_0,\bar{\mu}_1)^2.
\end{equation*}
In particular, uniform translations are geodesics, as in the case of classical optimal transport.

\emph{Some special cases.}
Section~\ref{section:examples} covers two examples where we can explicitly compute the metric tensor. We start with Gaussians: it is one of the few examples in optimal transport for which there are closed formulas and for which the geometry is well understood~\cite{Takatsu2011}; and it is one of the paradigmatic examples in information geometry which relates Gaussian measures to the Poincaré half-plane, see e.g.~\cite{Costa2015}. Building on an explicit formula for \(\OT_\varepsilon\) between Gaussians from~\cite{Mallasto2021}, Theorem~\ref{theorem:gaussians_gmu} shows that 
\begin{equation*}
    \g_{v} ( \dot{v}, \dot{v} ) = \frac{1}{ 4 \sqrt{ \frac{\varepsilon^2}{16} + v^2} } \dot{v}^2.
\end{equation*} 
Here we identify a one-dimensional Gaussian measure \(\mathcal{N}(0,v)\) with its variance \(v\). Theorem~\ref{theorem:gaussians_dE} then provides a formula for the geodesic restriction of \(\dS\) to Gaussian measures. Note that strictly speaking we are not in our framework since Gaussians are not supported on a compact space. As a second example, we consider the two-point space $X = \{ x_1,x_2 \}$, where we can compute the metric tensor. In particular, in Proposition~\ref{prop:example:two_point} we study the asymptotic of the distance in two different regimes, namely $c(x_1,x_2) \ll \varepsilon$ and $c(x_1,x_2) \gg \varepsilon$. The formula works for the squared distance cost function and asymptotically recovers the cost of horizontal movement on \(\R^d\) as \(\norm{x_1 - x_2} \to 0\).

\emph{Negative results.}
Section~\ref{section:negative_examples} collects some negative results on the Sinkhorn divergence and the metric tensor that were already mentioned throughout the introduction. Specifically, we provide examples showing that \(\sqrt{S_\varepsilon}\) does not satisfy the triangle inequality, that $(\mu, \nu) \mapsto S_\varepsilon(\mu,\nu)$ is not jointly convex, and that the metric tensor \(\g_\mu(\mudot,\mudot)\) is also not jointly convex in \(\mu\) and \(\mudot\).

\paragraph{Comparison with Schrödinger bridges.}
\label{paragraph:difference_bridges}
We emphasize that our geometry and its geodesics are fundamentally different from Schrödinger bridges, a notion of displacement interpolation associated with entropic optimal transport \cite{Gentil2017,Leonard2014}.
In our construction, given two probability measures $\mu_0, \mu_1$ and $\varepsilon > 0$ there exists $(\mu^{\text{geo}, \varepsilon}_t)_{t \in [0,1]}$ which is a geodesic in our metric $\dS$, i.e.~a minimizer in~\eqref{eq:intro:def_dS}. In contrast, when $c = \norm{x-y}^2$ is the quadratic cost on $\R^d$, the Schrödinger bridge $(\mu^{\text{Sbr}, \varepsilon}_t)_{t \in [0,1]}$ is constructed as follows~\cite{Leonard2014}: Let $\pi_\varepsilon$ be the optimal transport plan, that is, the minimizer in~\eqref{eq:OTeps}. Let $q_{t,\varepsilon}(x,y)$ denote the law of the Brownian bridge at time $t$ starting at $t=0$ in $x$ and ending at $t=1$ in \(y\) with diffusivity $\sqrt{\varepsilon/2}$. That is, $q_{t,\varepsilon}(x,y) \sim \mathcal{N}((1-t)x + ty, t (1-t) \varepsilon /2)$. Then $\mu^{\text{Sbr}, \varepsilon}_t$ is given by
\begin{equation} \label{eq:definition_Schrödinger_bridges}
    \mu^{\text{Sbr}, \varepsilon}_t = \iint_{X \times X} q_{t,\varepsilon}(x,y) \diff \pi_\varepsilon(x,y). 
\end{equation}
It is also a curve which satisfies $\mu^{\text{Sbr}, \varepsilon}_0 = \mu_0$, $\mu^{\text{Sbr}, \varepsilon}_1 = \mu_1$ and which interpolates smoothly in between.
Nevertheless, the two notions of interpolation are conceptually very different.
Let us point out three major differences between $\mu^{\text{geo}, \varepsilon}$ and $\mu^{\text{Sbr}, \varepsilon}$:
\begin{enumerate}
    \item The geodesic interpolation $\mu^{\text{geo}, \varepsilon}$ is defined on any compact space $X$ if the continuous cost function $c$ defines a universal positive definite kernel $k_c = \exp(-c/\varepsilon)$ whereas Schrödinger bridges need the quadratic cost on $\R^d$, or more generally a reference continuous-time Markov process (the Brownian motion in the case of quadratic cost) to be defined~\cite{Leonard2014}. 

    \item If $\mu_0 = \mu_1 = \mu$, then the geodesic interpolation $\mu^{\text{geo}, \varepsilon}$ is constant in time and coincides with $\mu$. In the same case $\mu^{\text{Sbr},\varepsilon}$ is \emph{not} constant in time, and typically has larger variance for intermediate $t \in (0,1)$ than for $t=0,1$. This is linked to the phenomenon that generically $\OT_\varepsilon(\mu,\mu) > 0$; see also the discussion in \cite{Leonard2013}, after Theorem 1.7.
    
    \item In both cases for the quadratic cost the parameter $\sqrt{\varepsilon}$ acts as a typical length. In the case of Schrödinger bridges, it is also related to a typical time, whereas for geodesics there is no typical time. To understand this, for $\tau \in (0,1)$, let us look at the curve $(\mu_t)_{t \in [0,\tau]}$ but rescaled to be over $[0,1]$. That is, let us consider $(\mu^{\text{geo}, \varepsilon}_{t\tau})_{t \in [0,1]}$ and $(\mu^{\text{Sbr}, \varepsilon}_{t\tau})_{t \in [0,1]}$. Then $(\mu^{\text{geo}, \varepsilon}_{t\tau})_{t \in [0,1]}$ is a geodesic from $\mu_0$ to $\mu_\tau$ for any value of $\tau \in (0,1)$, as usual in Riemannian geometry. On the other hand, $(\mu^{\text{Sbr}, \varepsilon}_{t\tau})_{t \in [0,1]}$ is a Schrödinger bridge between its endpoints $\mu^{\text{Sbr},\varepsilon}_0$ and $\mu^{\text{Sbr},\varepsilon}_\tau$, but with a regularization parameter $\varepsilon \tau$. This can be seen by using that the Schrödinger bridge is the marginal of a process which minimizes an entropy over the space of measures on paths~\cite{Leonard2014}, and to use the concatenation property of the entropy for Markov processes in this space~\cite[Lemma 3.4]{Benamou2018}. The key insight is that, when rescaled from time $[0,\tau]$ to time $[0,1]$, a Brownian motion with diffusivity $\sqrt{\varepsilon/2}$ becomes a Brownian motion with diffusivity $\sqrt{\tau \varepsilon/2}$. So in particular as $\tau \to 0$, that is, as we ``zoom in'' in time on a Schrödinger bridge, it increasingly resembles a Schrödinger bridge with $\varepsilon = 0$, that is, a Wasserstein geodesic~\cite{Leonard2012,mikami2004}. 
\end{enumerate}
We also refer to Figure~\ref{fig: Gauss} with Gaussian measures where we can do explicit computations for our geodesics and Schrödinger bridges.

\paragraph{Setting and notation.}
We consider a compact metric space \((X,d)\). We fix \(\varepsilon > 0\) and a symmetric, 
continuous, non-negative cost function \(c \in \Cont(X \times X; [0,\infty))\) that induces a positive definite universal kernel \(k_c(x,y) = \exp(-c(x,y)/\varepsilon)\).

We denote the dual pairing between a space and its dual space by \(\pair{\vdot}{\vdot}\) and write \(\pair{\vdot}{\vdot}_\Hil\) for the inner product on a Hilbert space \(\Hil\).

We work with the space of signed Radon measures \(\calM(X)\), which is identified with the dual space of \(\Cont(X)\) through the Riesz--Markov--Kakutani Theorem. We denote the subset of non-negative measures by \(\calM_+(X)\), the space of ``balanced measures'' \(\nu \in \calM(X)\) with \(\int \ones_X \diff \nu = 0\) by \(\calM_0(X)\), and Radon probability measures by \(\prm(X)\).

When \(X\) is the closure of a bounded open set in \(\R^d\) we write $\Cont^m(X)$ for the space of $m$ times differentiable functions over $X$. It is endowed with the norm \(\norm{\phi}_{\Cont^m} = \sup_{\abs{\alpha} \leq m} \norm{\partial^\alpha \phi}_\infty\) of uniform convergence of the function and all its (partial) derivatives up to order $m$ denoted using a multi-index \(\alpha \in \N_0^d\). We denote by \(\Cont^{(m,m)}(X,X)\) the space of functions for which mixed partial derivatives up to order $m$ in the first and second variable separately are continuous. With these notations $\Cont^0(X) = \Cont(X)$ and $\Cont^{(0,0)}(X) = \Cont(X \times X)$. We define $\Cont^m(X)^*$ as the dual space of $\Cont^m(X)$. It is endowed with the dual norm $\| \cdot \|_{\Cont^{m,*}}$ defined as $\| \nu \|_{\Cont^{m,*}} = \sup \{ \bpair{\nu}{\phi}  :   \norm{\phi}_{\Cont^m} \leq 1\}$. In addition to the strong topology on $\Cont^m(X)^*$ coming from the dual norm $\| \cdot \|_{\Cont^{m,*}}$, we can also endow it with the \textweakstar~topology. As $X$ is compact, note that $\Cont^0(X)^* = \calM(X)$ and \textweakstar~convergence in $\Cont^0(X)^*$ is the usual \textweakstar~convergence of measures. In this way, setting \(m=0\) throughout this paper recovers the assumption $X$ compact and $c$ is continuous.

For Landau big-O and little-o notation we use \(\Oh\) and \(\oh\) respectively.

\section{Preliminaries on entropic optimal transport} \label{section:preliminaries}

The entropic optimal transport problem, as defined in~\eqref{eq:OTeps}, is a constrained convex optimization problem whose dual formulation reads as
\begin{equation} \label{eq:entropic_dual}
    \OT_\varepsilon(\mu, \nu) = \sup_{f, g \in \Cont(X)} \, \pair{\mu}{f} + \pair{\nu}{g} - \varepsilon \bpair{\mu \otimes \nu}{\exp \left( \frac{1}{\varepsilon} (f \oplus g - c) \right) - 1}.
\end{equation}
Here the notation $f \oplus g$ stands for the function $(x,y) \mapsto f(x)+g(y)$ defined on $X \times X$. The dual problem admits maximizers \(\fmunu\), \(\gmunu\), which we call \emph{Schrödinger potentials}, that satisfy
\begin{equation} \label{eq:entropic_dual_max}
    \OT_\varepsilon(\mu, \nu) = \int_X \fmunu \diff \mu + \int_X \gmunu  \diff \nu
\end{equation}
and are solutions of the Schrödinger system
\begin{equation} \label{eq:Schroedinger_system}
\begin{cases}
    \fmunu & = T_\varepsilon(\gmunu,\nu), \\
    \gmunu & = T_\varepsilon(\fmunu, \mu),
\end{cases}
\end{equation}
where \(T_\varepsilon : C(X) \times \prm(X) \to C(X)\) is defined by 
\begin{equation} \label{eq:def:T_eps}
    T_\varepsilon(f,\mu)(y) \coloneqq - \varepsilon \log \int_X \exp \left( \frac{1}{\varepsilon} (f(x) - c(x,y)) \right) \diff \mu(x). 
\end{equation} 
We will sometimes write the condition $ \gmunu = T_\varepsilon(\fmunu, \mu)$ as
\begin{equation} \label{eq:kernel_slice_int_1}
    \int_X \exp \left( \frac{1}{\varepsilon} (\fmunu(x) + \gmunu(y) - c(x,y)) \right) \diff \mu (x) = 1
\end{equation}
for every \(y \in X\). The Schrödinger potentials are unique in the space \(\Cont(X) \times \Cont(X) / \sim\) with the relation \((f,g)~\sim~(f~+~\lambda\ones_X,~g~-~\lambda\ones_X)\) for all \(\lambda \in \R\). Alternatively, we can select a canonical representative by setting \(\fmunu(x_0) = \gmunu(x_0)\) for some fixed \(x_0 \in X\), which we do throughout the paper.
This choice yields \(\fmunu = \gnumu\) and, in particular, throughout this article \(\fmumu = \gmumu\). The optimal potentials inherit the modulus of continuity from the cost function \(c\)~\cite[Lem.~3.1]{Nutz2022}. The optimal transport plan \(\pi\) in~\eqref{eq:OTeps} is given by 
\begin{equation} \label{eq:formula_coupling_optimal}
    \pi = \exp \left( \frac{1}{\varepsilon} (\fmunu \oplus \gmunu - c) \right) (\mu \otimes \nu). 
\end{equation}
Note that the function \(\fmunu \oplus \gmunu\) does not depend on the choice of the representatives since constant shifts cancel out. Optimal \((\fmunu,\gmunu)\) can be obtained numerically by iterating the fixed point equation~\eqref{eq:Schroedinger_system}, which corresponds to the celebrated Sinkhorn algorithm. We refer to~\cite{PeyreCuturi} for an overview on computational optimal transport, and to \cite{Leonard2014, Nutz:IntroEOT} for a detailed discussion on entropic optimal transport and related literature.

\section{The Hessian of the Sinkhorn divergence} \label{section:sinkhorn_hessian}

\subsection{Statement of the expansion}

Theorem~\ref{theorem:FeydyEtAl} shows that the Sinkhorn divergence~\eqref{eq:intro:def_sinkhorn_div} is positive definite and convex in each of its input variables. However, as we show in Section~\ref{subsection:Seps_triangle}, \(S_\varepsilon\), or more precisely \(\sqrt{S_\varepsilon}\), in general does not satisfy the triangle inequality. To construct the metric \(\dS\) announced in~\eqref{eq:intro:def_dS} in the introduction, we first characterize how \(S_\varepsilon\) behaves close to the diagonal \(\mu = \nu\). Specifically, we want to find an expansion of the form
\begin{equation} \label{eq:Seps_expansion}
    S_\varepsilon(\mu, \mu_t) \sim t^2 \g_{\mu}(\mudot, \mudot)
\end{equation}
as \(t \to 0\) for a suitably differentiable path \((\mu_t)_t\) in \(\prm(X)\) with \(\mu_0 = \mu\), \(\mudot_0 = \mudot\). The term $\g_{\mu}(\mudot, \mudot)$ will then serve as the metric tensor and  it will be examined in detail in Section~\ref{section:FA_regularity}. By Theorem~\ref{theorem:FeydyEtAl} there is a good hope that \(\g_\mu\) will be a positive quadratic form on a meaningful tangent space to probability measures. The appropriate notion of tangent space will become apparent in Section~\ref{section:FA_regularity}. 

Though in our general setting $X$ is a compact space, in this section we will also consider a differentiable setting in which stronger results can be obtained.  

\begin{assumption}[Differentiable setting] \label{asp:diff_m}
    We fix an integer $m \geq 0$. If $m = 0$, we simply assume, as stated above, that $X$ is a compact metrizable space and $c \in \Cont(X \times X)$ is a continuous function. If $m \geq 1$ we assume that $X$ is the closure of a bounded open set in \(\R^d\) and that $c \in \Cont^{(m,m)}(X \times X)$.
\end{assumption}

\begin{definition}[$\Cont^m$-perturbations] \label{def:Cm_perturbations}
    Let $m \geq 0$ and \(\tau > 0\). We call a path \((\mu_t)_{t \in (-\tau,\tau)}\) valued in \(\prm(X)\) a $\Cont^m$-perturbation of \(\mu = \mu_0\) if it is continuously differentiable in the space $\Cont^m(X)^*$ endowed with the \textweakstar~topology. We write $\mudot \in \Cont^m(X)^*$ for the derivative of the curve at time $t=0$.
\end{definition}

As $\Cont^m(X)^*$ is endowed with its \textweakstar~topology,
it means that for any $\phi \in \Cont^m(X)$ the map $t \mapsto \langle \mu_t,\phi \rangle$ is differentiable, that the derivative, denoted by $\langle \dot{\mu}_t,\phi \rangle$, is linear in $\phi$, and that the mapping $(\phi,t) \mapsto \langle \dot{\mu}_t,\phi \rangle$ is jointly continuous from $\Cont^m(X) \times (-\tau,\tau)$ to $\R$ (see Lemma~\ref{lemma:pair_strong_weak}).

\begin{remark}[Vertical and horizontal perturbations] \label{remark:perturbations}
    For any compact set $X$, a balanced measure \(\nu \in \calM_0(X)\) induces a $\Cont^0$-perturbation of \(\mu\) in the form of \(\mu_t = \mu + t \nu\) if and only if \(\nu \ll \mu\) and the Radon--Nikodym derivative \(\frac{\mathrm{d} \nu}{\mathrm{d} \mu}\) is uniformly bounded from above and below. We call such perturbations ``vertical perturbations'' as they correspond to mass disappearing somewhere to be recreated elsewhere.
    The constraint \(\pair{\nu}{\ones_X} = 0\) in the space of balanced measures \(\calM_0(X)\) comes from the preservation of mass. 
    
    When $X$ is the closure of a bounded open set in \(\R^d\), we call ``horizontal perturbations'' curves of the form 
    \begin{equation*}      
        \mu_t = \Psi(t,\vdot)_\# \mu 
    \end{equation*}
    for a continuous function \(\Psi \colon (-\tau,\tau) \times X \to X\) with one continuous time derivative. It defines a $\Cont^1$-perturbation and 
    \(\mudot_t = - \ddiv(\mu_t v_t)\) with \(v_t = \frac{\partial \Psi}{\partial t}(t,\cdot)\). 
    Indeed we have, for any $\phi \in \Cont^1(X)$,  
    \begin{equation*}
        \bpair{\mudot_t}{\phi} = \dd{}{t} \langle \mu_t,\phi \rangle = \int_X \frac{\partial \Psi}{\partial t}(t,x) \cdot \nabla \phi(x)  \diff \mu_t(x) = \int_X \frac{\partial \Psi}{\partial t}(t,\Psi(t,x)) \cdot \nabla \phi(\Psi(t,x))  \diff \mu_t(x),
    \end{equation*}
    and the latter is jointly continuous in $(\phi,t)$ with respect to the norm topology of $\Cont^1(X)$.
\end{remark}

Before stating the main result of this section, we need to introduce integral operators related to the transport kernel. We define the self-transport kernel as follows: 
\begin{equation*}    
    k_{\mu} (x,y) \assign \exp \left( \frac{1}{\varepsilon}\left(\fmumu(x) + \fmumu(y) - c(x,y) \right) \right).    
\end{equation*}
Recall here that we adopt the convention \(\fmumu = \gmumu\) so that \(k_\mu\) is unambiguous.

Since we assume that \(c \in \Cont^{(m,m)}(X \times X)\), the smoothing effect of the Schrödinger system~\eqref{eq:Schroedinger_system} guarantees that the potentials \(\fmunu, \gmunu\) belong to \(\Cont^m(X)\) (see also Proposition~\ref{prop:reg_fmunu_cm} below). Thus \(k_{\mu}\) has the same regularity class as the cost function \(c\), and it is enough to guarantee that the following operators are well defined (see also Proposition~\ref{prop:Hmu_compact_CX} below).

\begin{definition} \label{def:kernel_operators}
    For \(\mu \in \prm(X)\) we define the kernel operators \(K_\mu\), \(H_\mu\) by 
    \begin{align*}
        &K_{\mu} \colon \Cont(X) \to \Cont^m(X), \quad K_{\mu} [\phi] (y) = \int_X \phi(x) k_\mu(x,y) \diff \mu(x),\\
        & H_{\mu} \colon \Cont^m(X)^* \to \Cont^m(X), \quad H_{\mu} [\nu] (y) = \bpair{\nu}{k_\mu(\cdot,y)}. 
    \end{align*}
    Analogously, we define \(H_{c} \colon \Cont^m(X)^* \to \Cont^m(X), ~ H_{c} [\nu] (y) = \bpair{\nu}{k_c(\cdot,y)}\).
\end{definition}

The operator \(K_\mu\) can be extended to the domain \(L^1(X,\mu)\). Note that $H_{\mu} [\nu] (y) = \int_X k_\mu(x,y) \diff \nu(x)$ when $\nu$ is a measure, thus if $\phi \in L^1(X,\mu)$, then $K_{\mu}[\phi] = H_\mu[\phi \mu]$. The operator \(H_\mu\) is the kernel mean embedding into the Reproducing Kernel Hilbert Space (RKHS) $\Hil_\mu$ induced by \(k_\mu\), and its domain can be extended to \( \Hil_\mu^* \). This will be discussed in Section~\ref{section:FA_regularity}.

Let us call $\Cont^m(X)^*_0$ the space of balanced distributions, that is, the space of all $\nu \in \Cont^m(X)^*$ satisfying $\langle \nu, \ones_X \rangle = 0$. It is dual to $\Cont^m(X)/\R$,  the space of $\Cont^m$ functions identified up to shifts by constant functions (identified with $\R$). Hence the pairing \(\pair{\nu}{\phi}\) is well defined between \(\nu \in \Cont^m(X)^*_0 \) and \(\phi \in \Cont^m(X)/\R\). Since we regard paths valued in \(\prm(X)\), all perturbations in Definition~\ref{def:Cm_perturbations} are of this form. The operator \(K_\mu\) descends to a well-defined operator 
\begin{equation*}
    K_\mu  \colon \Cont(X) / \R \to \Cont^m(X) / \R
\end{equation*}
because \(K_\mu[\ones_X]=\ones_X\) by \eqref{eq:kernel_slice_int_1}. Using the canonical projection $\Cont^m(X) \to \Cont^m(X) / \R$ we can always move to the quotient space.

\begin{theorem}[Hessian of the Sinkhorn divergence] \label{theorem:hessian_sinkhorn} 
    With Assumption~\ref{asp:diff_m} for some $m \geq 0$,
    let \(\mu \in \prm(X)\) and \((\mu_t)_t\) be a \(\Cont^m\)-perturbation of \(\mu\). 
    Then 
    \begin{equation} \label{eq:theorem:hessian_sinkhorn}
        \lim_{t \to 0} \frac{S_\varepsilon(\mu, \mu_t)}{t^2} = \frac{\varepsilon}{2} \bpair{ \mudot }{ (\id - K_\mu^2)^{-1} H_\mu \left[ \mudot \right] }.
    \end{equation}
\end{theorem}

The expression on the right-hand side is well defined, in the sense that \( (\id - K_\mu^2)\inv \) exists as an operator on $\Cont^m(X)/\R$ and thus the duality pairing is understood between this space and balanced measures or distributions. In Section~\ref{section:FA_regularity} we will take the formula on the right-hand side of~\eqref{eq:theorem:hessian_sinkhorn} as a starting point and analyze for which class of perturbations $\mudot$ we can make sense of it.

\begin{remark}[Discrete measures with finite support]
    When \( \mu = \sum_{i=1}^n a_i \delta_{x_i} \) is a discrete measure with $a_i > 0$ and the vertical perturbation reads \( \nu = \sum_{i=1}^n b_i \delta_{x_i} \) with $\sum_i b_i = 0$, the formula can be expressed as follows. With $f = \fmumu$, we introduce the $n \times n$ matrices $K,H$ as 
    \begin{equation*}
    K_{ij} = \exp \left( \frac{1}{\varepsilon} (f(x_i) + f(x_j) - c(x_i,x_j)) \right) a_j, \quad H_{ij} = \exp \left( \frac{1}{\varepsilon} (f(x_i) + f(x_j) - c(x_i,x_j)) \right)
    \end{equation*}
    for indices $i,j \in \{ 1, \ldots,n \}$. With $b = (b_i)_i \in \R^n$ Theorem~\ref{theorem:hessian_sinkhorn} reads 
    \begin{equation*}
    \lim_{t \to 0} \frac{S_\varepsilon(\mu, \mu_t)}{t^2} = \frac{\varepsilon}{2} b^\top (\id - K^2)^\dagger H b,
    \end{equation*}
    where $(\id - K^2)^\dagger$ denotes the Moore--Penrose pseudo-inverse of the matrix $\id - K^2$. 
\end{remark}

The proof of Theorem~\ref{theorem:hessian_sinkhorn} relies on the differentiability of the Schrödinger potentials along the path $(\mu_t)_t$ and on explicit formulas for the derivatives. 

For the technical arguments we follow the line of the works~\cite{Carlier2024, Carlier2020} which use the implicit function theorem to do so, though these articles are in the more involved context of a multi-marginal optimal transport problem and with different regularity assumptions on the path of measures $(\mu_t)_t$. 

We first prove preliminary results on the operators $K_\mu$ and $H_\mu$ before moving to the core of our strategy and arguments.

\subsection{Preliminary results on the operators \texorpdfstring{$K_\mu$ and $H_\mu$}{Hmu and Kmu}}

In this section we analyze the operators $K_\mu$ and $H_\mu$ which will appear naturally when differentiating the fixed point equation~\eqref{eq:Schroedinger_system} that characterizes the Schrödinger potentials. We then compute the derivatives of the Schrödinger potentials, and finally prove Theorem~\ref{theorem:hessian_sinkhorn} in Section~\ref{section:sub:proof_expansion}. 

We still fix $m \geq 0$ and work in the setting of Assumption~\ref{asp:diff_m}. We start by qualitative results on the operators $H_\mu, K_\mu$ and the Schrödinger potentials.

\begin{proposition}\label{prop:Hmu_compact_CX} 
    Fix \(\mu \in \prm(X)\). The operators \(H_\mu, H_c \colon \Cont^m(X)^* \to \Cont^m(X)\) are \textweakstar{}-to-norm continuous, as well as norm-to-norm continuous and compact. The operator \(K_\mu \colon \Cont(X) \to \Cont^m(X)\) is norm-to-norm continuous and compact.
\end{proposition} 

\begin{proof}
    The kernels $k_c$ and $k_\mu$ belong to $\Cont^{(m,m)}(X)$. The results for $H_c, H_\mu$ directly follow by standard arguments on integral kernel operators, see Lemma~\ref{lemma:H_k_Cm} in the appendix. As $K_\mu[\phi] = H_\mu[\phi \mu]$, continuity and compactness extend to $K_\mu$ as well.
\end{proof}

\begin{proposition} \label{prop:reg_fmunu_cm}
    The map $(\mu,\nu) \mapsto f_{\mu,\nu}$ from $\prm(X)^2$ to $\Cont^m(X)$ is continuous and the set $\setgiven{ f_{\mu,\nu} }{ \mu,\nu \in \prm(X) }$ is compact in $\Cont^m(X)$.    
\end{proposition}

\noindent By symmetry the same results holds for $(\mu,\nu) \mapsto g_{\mu,\nu}$.

\begin{proof}
    From continuity of $(\mu,\nu) \mapsto (f_{\mu,\nu}, g_{\mu,\nu})$ in $\Cont(X)^2$ \cite[Proposition B.1]{nutz2023stability}, we see that the map $(\mu,\nu) \to g_{\mu,\nu} \nu \in \calM(X)$ is \textweakstar{}-to-\textweakstar{} continuous. As $\calM(X) = \Cont^{0}(X)^*$, we can apply Lemma~\ref{lemma:H_k_Cm} and deduce that $(\mu,\nu) \mapsto H_c[ g_{\mu,\nu} \nu] = \exp(-f_{\mu,\nu}/\varepsilon)$ is continuous into $\Cont^m(X)$, and thus continuity as in our claim follows by taking the logarithm. The set $\setgiven{ f_{\mu,\nu} }{ \mu,\nu \in \prm(X) }$ is compact as a continuous image of a compact set.
\end{proof}

The following result is a key quantitative estimate which holds uniformly over $\prm(X)$.
We recall that \( \Cont(X) / \R \) is the quotient space of continuous functions by constant functions. We equip this space with the quotient norm
\begin{equation*}
\norm{\phi}_{\Cont/\R}\assign\inf_{\lambda\in\R}\norm{\phi-\lambda\cdot\ones_X}_\infty=\frac{1}{2}\left(\sup_{x\in X}\phi(x)-\inf_{y\in X}\phi(y)\right).
\end{equation*}

\begin{proposition}[Contraction property] \label{prop:Kmu_contraction}
    Let \(q \assign 1 - \exp(-4\norm{c}_\infty / \varepsilon) \in (0,1)\). 
    Then \(\norm{K_{\mu} [\phi]}_{\Cont/\R} \leq q \cdot \norm{\phi}_{\Cont/\R}\) for all \(\phi\in\Cont(X)\) and all \(\mu \in \prm(X)\).
\end{proposition}

\begin{proof}
    By~\cite[Thm.~1.2]{DiMarino2020} the Schrödinger potentials \(\fmumu \in \Cont(X)\) satisfy \(\norm{\fmumu}_\infty \leq \frac{3}{2} \norm{c}_\infty\). With $a = \exp(-4\norm{c}_\infty / \varepsilon)$ and using $c \geq 0$, we see that $k_\mu(y,z)  \geq a \cdot k_\mu(x,z)$ for any triplet $x,y,z \in X$. For any \(x,y\in X\), this, together with~\eqref{eq:kernel_slice_int_1}, leads to
    \begin{align*}
        &K_{\mu} [\phi](x) - K_{\mu} [\phi](y) \\
        &\qquad = \int_{\{\phi>0\}} \phi(z) \big(k_{\mu}(x,z) - k_{\mu}(y,z)\big) \diff \mu(z) + \int_{\{\phi<0\}} \! - \phi(z) \big( - k_{\mu}(x,z) + k_{\mu}(y,z)\big) \diff \mu(z) \\
        &\qquad \leq  \int_{\{\phi>0\}} \phi(z) \left( 1 - a \right) k_{\mu}(x,z) \diff \mu(z) + \int_{\{\phi<0\}} \! - \phi(z) \left( 1 - a \right) k_{\mu}(y,z) \diff \mu(z)  \\
        &\qquad \leq 2 \left( 1 - a \right) \cdot \norm{\phi}_\infty.
    \end{align*}
    The operator norm estimate with respect to the quotient norm \(\norm{.}_{\Cont / \R}\) on domain and codomain follows by taking a supremum over \(x\) and \(y\) and considering the infimum over shifts \(\phi-\lambda\) by constants \(\lambda\in\R\). Finally \(q = 1 - a\).
\end{proof}

From this estimate we prove that the inverse of the map \( \id - K_\mu^2 \), central in the expansion of the Sinkhorn divergence, actually exists. We use \(\BB(\mathbb X)\) to denote the set of bounded linear operators on a normed space \(\mathbb X\). 

\begin{theorem} \label{theorem:inverse_operators_continuous}
    Let \(\mu \in \prm(X)\). Then \((\id - K_\mu^2)\inv\) exists in \(\BB(\Cont^m(X) / \R)\) and \((\id + K_\mu)\inv\) exists in \(\BB(\Cont^m(X))\). 
\end{theorem} 

\noindent The result will be strengthened in Theorem~\ref{theorem:inverse_operators_Hmu} in the next Section. 

\begin{proof}
    \( \id - K_\mu^2 \) and \( \id + K_\mu \) are compact perturbations of the identity on, respectively, the Banach spaces $\Cont^m(X) / \R$ and $\Cont^m(X)$. To check that they are invertible, by the Fredholm alternative~\cite[Theorem 6.6]{Brezis2011} it is enough to verify that the equations $\phi = K_\mu^2[\phi]$ and $\phi = - K_\mu[\phi]$ only have the trivial solution $\phi = 0$. In both cases Proposition~\ref{prop:Kmu_contraction} directly yields that $\norm{\phi}_{\Cont / \R} = 0$. In the case $\phi = K_\mu^2[\phi]$, this gives $\phi = 0$ in $\Cont(X) / \R$. For the equation $\phi = - K_\mu[\phi]$, this yields that $\phi$ is constant, and from \(K_\mu[\ones_X] = \ones_X\) we obtain $\phi = 0$. 
\end{proof}

\begin{remark}[A probabilistic interpretation]
    Note that $K_\mu$ can be interpreted as follows. For $\mu \in \prm(X)$, we can consider a Markov chain on $X$ with equilibrium measure $\mu$, and whose transition kernel is given by $x \mapsto k_\mu(x,y) \, \mathrm{d} \mu(y)$. The operator $K_\mu$ is simply the dual Markov operator of this chain. Indeed, the entropic optimal transport coupling $\pi$ between $\mu$ and itself, given by~\eqref{eq:formula_coupling_optimal}, is the law of the pair of random variables $(x_1,x_2)$ following this Markov chain if the initial distribution is $\mu$.
    
    As $k_\mu$ is symmetric, we see that the Markov chain is reversible. The spectrum of the operator $K_\mu$ is studied in more detail in Proposition~\ref{prop:Hmu_compact_Hmu} which builds on Proposition~\ref{prop:Kmu_contraction}: all the eigenvalues belong to \([0,1]\), and the spectral gap of \(\id - K_\mu\) is at least $\exp(-4\norm{c}_\infty / \varepsilon)$.
\end{remark}

\subsection{Proof of the expansion} \label{section:sub:proof_expansion}

We are now ready for the proof of Theorem~\ref{theorem:hessian_sinkhorn}, again under Assumption~\ref{asp:diff_m}.
The starting point is this elementary lemma from calculus.

\begin{lemma} \label{lemma:C11_integral_expansion}
    Assume that $F$, defined in a neighborhood of $(0,0) \in \R^2$, is of class \(\Cont^{(1,1)}\). Then
    \begin{equation*}
        \lim_{t \to 0} \frac{F(0,t) + F(t,0) - F(0,0) - F(t,t)}{2} =  - \frac{1}{2} \frac{\partial^2 F}{\partial t \partial s}(0,0).
    \end{equation*}
\end{lemma} 

\begin{proof}
    By writing $F(0,t) + F(t,0) - F(0,0) - F(t,t) = F(t,0) - F(0,0) - (F(t,t) - F(t,0))$ and integrating first with respect to the second variable and then the first one
    \begin{equation*}
        F(0,t) + F(t,0) - F(0,0) - F(t,t) = -\int_0^t \int_0^t \partial_1 \partial_2 F(\tau,\sigma) \diff \sigma \diff \tau.
    \end{equation*}
    The conclusion follows by continuity of $\partial_1 \partial_2 F$.
\end{proof}

We will apply this lemma to \(F(t,s) \assign \OT_\varepsilon (\mu_t, \mu_s)\). Indeed as this function is symmetric we have 
\begin{equation*}
    S_\varepsilon(\mu,\mu_t) = \frac{F(0,t) + F(t,0) - F(0,0) - F(t,t)}{2}.
\end{equation*}
So provided we prove that $F$ is of class $\Cont^{(1,1)}$ around \((0,0)\), Lemma~\ref{lemma:C11_integral_expansion} yields
\begin{equation}
\label{eq:expansion_S_F}
    \lim_{t \to 0} \frac{S_\varepsilon(\mu, \mu_t)}{t^2} = - \frac{1}{2} \frac{\partial^2 F}{\partial t \partial s}(0,0).
\end{equation}
Thus to prove our theorem we need to prove $F \in \Cont^{(1,1)}$ and differentiate it twice. Taking one derivative of $F$ is standard \cite{Feydy2018}, we provide the following lemma adapted to our context.

\begin{lemma} \label{lemma:grad_OTeps}
    If the paths \((\mu_t)_{t \in (-\tau,\tau)}, (\nu_t)_{t \in (-\tau,\tau)}\) are $\Cont^m$-perturbations of \(\mu, \nu \in \prm(X)\),
    then 
    \begin{equation*} 
        \dd{}{t} \Big|_{t=0} \OT_\varepsilon(\mu_t, \nu_t) = \pair{\mudot}{\fmunu} + \pair{\nudot}{\gmunu}.
    \end{equation*}
\end{lemma}

\begin{proof}
    We follow the proof strategy of~\cite[Prop.~2]{Feydy2018}. As seen there (replacing \(\delta\mu\) by \((\mu_t - \mu)/t\)), taking \(f_{\mu_t,\nu_t}, g_{\mu_t,\nu_t}\) as a suboptimal pair in the dual formulation~\eqref{eq:entropic_dual} of \(\OT_\varepsilon(\mu,\nu)\) and likewise \(\fmunu, \gmunu\) for \(\OT_\varepsilon(\mu_t,\nu_t)\) yields 
    \begin{align*}
        \bpair{\frac{\mu_t - \mu}{t}}{\fmunu} + \bpair{\frac{\nu_t - \nu}{t}}{\gmunu} + \oh(1) 
        &\leq \frac{1}{t} \left( \OT_\varepsilon(\mu_t, \nu_t) - \OT_\varepsilon(\mu,\nu) \right) \\
        & \leq  \bpair{\frac{\mu_t - \mu}{t}}{f_{\mu_t,\nu_t}} + \bpair{\frac{\nu_t - \nu}{t}}{g_{\mu_t,\nu_t}} + \oh(1). 
    \end{align*}
    Taking the limit \(t \to 0\) gives the result. Indeed, as $t \to 0$, \((f_{\mu_t,\nu_t}, g_{\mu_t,\nu_t})\) converges to \((\fmunu, \gmunu)\) in \( (\Cont^m(X) / \R)^2\) (Proposition~\ref{prop:reg_fmunu_cm}), while $(\mu_t - \mu)/t$ and $(\nu_t - \nu) /t$ converge \textweakstar~to $\mudot$ and $\nudot$ in $\Cont^m(X)^*_0$. 
\end{proof}

From this lemma we obtain \(\frac{\partial F}{\partial t} (t,s) = \pair{\mudot_t}{f_{\mu_t,\mu_s}}\). To compute the second derivative of $F$, we will need to differentiate the Schrödinger potentials with respect to the input measures, as announced.
As we will use the fixed point equation \eqref{eq:Schroedinger_system} defining the optimal pair of potentials $(f_{\mu_t,\mu_s},g_{\mu_t,\mu_s})$, we first need to understand the derivatives of the operator $T_\varepsilon$ featured in it. Whereas $T_\varepsilon$ is a priori defined on $\Cont(X) \times \prm(X)$, from $T_\varepsilon(f+\lambda \ones_X, \mu) = T_\varepsilon(f, \mu) - \lambda \ones_X$ for any $f,\mu$ and $\lambda \in \R$, we see that $T_\varepsilon$ descends to an operator $\Cont(X) / \R \times \prm(X) \to \Cont^m(X) / \R$.

\begin{lemma} \label{lemma:derivatives_Teps}
    Let the path \((\mu_t)_{t \in (-\tau,\tau)}\) be a $\Cont^m$-perturbation.
    Then the parametrized map \(\Cont^m(X) \times (-\tau,\tau) \to \Cont^m(X),~(f,t) \mapsto T_\varepsilon(f,\mu_t)\) built from~\eqref{eq:def:T_eps} is continuously differentiable. For $\phi \in \Cont^m(X)$ it satisfies
    \begin{equation*}
        D_1 T_\varepsilon(\fmumu,\mu)[\phi] = - K_{\mu}[\phi], \qquad \dd{}{t}\Big|_{t=0} T_\varepsilon(\fmumu,\mu_t) = - \varepsilon H_\mu [ \mudot ].
    \end{equation*}
    These formulas remain true when \(\Cont^m(X)\) is replaced by \(\Cont^m(X) / \R\). 
\end{lemma}

\begin{proof}
    We first show that the mapping is continuously differentiable. It is enough to prove that $U(g,t)(y) = \int_X k_c(x,y) g(x) \diff \mu_t(x)$ is continuously differentiable as a map $\Cont^m(X) \times (-\tau,\tau) \to \Cont^m(X)$ as $T_\varepsilon(f,t) = - \varepsilon \log U(\exp(f/\varepsilon),t)$. The mapping $U$ is linear in $g$ and continuous (thus continuously differentiable) in $g$ by Lemma~\ref{lemma:H_k_Cm}.
    Regarding the variable $t$, for any given $g \in \Cont^m(X)$ and $y \in X$ and any \(\alpha \in \N_0^d\) with \(\abs{\alpha} \leq m\) we have $\partial^\alpha \dd{}{t} U(g,t)(y) = \bpair{\dot{\mu}_t}{\partial_y^\alpha k_c(\cdot ,y) g}$ as $\partial_y^\alpha k_c(\cdot ,y) g \in \Cont^m(X)$. Using the regularity of $\mudot_t$ given by Definition~\ref{def:Cm_perturbations}, this expression is continuous in $y,g$ and $t$. 
    If $g_n \to g$, Lemma~\ref{lemma:abstract_weak_unif} shows that the convergence of $\partial^\alpha \dd{}{t} U(g_n,t)(y)$ to $\partial^\alpha \dd{}{t} U(g,t)(y)$ is uniform in $y$. Thus we have differentiability of $U$ in its second variable with jointly continuous derivative. The differentiability of $T_\varepsilon$ follows.   
    
    Next we establish the expressions for the derivatives at $(f_{\mu,\mu},0)$. By the chain rule we have easily
    \begin{equation*} %\label{eq:Teps_phi_derivative}
        D_1 T_\varepsilon(f,\mu)[\phi](y) = - \varepsilon \frac{\int \frac{1}{\varepsilon} \phi(x) \exp \left( \frac{1}{\varepsilon} (f(x) - c(x,y)) \right) \diff \mu(x)} {\int  \exp\left( \frac{1}{\varepsilon} (f(x) - c(x,y)) \right) \diff \mu(x)}.
    \end{equation*}
    As \(T_\varepsilon(\fmumu,\mu) = \fmumu\), at $f = \fmumu$ the denominator is equal to \( \exp\left(-\frac{1}{\varepsilon}\fmumu(y)\right)\). Thus
    \begin{equation*}
        D_1 T_\varepsilon(\fmumu,\mu)[\phi](y) = - \int_X \phi(x) \exp\left(\frac{1}{\varepsilon}(\fmumu(y) + \fmumu(x) - c(x,y))\right) \diff \mu(x) = - K_{\mu}[\phi](y).
    \end{equation*}
    For the derivative in the second component a similar computation yields 
    \begin{equation*}
        \dd{}{t}\Big|_{t=0} T_\varepsilon (\fmumu,\mu_t) (y)  = - \varepsilon \frac{\pair{\mudot}{\exp(\frac{1}{\varepsilon} (f_{\mu,\mu} - c(\cdot,y)))}}{\pair{\mu}{\exp(\frac{1}{\varepsilon} (f_{\mu,\mu} - c(\cdot,y)))}}.
    \end{equation*}
    From the same argument as before the denominator is equal to \( \exp\left(-\frac{1}{\varepsilon}\fmumu(y)\right)\). 
    Thus the right hand side coincides with $- \varepsilon H_{\mu}[\mudot]$.
    Quotienting by constant functions to extend to the case where $f \in \Cont^m(X) / \R$ instead of $\Cont^m(X)$ is straightforward. 
\end{proof}

\begin{proposition} \label{prop:time_derivative_of_potentials}
    Let the path \((\mu_t)_{t \in (-\tau,\tau)}\) be a $\Cont^m$-perturbation of \(\mu \in \prm(X)\). 
    Then the Schrödinger potentials \(f_{t,s} = f_{\mu_t,\mu_s}\) are continuously differentiable in \(\Cont^m(X)/\R\) with respect to the time parameters in a neighborhood of \((t,s) = (0,0)\). Moreover, they satisfy
    \begin{align}
    \begin{split} \label{eq:time_derivatives_of_potentials}
        &\frac{\partial f_{0,s}}{\partial s} \Big|_{s=0} = - \varepsilon (\id - K_{\mu}^2)\inv H_\mu [ \mudot ] \in \Cont^m(X) /\R, \\ 
        &\frac{\partial f_{t,0}}{\partial t} \Big|_{t=0} = \varepsilon K_\mu (\id - K_{\mu}^2)\inv H_\mu [ \mudot ] \in \Cont^m(X) /\R, \\
        &\frac{\partial f_{t,t}}{\partial t} \Big|_{t=0} = - \varepsilon (\id + K_{\mu})\inv H_\mu [ \mudot ] \in \Cont^m
        (X).
    \end{split}
    \end{align}
\end{proposition}

\begin{proof} 
    We use that \(f_{t,s}\) and \(g_{t,s} = f_{s,t}\) are characterized by the fixed point equation~\eqref{eq:Schroedinger_system}. Specifically, we have
    \begin{equation} \label{eq:proof:time_derivatives_fixed_point}
    f_{t,s} = T_\varepsilon(g_{t,s}, \mu_s) = T_\varepsilon(T_\varepsilon(f_{t,s}, \mu_t), \mu_s) 
    \end{equation} 
    for all \(t\) and \(s\). We will use the implicit function theorem~\cite[Thm.~10.2.1]{Dieudonne1969} applied to the map 
    \begin{equation*} %\label{eq:proof:time_derivatives_carlier_map}
        \tilde{T} \colon \Cont^m(X) / \R \times (-\tau,\tau)^2 \to \Cont^m(X) / \R, \quad (f,t,s) \mapsto (f - T_\varepsilon ( T_\varepsilon ( f, \mu_t ), \mu_s ) ).
    \end{equation*} 
    The condition \(\tilde{T}(f,t,s) = 0\) is equivalent to \(f = f_{t,s}\) in $\Cont^m(X) / \R$ by \eqref{eq:proof:time_derivatives_fixed_point}. The assumptions of~\cite[Thm.~10.2.1]{Dieudonne1969} ask that the map $\tilde{T}$ is of class $C^1$ and that $D_1 \tilde{T}$ is invertible at the point \((t,s) = (0,0)\), \(f = f_{0,0}\). The former is a direct consequence of Lemma~\ref{lemma:derivatives_Teps}, while for the latter we have \( D_1 \tilde{T} (f_{0,0},0,0) = \id - K_\mu^2\) by Lemma~\ref{lemma:derivatives_Teps} and this is invertible on \(\Cont^m(X) / \R\) by  Theorem~\ref{theorem:inverse_operators_continuous}.
    
    We can thus apply~\cite[Thm.~10.2.1]{Dieudonne1969}, which ensures that the solution $f = f_{t,s}$ solving $\tilde{T}(f_t,t,s) = 0$ is continuously differentiable in \(\Cont^m(X) / \R\) with respect to the time parameters in a neighborhood of \((t,s) = (0,0)\), and the partial derivatives are given by 
    \begin{align*}
        \frac{\partial f_{t,0}}{\partial t} \Big|_{t=0} & = - ( D_1 \tilde{T} (f_{0,0},0))^{-1} \left( \frac{\partial \tilde{T}(f_{0,0},t,0)}{\partial t} \Big|_{t=0} \right), \\
         \frac{\partial f_{0,s}}{\partial s} \Big|_{s=0} & = - ( D_1 \tilde{T} (f_{0,0},0))^{-1} \left( \frac{\partial \tilde{T}(f_{0,0},0,s)}{\partial s} \Big|_{s=0} \right).
    \end{align*} 
    Following Lemma~\ref{lemma:derivatives_Teps} we obtain   
    \begin{equation*}
        \frac{\partial \tilde{T}(f_{0,0},t,0)}{\partial t} \Big|_{t=0} = - \varepsilon K_\mu  H_{\mu} [ \mudot ], \qquad  
         \frac{\partial \tilde{T}(f_{0,0},0,s)}{\partial s} \Big|_{s=0} = \varepsilon H_{\mu} [ \mudot ].
    \end{equation*}
    This yields the first two derivatives.
    
    To calculate the time derivative of \(f_{t,t}\) we may drop the quotient and use the map \(\tilde{T}'(f,t) \assign f - T_\varepsilon(f,\mu_t)\) defined on $\Cont^m(X) \times (-\tau, \tau)$ instead of \(\tilde{T}\). Here $D_1 \tilde{T}'(f_{0,0},0) = \id + K_\mu$ whose invertibility in the space $\Cont^m(X)$ can be found once again in Theorem~\ref{theorem:inverse_operators_continuous}. Analogous arguments to the previous case give the result.
\end{proof}

\begin{proof}[{\textbf{Proof of Theorem~\ref{theorem:hessian_sinkhorn}}}]
    Denote the Schrödinger potentials by \(f_{t,s} \assign f_{\mu_t,\mu_s}\). 
    Define the symmetric function \(F(t,s) \assign \OT_\varepsilon (\mu_t, \mu_s)\). 
    Recall that with these notations, once we prove that $F \in \Cont^{(1,1)}$, then~\eqref{eq:expansion_S_F} holds by Lemma~\ref{lemma:C11_integral_expansion}. By Lemma~\ref{lemma:grad_OTeps} we have:
    \begin{equation*}
        \frac{\partial F}{\partial t}(t,s) = \bpair{ \mudot_t }{ f_{t,s} }.
    \end{equation*} 
    The \textweakstar~continuity of \(\mudot_t\) and continuous differentiability of \(f_{t,s}\) by the time parameters (Proposition~\ref{prop:time_derivative_of_potentials}) ensure that  \(F\) is of class \(\Cont^{(1,1)}\) around \((0,0)\). Moreover, we see that 
    \begin{equation*}
        \frac{\partial^2 F}{\partial t \partial s}(0,0) = \left\langle \mudot, \frac{\partial f_{0,s}}{\partial s} \Big|_{s=0} \right\rangle.
    \end{equation*} 
    Note that on the right-hand side we have the derivative of the ``first'' Schrödinger potential when the ``second'' measure is changed.
    From Proposition~\ref{prop:time_derivative_of_potentials} we obtain 
    \begin{equation*}
        \frac{\partial f_{0,s}}{\partial s} \Big|_{s=0} = - \varepsilon (\id - K_\mu^2)\inv H_\mu \left[ \mudot \right] \in \Cont^m(X)/\R .
    \end{equation*}
    We conclude that
    \begin{equation*}
        \lim_{t \to 0} \frac{S_\varepsilon(\mu, \mu_t)}{t^2} = - \frac{1}{2} \frac{\partial^2 F}{\partial t \partial s}(0,0) = - \frac{1}{2} \bpair{ \mudot }{ \frac{\partial f_{0,s}}{\partial s} \Big|_{s=0} } = \frac{\varepsilon}{2} \bpair{ \mudot }{ (\id - K_\mu^2)^{-1} H_\mu [ \mudot ] }. \qedhere
    \end{equation*}
\end{proof}

\section{The metric tensor and its regularity}
\label{section:FA_regularity}

In this section, we analyze the expression obtained as the second-order approximation of the Sinkhorn divergence in Theorem~\ref{theorem:hessian_sinkhorn}. We provide a way to extend the expression to a larger class of perturbations. Recall that $\mathcal{M}_0(X)$ is the set of balanced signed measures $\nu \in \calM(X)$, such that $\pair{\nu}{\ones_X} = 0$. 

\begin{definition}[The metric tensor]
\label{def:metric_tensor_measure}
    For $\mu \in \prm(X)$ and $\dot{\mu}_1$, $\dot{\mu}_2$ in $\mathcal{M}_0(X)$, we define the metric tensor $\g_\mu(\dot{\mu}_1, \dot{\mu}_2)$ as 
    \begin{equation*}
        \g_\mu(\dot{\mu}_1, \dot{\mu}_2) = \frac{\varepsilon}{2} \bpair{\dot{\mu}_1}{(\id - K_\mu^2)\inv H_\mu[\dot{\mu}_2]}.
    \end{equation*}
\end{definition}

In the setting of Assumption~\ref{asp:diff_m} we saw that this expression can be extended to perturbations $\dot{\mu}_1$, $\dot{\mu}_2$ in $\Cont^m(X)^*_0$. Even without this assumption we will see that $\g_\mu$ defines an inner product on $\calM_0(X)$ whose completion with respect to $\g_\mu$ is canonically identified with (a subset of) the dual of a Reproducing Kernel Hilbert Space (RKHS) $\Hil_\mu$. 

The RKHS \(\Hil_\mu\), which can be thought of as a tangent space to the space of probability measures at \(\mu\), depends on $\mu$. In contrast, the spaces $\mathcal{M}_0(X)$, or even $\Cont^m(X)^*_0$ for which Theorem~\ref{theorem:hessian_sinkhorn} was established are independent of \(\mu\). 
We thus do a global change of variables to embed the set of probability distributions $\prm(X)$ into the unit sphere of another RKHS $\Hil_c$. This makes all tangent spaces subspaces of \(\Hil_c\). Further, the metric tensor $\tg_\mu$, expressed in these new coordinates, is equivalent to the norm on $\Hil_c$ (uniformly in \(\mu\)). This change of variables will be especially useful in Section~\ref{section:path_metric}, where we define our new geodesic distance, as it yields uniform coercivity of the metric tensor in these coordinates.

We first recall some results on RKHS before moving to the analysis of the completion of the space $\calM_0(X)$ endowed with $\g_\mu$, and then to our global change of coordinates.

\subsection{A primer on RKHS, and the operators \texorpdfstring{$K_\mu$ and $H_\mu$}{Hmu and Kmu} in the RKHS \texorpdfstring{$\Hil_\mu$}{Hmu}} \label{section:sub:RKHS_introduction}

Take $k : X \times X \to [0, \infty)$ a non-negative continuous positive definite and universal kernel on $X$ (see~\cite[Sec.~2]{Sriperumbudur2011}). In the sequel we will consider $k_c = \exp(-c/\varepsilon)$ or $k_\mu = \exp((\fmumu \oplus \fmumu - c) / \varepsilon)$.

We recall that there exists a unique Reproducing Kernel Hilbert Space (RKHS) induced by the kernel $k$~\cite{Berlinet2011,Paulsen2016}. It is a subset of $\Cont(X)$ defined as the completion of the linear span of \(\setgiven{k(x,\vdot)}{x \in X}\) in the norm induced by the inner product 
\begin{equation} \label{eq:def_inner_product_RKHS}
    \bpair{\sum_{i=1}^n a_i k(x_i,\vdot)}{\sum_{j=1}^m b_j k(y_j,\vdot)}_{\Hil_k} \assign \sum_{i=1}^n \sum_{j=1}^m a_i b_j k(x_i, y_j).
\end{equation}
Here \(n, m \in \N_0\), \(a_i, b_j \in \R\) and \(x_i, y_j \in X\). We write $\Hil_k$ for this Hilbert space, and use $\Hil_c$ for the kernel $k_c$ and $\Hil_\mu$ for the kernel $k_\mu$. By our assumptions on \(c\) both \(k_c\) and \(k_\mu\) have the required properties (recall that a kernel $k$ is positive definite and universal provided \(\pair{\vdot}{\vdot}_{\Hil_k}\) is positive definite and $\Hil_k$ is dense in $\Cont(X)$ in supremum norm). Moreover, when $k$ is bounded, the embedding of $\Hil_k$ into $\Cont(X)$ is continuous. The evaluation operator is continuous in $\Hil_k$: for any $\phi \in \Hil_k$ there holds $\phi(x) = \pair{\phi}{k(x,\vdot)}_{\Hil_k}$. More details can be found in Appendix~\ref{section:appendix_RKHS}.

If $\sigma \in \Hil_k^*$ is a bounded linear functional on $\Hil_k$, we denote by $H_k[\sigma] \in \Hil_k$ its representative given by the Riesz Theorem. By definition, for any $\phi \in \Hil_k$ it satisfies
\begin{equation} \label{eq:def_kme_RKHS}
    \pair{\sigma}{\phi} = \pair{H_k[\sigma]}{\phi}_{\Hil_k}.
\end{equation}    
Evaluating this formula on $\phi = k(\cdot,y)$ we see that the function $H_k[\sigma]$ is defined by 
\begin{equation} \label{eq:def_kme_pointwise_RKHS}
    H_k[\sigma](y) = \pair{\sigma}{k(\cdot,y)}.
\end{equation} 
In particular, for $k = k_\mu$ the notation is consistent: $H_\mu$, introduced in Definition~\ref{def:kernel_operators}, can be extended into an operator defined on $\Hil_\mu^*$ with values in $\Hil_\mu$. The space \(\Hil_k^*\) contains $\calM(X)$ as $\Hil_k \hookrightarrow \Cont(X)$. 

In the setting of Assumption~\ref{asp:diff_m} we have \(k_\mu \in \Cont^{(m,m)}(X \times X)\), hence \(\Hil_\mu \hookrightarrow \Cont^m(X)\) and \(\Cont^m(X)^* \hookrightarrow \Hil_\mu^*\) (see~\cite[Cor.~4]{SimonGabriel2018}). 

In addition, we see that the operator $K_\mu$, defined by $H_\mu[\phi \mu]$ for any $\phi \in \Cont(X)$, actually takes values in $\Hil_\mu$. We also recall the identity 
\begin{equation*}
    \| \nu \|^2_{\Hil_k^*} = \| H_k[\nu] \|^2_{\Hil_k} = \pair{\nu}{H_k[\nu]} = \iint_{X \times X} k(x,y) \diff \nu(x) \diff \nu(y),
\end{equation*} 
valid for all $\nu \in \calM(X)$, which is usually referred to as a Maximum Mean Discrepancy Distance on the space $\calM(X)$.

\begin{proposition}\label{prop:Hmu_compact_Hmu}
    With Assumption~\ref{asp:diff_m} for some \(m \geq 0\), the operators \(H_\mu \colon \Cont^m(X)^* \to \Hil_\mu\) and \(K_\mu \colon \Cont^m(X) \to \Hil_\mu\) are compact. Furthermore \(K_\mu \colon \Hil_\mu \to \Hil_\mu\) is self-adjoint. In particular, it has an eigendecomposition. The eigenvalue \(1\) has multiplicity \(1\) and all the other eigenvalues belong to $[0,q]$ with $q < 1$ as in Proposition~\ref{prop:Kmu_contraction}.
\end{proposition}

\begin{proof}
    Compactness of \(H_\mu\) follows from Lemma~\ref{lemma:h_k_Cm_compactness}. Compactness of \(K_\mu\) follows directly. 
    
    The self-adjointness of $K_\mu$ is straightforward, as for all \(\phi, \psi \in \Hil_\mu\) it holds
    \begin{equation} \label{eq:proof:Kmu_selfadjoint}
        \pair{K_\mu[\phi]}{\psi}_{\Hil_\mu} = \pair{H_\mu[\phi \mu]}{\psi}_{\Hil_\mu} = \pair{\phi}{\psi}_{L^2(X,\mu)} = \pair{\phi}{K_\mu[\psi]}_{\Hil_\mu}.
    \end{equation}
    As $\Hil_\mu$ injects continuously into $\Cont(X)$, the operator $K_\mu$ is also compact from $\Hil_\mu$ into $\Hil_\mu$. Thus the spectral theorem~\cite[Theorem 6.11]{Brezis2011} applies, giving the eigendecomposition.
    
    We now analyze the spectrum of $K_\mu$. The computation $\pair{K_\mu[\phi]}{\psi}_{\Hil_\mu} = \pair{\phi}{\psi}_{L^2(X,\mu)}$ from~\eqref{eq:proof:Kmu_selfadjoint} shows that $K_\mu$ is non-negative, so all its eigenvalues are non-negative. Now, if $\phi \in \Hil_\mu \setminus \{0\}$, $\lambda \in [0,\infty)$ with $K_\mu[\phi] = \lambda \phi$, we see that either $\norm{\phi}_{\Cont(X) / \R} = 0$ so that $\phi$ is constant and thus $\lambda = 1$, or \(\norm{K_\mu[\phi]}_{\Cont(X) / \R} = \abs{\lambda} \cdot \norm{\phi}_{\Cont(X) / \R}\) so that $\lambda \in [0,q]$ by Proposition~\ref{prop:Kmu_contraction}.  
\end{proof}

Based on these observations we now prove a strengthening of Theorem~\ref{theorem:inverse_operators_continuous}. As $\ones_X = H_\mu[\mu] \in \Hil_\mu$, we can consider $\Hil_\mu  / \R$  the space $\Hil_\mu$ quotiented by constant functions, equipped with the quotient norm \(\norm{\phi}_{\Hil_\mu / \R} = \inf_{\lambda \in \R} \norm{\phi - \lambda \ones_X}_{\Hil_\mu}\). The Banach space \(\Hil_\mu / \R\) is isometric to \(\{\ones_X\}^\perp \subseteq \Hil_\mu\) and thus canonically is a Hilbert space.

\begin{theorem} \label{theorem:inverse_operators_Hmu}
    Let $\mu \in \prm(X)$ and \(n \geq 1\). On the Hilbert space $\Hil_\mu / \R$ equipped with the quotient norm the operators $(\id - K_\mu^n)^{-1}$ are continuous and satisfy
    \begin{equation*}
        \id \leq (\id - K_\mu^n)^{-1} \leq \frac{1}{1-q^n} \id
    \end{equation*}   
    in the partial ordering of self-adjoint operators, with $q < 1$ as in Proposition~\ref{prop:Kmu_contraction}. The same conclusion holds in the space $L^2(X,\mu) / \R$. 
\end{theorem} 

\noindent These operators appear in the metric tensor \(\g_\mu\) for \(n=2\) and in \(\tg_\mu\) for \(n=1\).

\begin{proof}
    Note that $\Hil_\mu / \R$ is isometric to the subspace orthogonal to the eigenspace for the eigenvalue $1$. So in this space $K_\mu$ is self-adjoint and compact with its spectrum being contained in $[0,q]$. The estimate directly follows. Exactly the same arguments carry through to $L^2(X,\mu) / \R$.
\end{proof}

\begin{remark}[Comparison with the differentiable setting] \label{rmk:comp_Hmu_Cm}
    To make the link to the case of $\Cont^m$-perturbations in the setting of Assumption~\ref{asp:diff_m} from the previous section, we recall the following standard bound. 
    Not only $\Cont^m(X)^* \hookrightarrow \Hil_\mu^*$, but Lemma~\ref{lemma:injection_Hk_Cm_quantitative} shows that $\norm{\nu}_{\Hil^*_\mu} \leq \norm{\nu}_{\Cont^{m,*}} \cdot \norm{k_\mu}_{\Cont^{(m,m)}}^{{1}/{2}}$ for any \(\nu \in \Cont^m(X)^*\). Using Proposition~\ref{prop:reg_fmunu_cm}, we can make this bound uniform in \(\mu\). Thus there exists a constant $C$ depending only on $\varepsilon$, $c$ and $m$ such that for all $\mu \in \prm(X)$ and all $\nu \in \Cont^m(X)^*$, we have
    \begin{equation} \label{eq:comp_Hmu_Cm}
        \norm{\nu}_{\Hil^*_\mu} \leq C \cdot \norm{\nu}_{\Cont^{m,*}}
    \end{equation}
\end{remark}

\subsection{The metric tensor in the RKHS \texorpdfstring{$\Hil_\mu$}{Hmu}}

We again work under Assumption~\ref{asp:diff_m} for some \(m \geq 0\) throughout this section. Recall that perturbations integrate to zero against $\ones_X$ because of mass conservation, something that we also need to take into account here. For the sake of clarity we state without proof the following result, which would work for any Banach space of functions on $X$ containing $\ones_X$. 

\begin{lemma}
    The dual of $\Hil_\mu  / \R$ is $\Hil_{\mu,0}^*$ the space of bounded linear functionals $\sigma$ on $\Hil_\mu$ such that $\pair{\sigma}{\ones_X} = 0$.  
\end{lemma}

\noindent In particular $H_\mu$ can be restricted to an operator
\begin{equation*}
    H_\mu \colon \Hil_{\mu,0}^* \to \Hil_\mu / \R,
\end{equation*}
which still corresponds to mapping $\sigma \in \Hil_{\mu,0}^*$ to its representative given by the Riesz Theorem: $\pair{\sigma}{\phi} = \pair{H_\mu[\sigma]}{\phi}_{\Hil_\mu / \R}$ for any $\phi \in \Hil_\mu / \R$. From here we can rewrite the metric tensor as follows: If $\mudot_1, \mudot_2 \in \Cont^m(X)^*_0$, we see that $H_\mu[\mudot_1]$, $H_\mu[\mudot_2] \in \Hil_\mu / \R$. As $(\id - K_\mu^2)^{-1}$ is a bounded operator on this space, we have 
\begin{equation} \label{eq:metric_tensor_Hmu}
    \g_\mu(\dot{\mu}_1, \dot{\mu}_2) = \frac{\varepsilon}{2} \bpair{\dot{\mu}_1}{(\id - K_\mu^2)\inv H_\mu[\dot{\mu}_2]} = \frac{\varepsilon}{2}  \bpair{H_\mu[\dot{\mu}_1]}{(\id - K_\mu^2)\inv H_\mu[\dot{\mu}_2]}_{\Hil_\mu / \R}.
\end{equation}
This formula directly extends to and characterizes the completion of the space $\Cont^m(X)^*_0$ endowed with $\g_\mu$. 

\begin{theorem} \label{theorem:metric_tensor_Hmu}
    On $\Cont^m(X)^*_0$, the quadratic form $\g_\mu$ is positive definite. The completion of $\Cont^m(X)^*_0$ with respect to this quadratic form coincides with $\Hil_{\mu,0}^*$, the space of bounded linear functionals $\sigma$ on $\Hil_\mu$ such that $\pair{\sigma}{\ones_X} = 0$. Moreover for any $\mudot \in \Hil_{\mu,0}^*$ one has
    \begin{equation} \label{eq:control_metric_tensor_Hmu}
        \frac{\varepsilon}{2} \norm{\dot{\mu}}^2_{\Hil_{\mu,0}^*}  \leq \g_\mu(\dot{\mu}, \dot{\mu}) \leq \frac{\varepsilon}{2} \frac{1}{1-q^2} \norm{\dot{\mu}}^2_{\Hil_{\mu,0}^*},
    \end{equation}  
    with $q$ as in Proposition~\ref{prop:Kmu_contraction}.
\end{theorem}

\begin{proof}
    We start from the formula~\eqref{eq:metric_tensor_Hmu}, which is a rewriting of Definition~\ref{def:metric_tensor_measure} in the language of RKHS. Combining it with Theorem~\ref{theorem:inverse_operators_Hmu} we obtain
    \begin{equation*}
        \frac{\varepsilon}{2} \norm{H_\mu[\dot{\mu}]}^2_{\Hil_\mu / \R} \leq \g_\mu(\dot{\mu}, \dot{\mu}) \leq \frac{\varepsilon}{2} \frac{1}{1-q^2} \norm{H_\mu[\dot{\mu}]}^2_{\Hil_\mu / \R},
    \end{equation*}  
    at least for $\mudot \in \Cont^m(X)^*_0$. As $H_\mu$ is an isometry from $\Hil_{\mu,0}^*$ to the space $\Hil_\mu / \R$, we deduce that~\eqref{eq:control_metric_tensor_Hmu} is valid for all $\mudot \in \Cont^m(X)^*_0$. Thus $\g_\mu$ is positive definite.

    It remains to show that the completion of $\Cont^m(X)^*_0$ is $\Hil^*_{\mu,0}$. Given what we just proved, it is enough to prove that $\calM_0(X)$ is dense in $\Hil^*_{\mu,0}$ for the norm topology. We use the isometry of the latter space to \(\Hil_\mu / \R\) via $\Hil_\mu$. The set $H_\mu[\calM(X)]$ is dense in $\Hil_\mu$, as discrete measures already yield all linear combinations of \(\setgiven{k_\mu(x,\vdot)}{x \in X}\). To conclude, we observe that \(\calM(X) = \R \cdot \mu \oplus \calM_0(X)\) and \(H_\mu[\mu] = \ones_X\).
\end{proof}

\subsection{The metric tensor in the RKHS \texorpdfstring{$\Hil_c$}{Hc}}

Although Theorem~\ref{theorem:metric_tensor_Hmu} exhibits a very neat structure, it has the main drawback that the ambient space of the tangent space $\Hil_{\mu,0}^* \cong \Hil_\mu / \R$ depends on the point $\mu$. The space $\Hil_\mu$ is however in a simple bijection with the RKHS $\Hil_c$, built on the kernel $k_c = \exp(-c/\varepsilon)$, as the next proposition shows.
In the following we will abbreviate 
\begin{equation*}
    \ef_\mu \assign \exp(\fmumu / \varepsilon).
\end{equation*}

\begin{proposition} \label{prop:RKHS_isometry}
    The spaces \(\Hil_c\) and \(\Hil_\mu\) are isometric to one another under the map
    \begin{equation*}
        \iota_\mu \colon 
        \Hil_c \to \Hil_\mu, \quad \phi \mapsto \phi \cdot \exp \left( \frac{1}{\varepsilon} \fmumu \right) = \phi \cdot \ef_\mu.
    \end{equation*}
\end{proposition}

\begin{proof}
    We can directly verify \(\exp(\fmumu(y)/\varepsilon) k_c(x,y) = \exp(-\fmumu(x)/\varepsilon) k_\mu(x,y)\) and hence
    \begin{align*}
        \bnorm{\exp \left( \frac{1}{\varepsilon} \fmumu \right) \sum_{i=1}^n b_i k_c(x_i,\vdot)}_{\Hil_\mu}^2 
        &= \sum_{i,j=1}^n b_i b_j \exp \left( - \frac{1}{\varepsilon} \fmumu(x_i) \right) \exp \left( - \frac{1}{\varepsilon} \fmumu(x_j) \right) k_\mu(x_i,x_j) \\
        &= \sum_{i,j=1}^n b_i b_j k_c(x_i,x_j) = \bnorm{\sum_{i=1}^n b_i k_c(x_i,\vdot)}_{\Hil_c}^2. \tag*{\qedhere}
    \end{align*}
\end{proof}

We consider the maps \(\A \colon \prm(X) \to \calM(X)\) and \(\B = H_c \circ \A \colon \prm(X) \to \Hil_c\) given by
\begin{align*}
\begin{split}
    &\A(\mu) = \exp \left( \frac{1}{\varepsilon} \fmumu \right) \mu, \\ 
    &\B(\mu)  = H_c[\A(\mu)] = \int_X \exp \left( \frac{1}{\varepsilon} (\fmumu(y) - c(\vdot,y)) \right) \diff \mu(y) = \exp \left( - \frac{1}{\varepsilon} \fmumu \right).
\end{split}
\end{align*}
Recall that the RKHS embeddings $H_c \colon \Hil_c^* \to \Hil_c$ and \(H_\mu \colon \Hil_\mu^* \to \Hil_\mu\) are characterized by~\eqref{eq:def_kme_RKHS} and~\eqref{eq:def_kme_pointwise_RKHS}. 
From a path \((\mu_t)_t \in \prm(X)\) one obtains the paths \((H_\mu[\mu_t])_t \in \Hil_\mu\), \((\alpha_t)_t \in \calM_+(X)\), and \((\beta_t)_t \in \Hil_c\) with \(\alpha_t = \A(\mu_t)\), \(\beta_t = \B(\mu_t) = H_c[\A(\mu_t)]\). 
The change of variables \(\mu \leftrightarrow \alpha\) was already used by Feydy et al.~in the proof of~\cite[Prop.~3]{Feydy2018}. 
Note that, for a fix measure $\nu$, the space $\prm(X)$ embeds into $\Hil_\nu$ via $H_{\nu}$, and that the latter can be mapped into $\Hil_c$ with $\iota_\nu\inv$. However, this embedding \(\mu \mapsto \iota_\nu\inv(H_\nu[\mu]) = H_c[\ef_\nu \cdot \mu]\) of $\prm(X)$ into $\Hil_c$ differs from \(B : \mu \mapsto H_c[\ef_\mu \cdot \mu]\).

The measures \(\alpha_t\) are naturally interpreted as linear functionals on the common space \(\Hil_c\) (represented by \(\beta_t\)). One can hope that differentiability on this fixed space is well behaved.
We will see that the metric tensor can be expressed in the $\beta$ variable and that it is equivalent to the squared norm on the space $\Hil_c$ in Theorem~\ref{theorem:metric_tensor_in_betadot}.

We first prove that \(\mu \leftrightarrow \alpha\) and \(\mu \leftrightarrow \beta\) are valid changes of variables in the sense that $A$ and $B$ are homeomorphisms onto their images. Recall that \(\proj_1\) stands for the projection map \((x,y)\mapsto x\). 

\begin{theorem} \label{theorem:alpha_beta_homeo}
    The maps \(\A\) and \(\B\) are homemorphisms onto their images. Explicitly:
    \begin{enumerate}
        \item \(\A\) is a \textweakstar~homeomorphism onto \(\calA \assign \setgiven{\alpha \in \calM_+(X)}{\norm{H_c[\alpha]}_{\Hil_c} = 1}\) with inverse \(\A\inv(\alpha) = (\proj_1)_\# (k_c \cdot \alpha \otimes \alpha)\). In particular, \(\calA\) is \textweakstar-compact.
        \item \(\B\) is a \textweakstar-to-weak homeomorphism onto
        $\calB~\assign~\setgiven{\beta~\in~H_c[\calM_+(X)]}{\norm{\beta}_{\Hil_c} = 1}$. 
        Weak convergence and norm convergence agree on \(\calB\), and the set $\calB$ is weakly and norm compact in $\Hil_c$.
    \end{enumerate}
\end{theorem}

Note that $\calB$ is the intersection of the convex cone \(H_c[\calM_+(X)]\) with the unit sphere. The proof for the weak compactness of \(\calB \subset \Hil_c\) relies heavily on the positivity constraint of the cone \(H_c[\calM_+(X)]\). 
In contrast, the weak closure of the whole unit sphere in infinite dimension is the unit ball, which is not norm compact.

\begin{remark}
    As a sidenote which might be interesting on its own, the characterization of \(\calB\) also gives non-negativity of the entropic self-transport potential of any measure with respect to any non-negative cost function inducing a positive definite universal kernel. By Theorem~\ref{theorem:alpha_beta_homeo}, we have \(\norm{\exp(-\fmumu/\varepsilon)}_{\Hil_c} = 1\) and hence the estimate~\eqref{eq:norm_infty_bound_RKHS} implies that \( \norm{\exp(-\fmumu/\varepsilon)}_\infty \leq \norm{\exp(-\fmumu/\varepsilon)}_{\Hil_c} \cdot \sup_{x \in X} k_c(x,x) \leq 1\). In particular, 
    \begin{equation*} 
       \fmumu(x) \geq 0 \quad \text{for all } x \in X. 
    \end{equation*}   
    With this we can slightly strengthen the contraction bound in Proposition~\ref{prop:Kmu_contraction} to \(q = 1 - \exp(-\frac{5}{2}\norm{c}_\infty / \varepsilon)\).
\end{remark}

\begin{proof}[\textbf{Proof of Theorem~\ref{theorem:alpha_beta_homeo}}]
    \underline{The map $A$ and its image:}
    Denote \(\alpha = A(\mu) = \ef_\mu \mu\). Note that the map \(\A\) is injective since \(\mu\) can be reconstructed as \(\mu = F(\alpha) \assign (\proj_1)_\# (k_c \cdot \alpha \otimes \alpha)\) thanks to~\eqref{eq:formula_coupling_optimal}. This also implies that the inverse map is \textweakstar-continuous: let \(\alpha_n \to \alpha\) \textweakstar~with corresponding \(\mu_n = (\proj_1)_\# (k_c \cdot \alpha_n \otimes \alpha_n)\). Then $\alpha_n \otimes \alpha_n$ converges \textweakstar~to $\alpha \otimes \alpha$ (see e.g.~\cite[Lem.~7.3]{Santambrogio2015}). Weak convergence of $\mu_n$ to $\mu$ easily follows. To prove that the forward map is continuous, it suffices to use the continuity of \(\fmumu\) as a function of \(\mu\) in supremum norm (see Proposition~\ref{prop:reg_fmunu_cm}).

    We now characterize the image of \(A\): Clearly \(\A(\mu) \geq 0\) and 
    \begin{equation*}
        \norm{H_c[\A(\mu)]}_{\Hil_c}^2 = \pair{\ef_\mu \mu \otimes \ef_\mu \mu}{k_c} = \pair{\mu \otimes \mu}{k_\mu} = 1
    \end{equation*}
    for all \(\mu \in \prm(X)\). The map \(F \colon \mathcal{A} \to \prm(X)\) is a left inverse of \(\A \colon \prm(X) \to \mathcal{A}\). To show that \(F\) is also a right inverse, we may equivalently show that \(F\) is injective. We have \(F(\alpha) = \mu\) if and only if \(k_c \cdot \alpha \otimes \alpha \in \Pi(\mu,\mu)\). For every such \(\alpha \in \mathcal{A}\) it must hold that \(\alpha \ll \mu\). Thus \(k_c \cdot \alpha \otimes \alpha = \exp(\phi \oplus \phi) k_c \cdot \mu \otimes \mu\) for \(\phi = \log  \frac{\mathrm{d} \alpha}{\mathrm{d} \mu}  \colon X \to [-\infty,\infty)\). A transport plan of this form is necessarily optimal for the entropic optimal transport problem~\eqref{eq:intro_def_OT} (see \cite[Thm.~2.1]{Nutz:IntroEOT}). Hence \(\phi = \fmumu / \varepsilon\) and \(\alpha = \A(\mu)\). Finally, \(\calA\) is \textweakstar-compact as the \textweakstar-continuous image of the \textweakstar-compact set \(\prm(X)\). 
    
    \underline{The map $B$ and its image:}
    The map \(\A\) is injective. The map $H_c$ is a bijection between $\Hil_c^*$ and $\Hil_c$ by definition, so it is injective. Thus $B$ is injective. It is clear that $\calB$ is the image of $\calA$ by $H_c$. Thus $\calB$ is the image of $\prm(X)$ by $B$.
    
    As the injection of $\Hil_c$ into $\Cont(X)$ is continuous by~\eqref{eq:norm_infty_bound_RKHS}, the map $H_c$ is continuous from $\calM(X)$ endowed with the \textweakstar~topology into $\Hil_c$ endowed with the weak topology. Thus \(\B\) is \textweakstar-to-weak continuous and $\calB$ is weakly compact as the image of the \textweakstar-compact set $\calA$. As a continuous bijection between compact Hausdorff spaces, \(\B\) is a \textweakstar-to-weak homeomorphism.
    
    Finally, if $(\beta_n)_{n}$ is any sequence in \(\calB\) which converges weakly to a limit $\beta$ in $\Hil_c$, as $\calB$ is weakly compact we must have $\beta \in \calB$ so $\| \beta \|_{\Hil_c} = 1 = \lim_{n \to \infty} \| \beta_n \|_{\Hil_c}$. This implies norm convergence of the sequence $(\beta_n)_n$.
\end{proof}

After the investigation of the topological properties of these changes of variables, we move on to the differential properties. For a path \((\mu_t)_t\) in \(\prm(X)\), we denote \(\alpha_t = \A(\mu_t)\), \(\beta_t = \B(\mu_t)\). By definition this means \(\pair{\alpha_t}{\phi} = \pair{\beta_t}{\phi}_{\Hil_c}\) for all \(\phi \in \Hil_c\). We will express the metric tensor in the variable \(\betadot_t \assign \dd{}{t} \beta_t = \dd{}{t} \exp(-f_{\mu_t,\mu_t}/\varepsilon) \in \Hil_c\). Let us first define the push-forward of the metric tensor before making a rigorous link to Definition~\ref{def:metric_tensor_measure} in Theorem~\ref{theorem:metric_tensor_in_betadot}.  

\begin{definition}[{The metric tensor on $\Hil_c$}]
    If $\mu \in \prm(X)$ with $\beta = B(\mu) \in \calB$, and $\betadot \in \Hil_c$ with $\pair{\beta}{\betadot}_{\Hil_c} = 0$, we define the quadratic form $\tg_\mu(\betadot, \betadot)$ by 
    \begin{equation*}
        \tg_\mu(\betadot, \betadot) = \frac{\varepsilon}{2} \left( \norm{\betadot}_{\Hil_c}^2 + 2 \bpair{\ef_{\mu} \cdot \betadot }{ (\id - K_{\mu})\inv [\ef_{\mu} \cdot \betadot] }_{L^2(X,\mu)} \right). 
    \end{equation*}
\end{definition}

To explicitly link $\g$ and $\tg$, we need to consider curves valued in $\prm(X)$ and their corresponding curves in $\calB$. The following definition fixes the convention about what we mean by a \emph{differentiable path}.

\begin{definition}[Differentiability of paths] \label{def:weakdiff} 
    Let \(\mathbb{X}\) be a vector space with a topology \(\tau\). We say that a path \((x_s)_{s \in I}\) in \(\mathbb{X}\) is \(\tau\)-differentiable at \(t\) when \(t\) is an accumulation point of \(I\) and the difference quotients \((x_s - x_t) / (s - t)\) converge in \(\tau\) to some \(\dot{x}_t\) as \(s \to t\). 
    When \(\mathbb{X}\) is a normed vector space and \(\tau\) is the norm topology, we drop the \(\tau\) in the notation.
\end{definition}

We define \(\mudot_t\) by \(\pair{\mudot_t}{\phi} = \dd{}{s}\big|_{s=t} \pair{\mu_s}{\phi}\) for \(\phi \in \Hil_{\mu_t}\) whenever this is a bounded linear functional on \(\Hil_{\mu_t}\), i.e.~when the path \((\mu_s)_s\), interpreted in the dual space \(\Hil_{\mu_t}^*\), is weakly differentiable in \(\Hil_{\mu_t}^*\) at \(t\). Note that automatically $\mudot_t \in \Hil_{\mu_t,0}^*$.

In the sequel, for paths $(\mu_t)_t$, $(\alpha_t)_t$, $(\beta_t)_t$ we write $\mu = \mu_0$, $\mudot = \mudot_0$ and similarly $\alpha = \alpha_0$, $\alphadot = \alphadot_0$ and $\beta = \beta_0$, $\betadot = \betadot_0$.

\begin{theorem} \label{theorem:metric_tensor_in_betadot}
    Let \((\mu_t)_t\) be a path in \(\prm(X)\), such that the associated path \((\beta_t)_t = (B(\mu_t))_t\) is weakly differentiable in \(\Hil_c\) at \(t=0\). Then $\pair{\betadot}{\beta}_{\Hil_c} = 0$, the path \((\mu_t)_t\) is weakly differentiable in \(\Hil_\mu^*\) at \(t=0\), and 
    \begin{equation} \label{eq:tgmu_expression}
        \g_{\mu}(\mudot,\mudot) = \tg_{\mu}(\betadot, \betadot). 
    \end{equation} 
\end{theorem}

To prove this theorem we first prove a preliminary lemma relating $\betadot$ and $\mudot$. Note that since \(\beta_t = \exp(-f_{\mu_t,\mu_t}/\varepsilon)\), we have \(\betadot_t = - \frac{1}{\varepsilon} \beta_t \cdot \dd{}{t} f_{\mu_t,\mu_t}\) and \eqref{eq:lemma:mudot_by_betadot} below is formally equivalent to the last equation of Proposition~\ref{prop:time_derivative_of_potentials}. The differences are the regularity assumptions and the functional spaces involved.

\begin{lemma} \label{lemma:mudot_by_betadot_in_RKHS}
    Let \((\mu_t)_{t}\) be a \textweakstar~continuous path in \(\prm(X)\) and denote \(f_{t,t} = f_{\mu_t,\mu_t}\). Then the following are equivalent:
    \begin{enumerate}
        \item The corresponding path \((\beta_t)_t\) is weakly (resp.~norm-) differentiable in \(\Hil_c\) at \(t=0\),
        \item \((\mu_t)_t\) is weakly (resp.~norm-) differentiable in \(\Hil_\mu^*\) at \(t=0\) and the difference quotients \(({f_{t,t} - f_{0,0}})/t\) are bounded in \(\Cont(X)\) as \(t \to 0\).
    \end{enumerate}
    In this case  
    \begin{equation} \label{eq:lemma:mudot_by_betadot} 
        H_{\mu}[\mudot] = (\id + K_{\mu})[\ef_{\mu}\cdot \betadot].
    \end{equation}
    Explicitly, this means that for all \(\phi \in \Hil_{\mu}\) it holds
    \begin{equation*}
        \dd{}{t}\Big|_{t=0} \bpair{\mu_t}{\phi} = \bpair{(\id + K_{\mu})[\ef_{\mu} \betadot]}{\phi}_{\Hil_\mu} = \bpair{\betadot}{ \beta  \cdot (\id + K_{\mu})[\phi]}_{\Hil_c}. 
    \end{equation*}
\end{lemma}

\noindent The condition of boundedness of the difference quotients \(({f_{t,t} - f_{0,0}})/t\) in the second point is not so natural. However, it is not clear how to remove it because we would need to linearize the Schrödinger system around $(\mu,\mu)$ in the space $\Hil_\mu$ (which depends on $\mu$), and we did not find how to apply the implicit function theorem in this case.

\begin{proof}
    \underline{1 $\Rightarrow$ 2:} 
    Recall that we have $\mu_t = \alpha_t \exp(- f_{\mu_t,\mu_t} / \varepsilon)$ and $\beta_t = \ef_{\mu_t}\inv = \exp(- f_{\mu_t,\mu_t} / \varepsilon)$, thus $\mu_t = \alpha_t \beta_t$. We will apply Leibniz's rule to this formula, but we rederive it given the low regularity of the objects involved.
    We write 
    \begin{align} \label{eq:proof:mudot_by_betadot:mu_difference_quotient}
        H_{\mu} \left[ \frac{\mu_t - \mu}{t} \right]
        &= H_{\mu} \left[ \alpha_t \cdot \frac{\beta_t - \beta}{t} \right] + H_{\mu}\left[ \frac{\alpha_t - \alpha}{t} \cdot \beta \right].
    \end{align}
    In the first term, \(\frac{\beta_t - \beta}{t} \to \betadot\) weakly in \(\Hil_c\), hence uniformly, while \(\alpha_t \to \alpha\) \textweakstar~in \(\calM(X)\). In particular, \(\alpha_t \cdot \frac{\beta_t - \beta}{t} \to \alpha \cdot \betadot\) \textweakstar~in \(\calM(X)\) and its RKHS embedding converges in norm by Lemma~\ref{lemma:h_k_Cm_compactness}. The limit is \(H_{\mu}[\alpha \cdot \betadot] = K_{\mu}[\betadot/\beta]\).
    
    With \(H_{\mu}[\nu] = 1/\beta \cdot H_c[\nu / \beta]\) for all \(\nu \in \calM(X)\) the second term rewrites as \(H_c[\frac{\alpha_t - \alpha}{t}] / \beta = \frac{\beta_t - \beta}{t} / \beta\), which converges weakly (resp.~in norm) to \(\betadot / \beta\). The self-transport potentials are even differentiable in time in \(\Cont(X)\), which is immediate from weak differentiability of \((\beta_t)_t\) in \(\Hil_c\) and Lemma~\ref{lemma:h_k_Cm_compactness}.
    
    \underline{2 $\Rightarrow$ 1:} 
    We use \(\beta_t = H_c[\mu_t / \beta_t]\) and \(H_{\mu}[\nu] = 1/\beta \cdot H_c[\nu / \beta]\) for \(\nu \in \calM(X)\) to calculate
    \begin{align}
    \begin{split} \label{eq:proof:mudot_by_betadot:difference_quotients}
        \frac{\beta_t - \beta}{t} 
        &= H_c \left[ \frac{(\beta_t\inv - \beta\inv) \mu_t + \beta\inv (\mu_t - \mu)}{t} \right] \\
        &= H_c \left[ \frac{\beta_t\inv - \beta\inv}{t} \cdot \mu_t \right] + \beta \cdot H_{\mu} \left[ \frac{\mu_t - \mu}{t} \right].
    \end{split}
    \end{align}
    We now show the desired convergence of these difference quotients in \(\Hil_c\) as \(t \to 0\). For the second summand this is due to Proposition~\ref{prop:RKHS_isometry} and the weak (resp.~norm-) differentiability of \((\mu_t)_t\). The first summand is trickier. Boundedness of the difference quotients and \textweakstar~convergence of \(\mu_t\) to \(\mu\) give \(\norm{\vdot}_{\Cont(X)^*}\)-boundedness of the measure in the argument of \(H_c\). The embedding \(H_c \colon \calM(X) \to \Hil_c\) is compact by Lemma~\ref{lemma:h_k_Cm_compactness}. Thus a convergent subsequence at times \((t_k)_{k\in\N}\) with \(\lim_{k \to \infty} t_k = 0\) can be extracted along which the first summand converges in norm in \(\Hil_c\) to a limit. This yields weak (resp.~norm-) convergence of the left-hand side in \(\Hil_c\). In particular, we obtain uniform convergence of the difference quotients \((\beta_{t_k} - \beta)/t_k \to v\) for some \(v \in \Hil_c\) and \((\beta_{t_k}\inv - \beta\inv)/t_k \to - v / \beta^2\) in \(\Cont(X)\). Taking the limit \(k \to \infty\) in~\eqref{eq:proof:mudot_by_betadot:difference_quotients} yields
    \begin{equation} \label{eq:proof:mudot_by_betadot:betadot_formula}
        v =  - H_c[v/\beta^2 \cdot \mu] + \beta \cdot H_\mu[\mudot] = \beta \cdot (H_\mu[\mudot] - K_\mu[v/\beta]),
    \end{equation}
    or equivalently \(v/\beta = (\id + K_\mu)\inv H_\mu[\mudot]\). As all subsequences of \((\beta_t - \beta)/t\) have convergent subsequences with the same limit \(v\), we have overall convergence as \(t \to 0\) and \(v = \betadot\). With self-adjointness of \(K_\mu\) on \(\Hil_\mu\) (Proposition~\ref{prop:Hmu_compact_Hmu}) and \(\ef_\mu = 1 / \beta\), the explicit formula follows from~\eqref{eq:proof:mudot_by_betadot:betadot_formula}.
\end{proof}

\begin{proof}[\textbf{Proof of Theorem~\ref{theorem:metric_tensor_in_betadot}}]
    The equation $\pair{\betadot}{\beta}_{\Hil_c} = 0$ is directly obtained by differentiating the equation $\| \beta_t \|_{\Hil_c}^2 = 1$ in time.
    
    We then plug the formula for \(\mudot\) from Lemma~\ref{lemma:mudot_by_betadot_in_RKHS} into the metric tensor 
    \begin{equation*}
        \g_\mu(\mudot,\mudot) = \frac{\varepsilon}{2} \bpair{\mudot}{(\id - K_{\mu}^2)\inv H_{\mu}[\mudot]}
        = \frac{\varepsilon}{2} \bpair{(\id + K_{\mu})[ \betadot / \beta ]}{(\id -  K_{\mu}^2)\inv (\id + K_{\mu})[ \betadot / \beta ]}_{\Hil_{\mu}}.
    \end{equation*}
    The operator \((\id +  K_{\mu})\) is self-adjoint on \(\Hil_{\mu}\). Furthermore, \((\id +  K_{\mu})(\id -  K_{\mu}^2)\inv(\id +  K_{\mu}) \! = (\id +  K_{\mu})(\id -  K_{\mu})\inv = \id + 2  K_{\mu} (\id -  K_{\mu})\inv\). To conclude, use \(\pair{\phi}{K_{\mu}[\psi]}_{\Hil_{\mu}} = \pair{\phi}{\psi}_{L^2(X,\mu)}\) and Proposition~\ref{prop:RKHS_isometry}. 
\end{proof}

We finish this section by studying the continuity properties of the metric tensor $\tg$, and by proving that it can be uniformly compared to the norm of $\Hil_c$.

\begin{proposition} \label{prop:tgmu_joint_cont}
    The metric tensor \(\tg_\mu(\betadot,\betadot)\) is jointly continuous (resp.~lower semicontinuous) with respect to \textweakstar~convergence of \(\mu\) in \(\prm(X)\) and norm convergence (resp.~weak convergence) of \(\betadot\) in \(\Hil_c\).
\end{proposition}

\begin{proof}
    The first part \(\frac{\varepsilon}{2} \norm{\betadot}_{\Hil_c}^2\) of the metric tensor \(\tg_\mu\) does not depend on $\mu$ and is of course continuous (resp.~lower semicontinuous) with respect to norm convergence (resp.~weak convergence) of \(\betadot\) in \(\Hil_c\). 
    
    For the second part, we actually claim that the map
    \begin{equation*}
        \setgiven{(\mu,\phi) \in \prm(X) \!\times\! \Hil_c}{\pair{\B(\mu)}{\phi}_{\Hil_c} = 0}  \to \R, \quad (\mu,\phi) \mapsto \bpair{\ef_{\mu} \cdot \phi }{ (\id - K_{\mu})\inv [\ef_{\mu} \cdot \phi] }_{L^2(X,\mu)}    
    \end{equation*}
    is jointly continuous with respect to \textweakstar~convergence on \(\prm(X)\) and weak convergence on \(\Hil_c\). Indeed, let \(\mu_n \to \mu\) \textweakstar~in \(\prm(X)\), 
    \(\phi_n \to \phi\) weakly in \(\Hil_c\), and \(\pair{\B(\mu_n)}{\phi_n} = 0\). In particular, we have uniform convergence \(f_n = f_{\mu_n,\mu_n} \to \fmumu\), hence \(\ef_{\mu_n} \to \ef_\mu\), and by Lemma~\ref{lemma:h_k_Cm_compactness} also \(\phi_n \to \phi\) uniformly. Let us write 
    \begin{equation*}
        u_n = (\id - K_{\mu_n})\inv [ \ef_{\mu_n} \cdot \phi_n]  \quad \Leftrightarrow \quad u_n = \ef_{\mu_n} \cdot \phi_n + K_{\mu_n} u_n. 
    \end{equation*}
    We show that the sequence \(u_n\) converges to \(u = (\id - K_{\mu})\inv [ \ef_{\mu} \cdot \phi]\) in \(\Cont(X) / \R\). Indeed, 
    the uniform convergence of \(\ef_{\mu_n} \cdot \phi_n\) and the operator norm estimate in Proposition~\ref{prop:Kmu_contraction} show that the functions \(u_n\) are uniformly bounded in \(\Cont(X) / \R\). 
    
    By Proposition~\ref{prop:reg_fmunu_cm} and the Arzelà--Ascoli Theorem the potentials \(f_{\mu_n,\mu_n}\) are equi-bounded and equi-continuous.

    Thus it is clear that the functions \(K_{\mu_n} u_n\) are equi-continuous and equi-bounded. Hence, up to extraction of a subsequence, \(K_{\mu_n} u_n\) converges uniformly to a limit. Now also \(u_n = \ef_{\mu_n} \cdot \phi_n + K_{\mu_n} u_n\) converges uniformly to a limit \(u\). Once we have this information, we can identify this limit with \((\id - K_{\mu})\inv [ \ef_{\mu} \cdot \phi]\), because 
    \begin{equation*}
        u_n = \ef_{\mu_n} \cdot \phi_n + K_{\mu_n} u_n \to \ef_\mu \cdot \phi + K_{\mu} u = u.
    \end{equation*}
    Once we have this convergence, 
    \begin{equation*}
        \bpair{\ef_{\mu_n} \cdot \phi_n }{ (\id - K_{\mu_n})\inv [\ef_{\mu_n} \cdot \phi_n] }_{L^2(X,\mu_n)} = \int_X \ef_{\mu_n}  \phi_n  u_n \diff \mu_n,
    \end{equation*}     
    and the integrand \(\ef_{\mu_n}  \phi_n  u_n\) converges uniformly to \(\ef_\mu  \phi  u\), while the measure \(\mu_n\) converges \textweakstar~to \(\mu\). It is easy to pass to the limit using Lemma~\ref{lemma:pair_strong_weak}.
\end{proof}

\begin{proposition}[Equivalence to the norm on \(\Hil_c\)] \label{prop:tgmu_Hcalphadot_equiv}
    Let \(\mu \in \prm(X)\), \(\betadot \in \Hil_c\) with \(\pair{\B(\mu)}{\betadot}_{\Hil_c} = 0\). It holds that
    \begin{equation*}
        \frac{\varepsilon}{2} \norm{\betadot}_{\Hil_c}^2 \leq\tg_\mu(\betadot,\betadot) \leq C \frac{\varepsilon}{2} \norm{\betadot}_{\Hil_c}^2
    \end{equation*}
    for \(C = 1 + 2 \exp\left( \frac{11}{2 \varepsilon} \norm{c}_\infty \right) \).
\end{proposition}

\begin{proof}
    From Theorem~\ref{theorem:inverse_operators_Hmu}, on the space $L^2(X,\mu) / \R$ we have \(\id \leq (\id - K_\mu)\inv \leq \frac{1}{1-q} \id\) with $q = 1 - \exp(-4\norm{c}_\infty / \varepsilon)$ as in Proposition~\ref{prop:Kmu_contraction} (and in particular \((\id - K_\mu)\inv\) is non-negative). Thus we get
    \begin{equation*}
        \frac{\varepsilon}{2} \norm{\betadot}_{\Hil_c}^2 \leq \tg_\mu(\betadot,\betadot) \leq \frac{\varepsilon}{2} \left( \norm{\betadot}_{\Hil_c}^2 + \frac{2}{1-q} \norm{\ef_{\mu} \cdot \betadot}^2_{L^2(X,\mu) / \R} \right).
    \end{equation*}
    We then control the $L^2(X,\mu) / \R$ norm by the supremum norm, using that $\norm{\betadot}_{\infty} \leq \norm{\betadot}_{\Hil_c}$ as the cost function is non-negative. Thus 
    \begin{equation*}
        \frac{\varepsilon}{2} \norm{\betadot}_{\Hil_c}^2 \leq \tg_\mu(\betadot,\betadot) \leq \frac{\varepsilon}{2} \left( \norm{\betadot}_{\Hil_c}^2 + \frac{2}{1-q} \| \ef_\mu \|_\infty \norm{\betadot}^2_{\Hil_c} \right).
    \end{equation*}
    Eventually, using~\cite[Thm.~1.2]{DiMarino2020} we have \( \norm{\ef_\mu}_\infty \leq \exp( \norm{\fmumu}_\infty / \varepsilon )  \leq \exp( \frac{3}{2} \norm{c}_\infty / \epsilon)\), and together with the definition of $q$ we conclude the proof.
\end{proof}

\subsection{The metric tensor as \texorpdfstring{$\varepsilon$}{epsilon} varies} \label{sec:sub:epsilon_limits}

Let us now discuss how the metric tensor varies as a function of $\varepsilon$. Though it is not the main goal of our work, which rather focuses on showing that, for $\varepsilon$ fixed, $S_\varepsilon$ induces a Riemannian geometry, it can give some intuition on the influence of the parameter $\varepsilon$. For this discussion we fix $X$ to be a bounded convex subset of $\R^d$ with non-empty interior, and $c(x,y) = \norm{x-y}^2$ is the quadratic cost. We add a superscript $\varepsilon$ to stress the fact that, contrarily to the rest of this work, $\varepsilon$ is not fixed anymore. At the level of the Sinkhorn divergence, as the Kullback--Leibler divergence is invariant under reparametrization, 
\begin{equation} \label{eq:S_rescaled}
    S_\varepsilon(\mu,\nu) = \varepsilon S_1(L^\varepsilon_\# \mu, L^\varepsilon_\# \nu),     
\end{equation}
with $L^\varepsilon(x) = x/\sqrt{\varepsilon}$, see \cite{MenaWeed}. It shows how $\sqrt{\varepsilon}$ acts as a typical length scale.

We first show that, for $\varepsilon \in (0, \infty)$, the metric tensor changes smoothly with respect to $\varepsilon$ (however it is not a monotone function of $\varepsilon$ for $\mu, \dot \mu$ fixed: this can be seen in Figure~\ref{figure:nonconvexity_gmu} below). 

\begin{proposition}
    Fix  \(\mu \in \prm(X)\) and \(\dot \mu \in \Cont^m(X)^*_0\) for some $m \geq 0$, then the function $\varepsilon \mapsto \g^\varepsilon_{\mu}(\mudot,\mudot)$ defined on $(0, \infty)$ is continuous.
\end{proposition}

\noindent In this proposition we restrict to tangent vectors in $\Cont^m(X)^*_0$, which embed into the \(\varepsilon\)-dependent tangent space $\Hil^*_{\mu,0}$ for any $\varepsilon$. 

\begin{proof}
    By checking that it satisfies the Schrödinger system, we have, with $L_\varepsilon(x) = x/\sqrt{\varepsilon}$,
    \begin{equation*}
        f^\varepsilon_{\mu,\mu}(x) = \varepsilon f^1_{L^\varepsilon_\# \mu,L^\varepsilon_\# \mu} \left( \frac{x}{\sqrt{\varepsilon}} \right).
    \end{equation*}
    With Proposition~\ref{prop:reg_fmunu_cm}, it gives that $f^\varepsilon_{\mu,\mu}$ varies continuously with $\varepsilon$ in $\Cont^m(X)$. From this and the bound of Proposition~\ref{prop:Kmu_contraction} which is locally uniform in $\varepsilon$, we see that both $(\id - K^\varepsilon_\mu)^{-1}$ and $H^\varepsilon_\mu$ vary continuously in operator norm with $\varepsilon$. The conclusion follows from the explicit expression of $ \g^\varepsilon_{\mu}(\mudot,\mudot)$, see Definition~\ref{def:metric_tensor_measure}. 
\end{proof}

We now move to the limit cases $\varepsilon \to 0$ and $\varepsilon \to \infty$. The connection between optimal transport, regularized optimal transport and MMD distances in these limit cases has been investigated in the literature \cite{ramdas2017wasserstein}. It is thus a natural question to study what happens to our metric metric tensor in these same limiting cases. We first focus on the large $\varepsilon$ limit as it is the easiest case.

\begin{theorem}
    For $\mu \in \prm(X)$ and $\mudot = - \ddiv(\mu v)$ with $v \in L^1(X,\mu;\R^d)$, it holds that
    \begin{equation*}
        \lim_{\varepsilon \to \infty} \mathbf{g}^\varepsilon_\mu(\dot{\mu},\dot{\mu}) 
        = \left\| \int_X v(x) \diff \mu(x) \right\|^2.
    \end{equation*}
\end{theorem}

\noindent The limit of the metric tensor is degenerate, as $\left\| \int v(x) \diff \mu(x) \right\| = \left\| \dot{m}_\mu \right\|$ with $m_\mu \in \R^d$ the center of mass of $\mu$. Thus in the limit only the motion of the center of mass matters. 

\begin{proof}
    From Equation~\eqref{eq:control_metric_tensor_Hmu}, as $q \to 0$ when $\varepsilon \to \infty$, we have that
    \begin{align*}
        \mathbf{g}^\varepsilon_\mu(\dot{\mu},\dot{\mu}) \sim \frac{\varepsilon}{2} \| \dot{\mu} \|^2_{\mathcal{H}^{*,\varepsilon}_{\mu,0}} 
        &= \frac{\varepsilon}{2} \iint_{X \times X} k^\varepsilon_\mu(x,y) \diff \dot{\mu}(x) \diff \dot{\mu}(y)\\ 
        &= \frac{\varepsilon}{2} \iint_{X \times X} [ \nabla_{xy} k^\varepsilon_\mu(x,y) ] \cdot [v(x) \otimes v(y)] \diff \mu(x) \diff \mu(y),
    \end{align*}
    where the last equality follows from integration by parts. We then compute
    \begin{equation*}
        \nabla_{xy} k^\varepsilon_\mu(x,y) = \left(  \frac{2}{\varepsilon} \id + \frac{(\nabla f^\varepsilon_{\mu,\mu}(x) - 2(x-y)) \otimes (\nabla f^\varepsilon_{\mu,\mu}(y) - 2(y-x))}{\varepsilon^2} \right) k^\varepsilon_{\mu}(x,y).
    \end{equation*}
    But $k^\varepsilon_{\mu} \to 1$ in $\Cont(X \times X)$ and $\pi_{\mu,\varepsilon} \to \mu \otimes \mu$ as $\varepsilon \to \infty$. This is clear as $\|f^\varepsilon_{\mu,\mu}\|_\infty \leq 3/2 \| c \|_\infty$ for any $\varepsilon$. Moreover $2y - \nabla f^\varepsilon_{\mu,\mu}(y) = 2 \int k^\varepsilon_{\mu}(y,z) z  \diff \mu(z)$, this is a classical identity that we rederive below in~\eqref{eq:gradient_fmumu}. We deduce that $2\id - \nabla f^\varepsilon_{\mu,\mu}$ converges to $2 m_\mu$, as $\varepsilon \to \infty$, hence is bounded. As moreover $X$ is bounded, putting these considerations together, as a (matrix-valued) function,
    \begin{equation*}
        \frac{\varepsilon}{2} \nabla_{xy} k^\varepsilon_\mu(x,y) \to \id, \quad \text{ as }\varepsilon \to \infty
    \end{equation*}
    uniformly on \(X \times X\). Plugging this into the previous computations, we find 
    \begin{equation*}
        \lim_{\varepsilon \to \infty} \mathbf{g}^\varepsilon_\mu(\dot{\mu},\dot{\mu}) = \lim_{\varepsilon \to \infty} \frac{\varepsilon}{2} \| \dot{\mu} \|^2_{\mathcal{H}^{\varepsilon,*}_{\mu,0}} = \iint_{X \times X} v(x) \cdot v(y) \diff \mu(x) \diff \mu(y). \qedhere
    \end{equation*}
\end{proof}

We then turn to the regime of vanishing $\varepsilon$. In this case we only do informal computations, as making the expansion rigorous would require a fine analysis of the behavior of the Schrödinger potentials as $\varepsilon \to 0$ in the compact case. We identify a measure with its density with respect to the Lebesgue measure. Writing $\Delta$ for the Laplacian (the trace of the Hessian matrix) and $L f \assign  - \Delta f  - \nabla f \cdot \nabla \log \mu $, we claim that it should hold 
\begin{align} \label{eq:informal_limit_small_eps}
	{\bf{g}}^\varepsilon_{\mu}(\dot\mu, \dot\mu) 
	& \to \left\langle \frac{\dot\mu}{\mu} , L ^{-1} \left[ \frac{\dot\mu}{{\mu}} \right] \right \rangle_{L^2(X,\mu)}, \quad \text{ as } \varepsilon\to 0.
\end{align}

\begin{remark}
    In the case where $ \dot \mu = - \ddiv (\mu \nabla \psi)$, we see that the limit in~\eqref{eq:informal_limit_small_eps} is the metric tensor $\g_{\mu}^0$ of classical optimal transport as defined in~\eqref{eq:intro_gmu_OT}. Indeed, in this setting $L^{-1}[\mudot/\mu] = \psi$, which can be checked by expanding the divergence:
    \begin{equation*}
        \frac{\dot \mu}{\mu} = - \frac { \nabla\mu \cdot \nabla \psi }{\mu} - \Delta \psi = -\nabla \log\mu \cdot \nabla \psi  - \Delta \psi = L \psi.         
    \end{equation*}
    Thus we obtain from~\eqref{eq:informal_limit_small_eps}
    \begin{equation*}
        {\bf{g}}^\varepsilon_{\mu}(\dot\mu, \dot\mu) 
        \to \left\langle \ddiv (\mu \nabla \psi), \psi \right \rangle 
        =  \int_X \lvert \nabla\psi \rvert^2 \diff \mu = \g^0_\mu(\mudot,\mudot).  
    \end{equation*}
\end{remark}

We now discuss the sketch of the proof of the claim~\eqref{eq:informal_limit_small_eps}. The main observation is that if one has an operator expansion \( K^\varepsilon_\mu = \id + \varepsilon \tilde{L} + \oh(\varepsilon) \) for some invertible operator $\tilde{L}$, then it holds that
\(
 (K^\varepsilon_\mu)^2 = \id  + 2\varepsilon \tilde{L} +\oh(\varepsilon).
\)
Therefore, as $H^\varepsilon_\mu[\vdot] = K^\varepsilon_\mu[\vdot / \mu]$, we have 
\begin{align*}
    \frac{\varepsilon}2 \left\langle \dot\mu, (\id - (K^\varepsilon_\mu)^2)^{-1} K^\varepsilon_\mu \left[\frac{\dot\mu}{\mu}\right]\right\rangle
    &=  \frac14 \left\langle \dot\mu, \left(-  \tilde{L} +\oh(1)\right)^{-1} \big( \id + \varepsilon \tilde{L} +\oh(\varepsilon) \big) \left[\frac{\dot\mu}{\mu}\right]\right\rangle \\
     &=  \frac14 \left\langle \dot\mu, (-\tilde{L})^{-1} \left[\frac{\dot\mu}{\mu}\right]\right\rangle 
    +\oh(1), 
\end{align*} 
and the conclusion follows if we check that $L = - 4\tilde{L}$. Thus we need to identify $\tilde{L}$, i.e.~the first order term in the expansion of $K^\varepsilon_\mu$ as $\varepsilon \to 0$. It has already been characterized in some cases: in \cite{marshall2019manifold},  where the ambient space is a Riemannian manifold, and in \cite{agarwal2025langevin}, which gives a probabilistic interpretation of the resulting expansion, however it needs measures supported on the whole space $\R^d$. None of these works prove a limit in a sense strong enough to rigorously deduce~\eqref{eq:informal_limit_small_eps} as far as we understand. We leave it for future work to give a mathematical proof of~\eqref{eq:informal_limit_small_eps}.

We will only explain how one can guess the limit, following the arguments of \cite{mordant2024entropic}. To identify $\tilde{L}$ at least informally, we need a better understanding of the self-dual potentials. Identifying a measure with its density, in the limit $\varepsilon \to 0$ the Laplace formula gives
\begin{equation*}
    H^\varepsilon_{c}[\phi](y) = \int_X \exp\left( -\frac1{\varepsilon} \| x- y\|^2 \right) \phi (x) \diff x = (\pi \varepsilon)^{d/2} \left( \phi(x) + \frac{\varepsilon}{4} \Delta \phi(y) + \oh(\varepsilon) \right).
\end{equation*}
With $a_\varepsilon = \exp(f^\varepsilon_{\mu,\mu}/ \varepsilon)$, the condition $K^\varepsilon_\mu[\ones_X] = \ones_X$ reads $a_\varepsilon H^\varepsilon_{c}[a_\varepsilon \mu] = \ones_X$, that is, 
\begin{equation*}
    a_\varepsilon^2 = \left( (\pi \varepsilon)^{d/2} \left( \mu + \frac{\varepsilon}{4} \frac{\Delta (a_\varepsilon \mu)}{a_\varepsilon} + \oh(\varepsilon)  \right) \right)^{-1} = \frac{1}{(\pi \varepsilon)^{d/2} \mu} \left( 1 - \frac{\varepsilon}{4} \frac{\Delta (a_\varepsilon \mu)}{a_\varepsilon \mu} + \oh(\varepsilon) \right). 
\end{equation*}
Thus at first order $a_\varepsilon \sim (\pi \varepsilon)^{-d/4} \mu^{-1/2}$, which plugged again in the expansion in the right-hand side yields
\begin{equation*}
    a_\varepsilon = \frac{1}{(\pi \varepsilon)^{d/4} \sqrt{\mu}} \left( 1 - \frac{\varepsilon}{8} \frac{\Delta \sqrt{\mu}}{\sqrt{\mu}} + \oh(\varepsilon) \right).
\end{equation*}
As $K^\varepsilon_\mu[\phi] = a_\varepsilon H^\varepsilon_{c}[ a_\varepsilon \mu \phi ]$, we can again use the Laplace expansion, keeping only the terms of order $1$ in $\varepsilon$ and noticing that the scaling factors $(\pi \varepsilon)^{d/2}$ cancel,
\begin{align*}
    K^\varepsilon_\mu[\phi] & = (\pi \varepsilon)^{d/2} a_\varepsilon \left( a_\varepsilon \mu \phi + \frac{\varepsilon}{4} \Delta( a_\varepsilon \mu \phi) + \oh(\varepsilon) \right) \\
    & =  \frac{1}{\sqrt{\mu}} \left( 1 - \frac{\varepsilon}{8} \frac{\Delta \sqrt{\mu}}{\sqrt{\mu}} \right) 
    \left( \sqrt{\mu}  \left(  1 - \frac{\varepsilon}{8} \frac{\Delta \sqrt{\mu}}{\sqrt{\mu}} \right) \phi + \frac{\varepsilon}{4} \Delta (\sqrt{\mu} \phi) \right)  + \oh(\varepsilon) \\
    & = \phi + \frac{\varepsilon}{4 \sqrt{\mu}} \left( - \phi \Delta \sqrt{\mu} + \Delta(\sqrt{\mu} \phi)  \right) + \oh(\varepsilon).
\end{align*}
As $\Delta(\sqrt{\mu} \phi) - \phi \Delta \sqrt{\mu} = \sqrt{\mu}\Delta(\phi) + 2 \nabla \phi \cdot \nabla \sqrt{\mu}$ and $2 \nabla \sqrt{\mu} = \sqrt{\mu} \nabla \log \mu$, we obtain
\begin{equation*}
    K^\varepsilon_\mu[\phi] = \phi + \frac{\varepsilon}{4} (\Delta \phi + \nabla \log \mu \cdot \nabla \phi) + \oh(\varepsilon),
\end{equation*}
which is sufficient to identify $\tilde{L} = \frac{1}{4}(\Delta + \nabla \log \mu \cdot \nabla ) $ and to see that $L = - 4\tilde{L}$ coincides with the expression given above~\eqref{eq:informal_limit_small_eps}.

\section{The induced path metric} \label{section:path_metric}

\subsection{Definition and metrization of the weak-\texorpdfstring{$\ast$}{star}~topology} \label{section:sub:path_metric_intro}

\begin{definition} \label{def:distance}
\begin{enumerate}
    \item 
    We call a path \((\mu_t)_{t \in [0,1]}\) admissible when the corresponding path \((\beta_t)_{t \in (0,1)} = (\B(\mu_t))_{t \in (0,1)}\) belongs to the Sobolev space \(\Hil^1((0,1);\Hil_c)\) (defined in Appendix~\ref{section:appendix_sobolev}). We denote the set of admissible paths as \(\Pall\) and write \(\Pall(\mu_0, \mu_1)\) for the subset of \(\Pall\) with fixed start and end points \(\mu_0, \mu_1 \in \prm(X)\).
    
    \item 
   The energy of a path \((\mu_t)_t \in \Pall\) is defined as
    \begin{equation*}
        E((\mu_t)_t)\assign \int_0^1 \tg_{\mu_t}(\betadot_t,\betadot_t) \diff t = \int_0^1 \g_{\mu_t}(\mudot_t,\mudot_t) \diff t.
    \end{equation*}
    (These two integrals indeed coincide if  \((\mu_t)_t \in \Pall\) thanks to Theorem~\ref{theorem:metric_tensor_in_betadot}.)

    \item 
    For \(\mu_0,\mu_1 \in \prm(X)\) we define 
    \begin{equation} \label{eq:def_dE}
        \dS(\mu_0,\mu_1) \assign \left( \inf_{(\mu_t)_t\in\Pall(\mu_0,\mu_1)} E((\mu_t)_t) \right)^\frac{1}{2}.
    \end{equation}
\end{enumerate}
\end{definition}

Theorem~\ref{theorem:alpha_beta_homeo} tells us that \(\Pall\) corresponds to \(\Hil^1((0,1) ; \calB)\) and the paths \((\mu_t)_t\) are \textweakstar-continuous. 

If the cost function is of class $\Cont^{(m,m)}$, then \(\Cont^m\)-perturbation are admissible, as we will see later in Corollary~\ref{cor:Cm_admissible}. Intuitively, \(\Cont^m\)-perturbations (which include combinations of vertical and horizontal perturbations) would seem sufficiently rich to describe $\mudot$ for all relevant trajectories, but we extend admissible paths to $\Pall$ to guarantee completeness. We now show that $\dS$ is a distance (Theorem~\ref{theorem:dS_metric}) which is geodesic (Theorem~\ref{theorem:dE_minimizers_existence}).

\begin{theorem} \label{theorem:dS_metric}
    The function \(\dS\) is a metric on \(\prm(X)\) that metrizes the \textweakstar~topology. 
    Further, for all \(\mu_0, \mu_1 \in \prm(X)\), \(\beta_0 = \B(\mu_0)\), \(\beta_1 = \B(\mu_1)\) there holds 
    \begin{equation*}
        \sqrt{\frac{\varepsilon}{2}} \, \norm{ \beta_0 - \beta_1 }_{\Hil_c} \leq \dS(\mu_0,\mu_1) \leq {\frac{\pi}{2} \sqrt{\frac{\varepsilon}{2} C}} \, \norm{ \beta_0 - \beta_1 }_{\Hil_c}
    \end{equation*}
    with \(C =  1 + 2 \exp\left( \frac{11}{2 \varepsilon} \norm{c}_\infty \right)\) as in Proposition~\ref{prop:tgmu_Hcalphadot_equiv}.
\end{theorem}

\begin{proof}
    It is easy to see that \(\dS(\mu_0, \mu_1) = \dS(\mu_1, \mu_0) \geq 0\). We show the triangle inequality and the relation to the \(\Hil_c\) norm. The latter also implies finiteness, positive definiteness, and the metrization of the \textweakstar~topology. 
    
    The proof of the triangle inequality is analogous to that of~\cite[Thm.~1]{Brenier2020}. Let \(\mu,\nu,\eta \in \prm(X)\), \((\mu_t)_t \in \Pall(\mu,\nu)\), \((\nu_t)_t \in \Pall(\nu,\eta)\) and denote \(\beta^1_t = \B(\mu_t)\), \(\beta^2_t = \B(\nu_t)\). Assume that \(\dS(\mu,\nu) > 0\) and \(\dS(\nu,\eta) > 0\). Let \(\tau\in (0,1)\) and set \(\gamma_t = \mu_{t/\tau}\) for \(t\in[0,\tau]\), \(\gamma_t = \nu_\frac{t-\tau}{1-\tau}\) for \(t \in [\tau,1]\). Using Theorem~\ref{theorem:H1_reformulation} and the \textweakstar~continuity of \((\gamma_t)_t\), one can see that \((\gamma_t)_t \in \Pall(\mu,\eta)\) and by simple rescaling there holds 
    \begin{align*}
        E((\gamma_t)_t) &= \frac{1}{\tau} \int_0^1 \tg_{\mu_t}(\betadot^1_t,\betadot^1_t) \diff t + \frac{1}{1-\tau} \int_0^1 \tg_{\nu_t}(\betadot^2_t,\betadot^2_t) \diff t \\
        &= \frac{1}{\tau} E((\mu_t)_t) + \frac{1}{1-\tau} E((\nu_t)_t).    
    \end{align*}
    Setting \(\tau = {E((\mu_t)_t)^\frac{1}{2}} / {(E((\mu_t)_t)^\frac{1}{2} + E((\nu_t)_t)^\frac{1}{2})}\) gives \[\frac{1}{\tau} E((\mu_t)_t) + \frac{1}{1-\tau} E((\nu_t)_t) = (E((\mu_t)_t)^\frac{1}{2} + E((\nu_t)_t)^\frac{1}{2})^2.\] Taking the infimum over \((\mu_t)_t \in \Pall(\mu,\nu)\) and \((\nu_t)_t \in \Pall(\nu,\eta)\) yields the triangle inequality for \(\dS\).
    
    To show equivalence to the \(\Hil_c\)-norm, we use that \( \frac{\varepsilon}{2} \norm{ \betadot_t }_{\Hil_c}^2 \leq \tg_{\mu_t} (\betadot_t, \betadot_t) \leq C \frac{\varepsilon}{2} \norm{ \betadot_t }_{\Hil_c}^2 \) from Proposition~\ref{prop:tgmu_Hcalphadot_equiv}. 
    
    For the lower bound, let \((\mu_t)_t \in \Pall(\mu_0,\mu_1)\) and denote \(\beta_t = \B(\mu_t)\). Then Jensen's inequality, the norm estimate from Theorem~\ref{theorem:H1_norm_estimate}, and the fundamental theorem of calculus in \(\Hil^1((0,1);\Hil_c)\) from Theorem~\ref{theorem:H1_reformulation} show
    \begin{equation*}
        \frac{2}{\varepsilon} \int_0^1 \tg_{\mu_t} (\betadot_t,\betadot_t) \diff t 
        \geq \int_0^1 \norm{ \betadot_t }_{\Hil_c}^2 \diff t
        \geq \bnorm{\int_0^1 \betadot_t \diff t}_{\Hil_c}^2 
        = \norm{ \beta_1 - \beta_0 }_{\Hil_c}^2.
    \end{equation*} 
    This gives the lower bound when we take the infimum over all curves.
    
    For the upper bound, let \((\beta_t)_t\) be the arc connecting \(\beta_0\) to \(\beta_1\) with \(\norm{\beta_t}_{\Hil_c} = 1\) parametrized by arc length of \((\beta_t)_t\) in \(\Hil_c\). By Theorem~\ref{theorem:alpha_beta_homeo}, this is an admissible path. Proposition~\ref{prop:tgmu_Hcalphadot_equiv} and a basic geometry argument in the two-dimensional Hilbert space \(\vecspan\{\beta_0, \beta_1\}\) yield 
    \begin{align*}
        \sqrt{\frac{2}{C \varepsilon }} \dS(\mu_0,\mu_1) \leq 
        \left( \int_0^1 \norm{\betadot_t}_{\Hil_c}^2 \diff t \right)^\frac{1}{2} = \int_0^1 \norm{\betadot_t}_{\Hil_c} \diff t  = 2 \arcsin \left( \frac{1}{2} \norm{\beta_1 - \beta_0}_{\Hil_c} \right), 
    \end{align*}
    which, in turn, is bounded by \(\frac{\pi}{2} \norm{\beta_1 - \beta_0}_{\Hil_c}\). 
    
    Finally, for \(\mu, \mu_n \in \prm(X)\), the established estimate and Theorem~\ref{theorem:alpha_beta_homeo} show that 
    \begin{align*}
        \dS(\mu_n,\mu) \to 0 \quad \Leftrightarrow \quad \norm{\B(\mu_n) - \B(\mu)}_{\Hil_c} \to 0 \quad \Leftrightarrow \quad \mu_n \to \mu \textnormal{ \textweakstar.} \tag*{\qedhere}
    \end{align*}
\end{proof}

\begin{remark} \label{remark:equivalence_Hc}
    Theorem~\ref{theorem:dS_metric} yields a metric equivalence between the flat norm on $\Hil_c$ and $\mathsf{d}_S$. In particular these two distances define the same class of absolutely continuous curves. However, the geometry induced by the metric tensor $\tilde{\mathbf{g}}$ differs from the flat one. We give one example of this: the shortest path (a.k.a. minimizing geodesic) between \(\beta_0\) and \(\beta_1\) in \(\calB \subseteq \Hil_c\) in the RKHS geometry is given by arc interpolation. This path lies entirely in the plane spanned by \(\beta_0, \beta_1\). The equivalent holds for the corresponding path between \(\alpha_0\) and \(\alpha_1\) in \(\calM_+(X)\). Hence the support of the intermediate measures \(\alpha_t\) (a fortiori of \(\mu_t\)) is contained in the union of the support of the initial and final measure. In contrast, geodesics in \(\dS\) can have horizontal movement of mass that leaves the original support, for instance in the case of translations as in Section~\ref{section:sub:mean_estimate}. 
\end{remark}

\subsection{Existence of geodesics} \label{section:sub:geodesics}

Given the definition of $\dS$ as infimization over the energy of paths, it is natural to ask whether geodesics in the sense of energy minimizing paths exist. The set of admissible paths \(\Pall\) has been chosen carefully to obtain existence of geodesics for \(\dS\).

\begin{theorem} \label{theorem:dE_minimizers_existence}
    The variational problem~\eqref{eq:def_dE} defining \(\dS\) has a minimizer. Hence there exist Riemannian geodesics for the metric \(\dS\).
\end{theorem}

\noindent The proof will follow the direct method of calculus of variations. Indeed, we first show lower semicontinuity of \(E\) on \(\Hil^1((0,1);\Hil_c)\) by rewriting \(E((\mu_t)_t)\) as a supremum of weakly lower semicontinuous problems in Lemma~\ref{lemma:tgmu_sup}. Using that boundedness in \(E\) implies boundedness in the Sobolev norm, we then find that the weak cluster points of any minimizing sequence for the least energy problem~\eqref{eq:def_dE} are minimizers, hence geodesic. This geodesic belongs to \(\Pall\), as \(\Hil^1((0,1) ; \calB)\) is weakly closed.

\begin{remark}
    We can obtain the existence of geodesics for \(\dS\) in the metric sense in a much simpler way by showing that in \((\prm(X), \dS)\) all pairs of points almost admit a midpoint (see~\cite[Def.~3.6]{Brenier2020}) and applying the Hopf--Rinow Theorem \cite[Lem.~B.3]{Brenier2020}. For any \(\mu_0,\mu_1 \in \prm(X)\) this would give a path \((\mu_t)_t\) in \(\prm(X)\), such that for all \(t,s \in [0,1]\) it holds
    \begin{equation*}
        \dS(\mu_t,\mu_s) = \abs{t - s} \cdot \dS(\mu_0,\mu_1).
    \end{equation*}
    However, it is not immediately clear whether the path \((\mu_t)_t\) belongs to \(\Pall(\mu_0, \mu_1)\) and whether \(\dS(\mu_0,\mu_1) = E((\mu_t)_t)\), i.e.~whether \((\mu_t)_t\) is also a geodesic in the Riemannian sense as a minimizer of the energy in~\eqref{eq:def_dE}.
\end{remark}

\begin{lemma} \label{lemma:tgmu_sup}
    Let \((\mu_t)_t \in \Pall\) and define \(\boldsymbol{\mu} \in \calM([0,1] \times X)\) via \(\pair{\boldsymbol{\mu}}{\psi} = \iint \psi(t,\vdot) \diff \mu_t \diff t\) for any \(\psi \in \Cont([0,1] \times X)\). The energy of the path \((\mu_t)_t\) rewrites as 
    \begin{multline} \label{eq:lemma:tgmu_sup}
        E((\mu_t)_t) = 
        \frac{\varepsilon}{2} \int_0^1 \norm{ \betadot_t }_{\Hil_c}^2 \diff t 
        + 2 \varepsilon \cdot \sup_\psi \Bigg\{ \int_0^1 \pair{ H_c[\psi(t,\vdot) A(\mu_t)] }{ \betadot_t }_{\Hil_c} \, \mathrm{d} t  \\ 
        - \frac{1}{2} \int_0^1  \pair{ \psi(t,\vdot) }{ (\id - K_{\mu_t}) \psi(t ,\vdot) }_{L^2(X,\mu_t)} \, \mathrm{d} t \Bigg\}, 
    \end{multline}
    where the supremum in \(\psi\) is taken over any space of functions dense in \(L^2([0,1] \times X, \boldsymbol{\mu})\), e.g.~the jointly continuous functions over \([0,1] \times X\). 
\end{lemma}

\begin{proof}
    It is a well-known fact in convex analysis that if \(T\) is a symmetric, coercive and continuous operator on a Hilbert space \(\Hil\) it holds
    \begin{equation} \label{eq:norm_fenchel_transform}
        \frac{1}{2} \pair{x}{T\inv x}_\Hil = \sup_y \, \pair{x}{y}_\Hil - \frac{1}{2} \pair{y}{Ty}_\Hil
    \end{equation}
    for all \(x \in \Hil\). This is Fenchel-Legendre duality, or could also be seen as a corollary of the Lax-Milgram theorem~\cite[Corollary 5.8]{Brezis2011}. Continuity of $T$ allows the restriction of the supremum to a dense subset of \(\Hil\). 
    We apply~\eqref{eq:norm_fenchel_transform} to 
    \begin{itemize}
        \item the space \(\Hil = \setgiven{ \psi \in L^2([0,1] \times X, \boldsymbol{\mu}) }{ \pair{\mu_t}{\psi(t,\vdot)} = 0 \textup{ a.e.} }\), i.e.~the orthogonal complement of functions which are \(\mu_t\)-a.e.~constant in space at almost all \(t \in [0,1]\),
        \item the operator \(T[\psi] (t,y) = (\id - K_{\mu_t})[\psi(t,\vdot)](y)\),
        \item the point \(x = (\ef_{\mu_t} \betadot_t)_t \in \Hil\).
    \end{itemize}
    The operator \(T\) is coercive and continuous on \(\Hil\) by Theorem~\ref{theorem:inverse_operators_Hmu}. In this notation
    \begin{equation*}
        E((\mu_t)_t) = \int_0^1 \tg_{\mu_t} ( \betadot_t, \betadot_t ) \diff t = \frac{\varepsilon}{2} \int_0^1 \norm{\betadot_t}^2_{\Hil_c} \diff t + \varepsilon \pair{x}{T\inv x}_{\Hil}.
    \end{equation*}
    We then use that \(\pair{\ef_{\mu_t} \betadot_t}{\psi(t,\vdot)}_{L^2(X,\mu_t)} = \pair{\psi(t,\vdot) \ef_{\mu_t} \mu_t}{\betadot_t} = \pair{H_c[\psi(t,\vdot) A(\mu_t)]}{\betadot_t}_{\Hil_c}\). 
    With $\pair{\betadot_t}{\beta_t} = 0$ and \((\id - K_{\mu_t}) [\ones_X] = 0\), one sees that the expression on the right-hand side of~\eqref{eq:lemma:tgmu_sup} gives the same value on functions whose difference is constant in space for almost all times \(t\). Hence the supremum remains the same when taken on \(\Hil\) or \(L^2([0,1] \times X; \boldsymbol{\mu})\).
\end{proof}

\begin{proof}[\textbf{Proof of Theorem~\ref{theorem:dE_minimizers_existence}}]
    We apply the direct method of calculus of variations. Let \((\mu^n_t)_t \in \Pall(\mu_0,\mu_1)\) with corresponding \((\beta^n_t)_t \in \Hil^1((0,1); \Hil_c)\) be such that 
    \begin{equation*}
        \dS(\mu_0,\mu_1)^2 = \lim_{n \to \infty} \int_0^1 \tg_{\mu^n_t}(\betadot^n_t, \betadot^n_t) \diff t. 
    \end{equation*}
    From $\norm{\beta_t^n}_{\Hil_c} = 1$ and Proposition~\ref{prop:tgmu_Hcalphadot_equiv} we see that the sequence of paths \((\beta^n_t)_t\) must be bounded in \(\Hil^1((0,1);\Hil_c)\). Hence, up to taking a subsequence, \((\beta^n_t)_t\) converges weakly in \(\Hil^1((0,1);\Hil_c)\) to some limit path \((\beta_t)_t\). 
    It remains to check that \(\beta_t \in \calB\) for all \(t \in (0,1)\) and that the corresponding path of \(\mu_t = \B\inv(\beta_t)\) satisfies \(E((\mu_t)_t) = \dS(\mu_0,\mu_1)^2\).
    
    Note that the sequence \((\beta^n_t)\) is uniformly equi-continuous: For any \(n \in \N\), \(s<t \in [0,1]\) set \(v = (\beta^n_t - \beta^n_s) / \norm{\beta^n_t - \beta^n_s}_{\Hil_c}\). Then Theorem~\ref{theorem:H1_reformulation} and the Cauchy--Schwarz inequality give
    \begin{equation*}
        \norm{\beta^n_t - \beta^n_s}_{\Hil_c} = \int_s^t \pair{\betadot^n_\tau}{v} \diff \tau = \pair{(\betadot^n_\tau)_\tau}{v \cdot \ones_{[s,t]}}_{L^2([0,1];\Hil_c)} \leq \norm{(\beta^n_\tau)_\tau}_{\Hil^1((0,1);\Hil_c)} \cdot \sqrt{|t-s|}.
    \end{equation*}
    Thus, again up to taking a subsequence, we can assume that \((\beta^n_t)\) converges uniformly, hence the limit satisfies $\beta_t \in \calB$ for all \(t \in (0,1)\). 
    
    We now estimate the energy. By lower semicontinuity of the norm for weak convergence, we have directly 
    \begin{equation*}
    \int_0^1 \norm{ \betadot_t }_{\Hil_c}^2 \diff t \leq \liminf_{n \to \infty} \int_0^1 \norm{ \betadot^n_t }_{\Hil_c}^2 \diff t.
    \end{equation*}
    Thus by Lemma~\ref{lemma:tgmu_sup}, using $\alpha^n_t = \A(\mu^n_t) = a_{\mu^n_t} \mu^n_t$ and rewriting the operator \(K_{\mu^n_t}\) using \(k_{\mu^n_t} = k_c \cdot \ef_{\mu^n_t} \otimes \ef_{\mu^n_t}\), to conclude the proof it suffices to show that for all \(\psi \in \Cont([0,1] \times X)\) we have 
    \begin{align*}
        &\liminf_{n \to \infty}  \int_0^1 \bpair{  H_c[\psi(t,\vdot) \alpha^n_t] }{  \betadot^n_t }_{\Hil_c} \mathrm{d} t -  \frac{1}{2} \int_0^1 \int_X \psi(t,x)^2 \diff \mu^n_t(x) \diff t \\ 
        &\qquad\qquad + \frac{1}{2} \int_0^1 \iint_{X \times X} k_c(x,y) \psi(t,x) \psi(t,y) \diff \alpha^n_t(x) \diff \alpha^n_t(y) \diff t \\
        &\qquad \geq \int_0^1  \bpair{  H_c[\psi(t,\vdot) \alpha_t] }{  \betadot_t }_{\Hil_c} \mathrm{d} t -  \frac{1}{2} \int_0^1 \int_X \psi(t,x)^2 \diff \mu_t(x) \diff t \\ 
        &\qquad\qquad + \frac{1}{2} \int_0^1 \iint_{X \times X} k_c(x,y) \psi(t,x) \psi(t,y) \diff \alpha_t(x) \diff \alpha_t(y) \diff t. 
    \end{align*}
    As $\beta^n_t$ converges to $\beta_t$, we know from Theorem~\ref{theorem:alpha_beta_homeo} that \(\mu^n_t \to \mu_t\), \(\alpha^n_t \to \alpha_t\) \textweakstar~in \(\calM(X)\). In addition, we can bound the mass of \(\alpha^n_t\) and the supremum of \(\psi(t,\vdot)\) uniformly in \(t\) and \(n\), and so the second and third integral are easy to pass to the limit. For the first one note that
    \begin{equation*}
    (H_c[\psi(t,\vdot)\alpha^n_t])_t \to (H_c[\psi(t,\vdot)\alpha_t])_t.
    \end{equation*}
    The convergence is pointwise in $t$ thanks to Lemma~\ref{lemma:h_k_Cm_compactness}, and also in norm in \(L^2([0,1]; \Hil_c)\) thanks to the Lebesgue dominated convergence theorem. Since the first integral rewrites as
    \begin{equation*}
        \int_0^1 \bpair{ H_c[\psi(t,\vdot) \alpha^n_t] }{ \betadot^n_t }_{\Hil_c} \diff t = \bpair{(H_c[\psi(t,\vdot)\alpha^n_t])_t}{(\betadot^n_t)_t}_{L^2([0,1];\Hil_c)},
    \end{equation*}
    convergence to the desired limit holds as a pairing between weakly and strongly convergent sequences in $L^2([0,1];\Hil_c)$ (see Lemma~\ref{lemma:pair_strong_weak}).  
\end{proof}

\subsection{Absolutely continuous curves}

Let us first recall the definition of absolutely continuous curves, in Hilbert spaces they correspond to a.e.~differentiable curves for which the fundamental theorem of calculus holds (Theorem~\ref{theorem:H1_reformulation}).

\begin{definition}[{\cite[Def.~1.1.1]{AGS2008}}] \label{def:AC_curves}
    Let \((\calX,\mathsf{d})\) be a complete metric space and let \(p \in [1,\infty]\). A curve \(\gamma \colon (a,b) \to \calX\) defined on a real interval belongs to \(\AC^p((a,b);\calX)\), if there exists a function \(m \in L^p(a,b)\) such that for all \(a < s \leq t \leq b\) it holds
    \begin{equation*}
        \mathsf{d}(\gamma(s),\gamma(t)) \leq \int_s^t m(\tau) \diff \tau.
    \end{equation*}
\end{definition}

Due to the equivalence between the norm on \(\Hil_c\) and the metric \(\dS\) from  Theorem~\ref{theorem:dS_metric}, and as \(\Hil^1((0,1);\Hil_c) = \AC^2((0,1);\Hil_c)\) from Theorem~\ref{theorem:H1_reformulation}, if $\beta_t = B(\mu_t)$ we have
\begin{equation*}
    (\mu_t)_t \text{ is admissible} \quad \Leftrightarrow \quad (\beta_t)_t \in \AC^2((0,1);\Hil_c) \quad \Leftrightarrow \quad (\mu_t)_t \in \AC^2((0,1);\dS).
\end{equation*}
Building on Lemma~\ref{lemma:mudot_by_betadot_in_RKHS} we can characterize the set of admissible paths in terms of differentiability of the path of measures \((\mu_t)_t\) in duality with each respective RKHS \(\Hil_{\mu_t}\) and regularity of the path of self-transport potentials \(f_{\mu_t,\mu_t}\) in the space of continuous functions as opposed to the less intuitive space \(\Hil_c\).

\begin{proposition} \label{prop:admissible_paths}
    Let \((\mu_t)_t\) be a \textweakstar-continuous path in \(\prm(X)\). Denote \(f_{t,t} = f_{\mu_t,\mu_t}\) and $\beta_t = B(\mu_t)$. Then \((\mu_t)_t\) is an admissible path for our metric, i.e.~\((\beta_t)_t \in \Hil^1((0,1);\Hil_c)\), or equivalently \((\mu_t)_t \in \AC^2((0,1);(\prm(X),\dS))\), if and only if the following conditions are satisfied: 
    \begin{enumerate}
        \item \label{item:mut_strong_diff} \((\mu_t)_t\) is strongly differentiable, i.e.~\(\frac{\mu_t - \mu_s}{t-s} \to \mudot_s\) in norm in \(\Hil_{\mu_s}^*\) at time \(t=s\) for almost all \(s \in [0,1]\).
        \item \label{item:finite_energy} \( \displaystyle{\int_0^1 \g_{\mu_t}(\mudot_t,\mudot_t) \diff t < \infty}\).
        \item \label{item:ftt_AC1} \((f_{t,t})_t \in \AC^1((0,1);\Cont(X))\).
    \end{enumerate}
    Further, when these are satisfied 
    the following integration by parts formula holds for all functions of the form \(\psi_t = \exp(f_{t,t}/\varepsilon) \cdot \phi_t \) with \(\phi \in \Contc^\infty((0,1);\Hil_c)\):
    \begin{equation} \label{eq:continuity_equation}
        \int_0^1 \pair{\mu_t}{\partial_t \psi_t} \diff t = - \int_0^1 \pair{\mudot_t}{\psi_t} \diff t.
    \end{equation}
    In this equation $\psi_t \in \Hil_{\mu_t}$ for every $t$ by Lemma~\ref{prop:RKHS_isometry}, so that the pairing $\pair{\mudot_t}{\psi_t}$ is well-defined. 
\end{proposition}

\begin{proof}
    \underline{Direct implication:} 
    Let \((\beta_t)_t \in \Hil^1((0,1);\Hil_c)\). Then \(\betadot_t\) exists in \(\Hil_c\) for almost all \(t \in (0,1)\). From Lemma~\ref{lemma:mudot_by_betadot_in_RKHS}, \((\mu_t)_t\) is strongly differentiable in \(\Hil_{\mu_s}^*\) at time \(t=s\) for almost all \(s \in (0,1)\). This is condition~\ref{item:mut_strong_diff}. 
    By Proposition~\ref{prop:tgmu_Hcalphadot_equiv}, we have \(\int_0^1 \g_{\mu_t}(\mudot_t,\mudot_t) \diff t = \int_0^1 \tg_{\mu_t}(\betadot_t,\betadot_t) \diff t \leq C \frac{\varepsilon}{2} \norm{(\beta_t)_t}_{\Hil^1((0,1);\Hil_c)}^2 < \infty\), proving condition~\ref{item:finite_energy}. 
    Due to the embedding \(\Hil_c \to \Cont(X)\) and equi-boundedness of the potentials \(f_{t,t}\), we obtain for $s \leq t$ 
    \begin{equation*}
        \norm{f_{t,t} - f_{s,s}}_\infty \leq C \cdot \norm{\beta_t - \beta_s}_\infty \leq C \cdot \norm{\beta_t - \beta_s}_{\Hil_c} \leq C \cdot \int_s^t \norm{\betadot_\tau}_{\Hil_c} \diff \tau
    \end{equation*}
    for some \(C \in (0,\infty)\), where the last inequality is due to Theorem~\ref{theorem:H1_norm_estimate}. As \(( \| \betadot_t \| )_t \in L^2((0,1)) \subset L^1((0,1))\), this proves condition~\ref{item:ftt_AC1}.

    \underline{Converse implication:} 
    We show that \((\beta_t)_t \in \Hil^1((0,1);\Hil_c)\) using the characterization of Theorem~\ref{theorem:H1_reformulation}. Specifically, we show that \((\beta_t)_t\) is strongly differentiable almost everywhere, \((\betadot_t)_t \in L^2((0,1);\Hil_c)\) and that 
    %\begin{equation*}
    \(
        \beta_t - \beta_s = \int_s^t \betadot_\tau \diff \tau
    \)
    %\end{equation*}
    for all \(t,s \in (0,1)\).

    Condition~\ref{item:ftt_AC1} gives \(\norm{(f_{t,t} - f_{s,s})/(t-s)}_\infty \leq \frac{1}{t-s} \int_s^t g(\tau) \diff \tau\) for some \(g \in L^1((0,1);\R)\). By the Lebesgue differentiation theorem~\cite[Thm.~7.10]{Rudin1987}, the right-hand side converges to \(g(s)\) as \(t \to s\) for almost every \(s \in (0,1)\). In particular, the difference quotients are bounded in the limit for almost all \(s\).
    Differentiability of \((\beta_t)_t\) almost everywhere follows from Lemma~\ref{lemma:mudot_by_betadot_in_RKHS} and condition~\ref{item:mut_strong_diff}.
    By Theorem~\ref{theorem:metric_tensor_in_betadot} and Proposition~\ref{prop:tgmu_Hcalphadot_equiv} we have \(\frac{\varepsilon}{2} \norm{\betadot_s}_{\Hil_c}^2 \leq \tg_{\mu_s}(\betadot_s,\betadot_s) = \g_{\mu_s}(\mudot_s,\mudot_s)\) almost everywhere. Thus condition~\ref{item:finite_energy} yields \((\betadot_t)_t \in L^2((0,1);\Hil_c)\). It only remains to check the compatibility condition: for all $t,s$
    \begin{equation} \label{eq:to_prove_b_=_intbetdadot}
        \beta_t - \beta_s = \int_s^t \betadot_\tau \diff \tau
    \end{equation}
    as a Bochner integral in \(\Hil_c\) (see~\cite[Sec.~2.1]{Kreuter2015}). The integral in the right-hand side of~\eqref{eq:to_prove_b_=_intbetdadot} exists in $\Hil_c$, see Theorem~\ref{theorem:H1_norm_estimate}. To identify the left-hand side, as \((f_{t,t})_t \in \AC^1((0,1);\Cont(X))\) and the potentials are equi-bounded, we obtain \((\beta_t)_t \in \AC^1((0,1);\Cont(X))\). In particular for any $x \in X$ the curve $(\beta_t(x))_t$ belongs to $\AC^1((0,1);\R)$ and thus~\eqref{eq:to_prove_b_=_intbetdadot}, evaluated at $x \in X$, holds as a classical integral. We conclude to~\eqref{eq:to_prove_b_=_intbetdadot} as Bochner integral in $\Hil_c$ as the evaluation maps are continuous in $\Hil_c$.

    \underline{Integration by parts:} 
    Let \((\beta_t)_t \in \Hil^1((0,1);\Hil_c)\). Let \((\psi_t)_t\) and \((\phi_t)_t\) be as in~\eqref{eq:continuity_equation}. By our assumptions on \((\beta_t)_t\), the rescaled function \(\psi_t = \phi_t / \beta_t\) is differentiable in \(\Cont(X)\) for almost every \(t\). Lemma~\ref{lemma:mudot_by_betadot_in_RKHS} yields
    \begin{align*}
        \int_0^1 \pair{\mudot_t}{\psi_t} \diff t
        &= \int_0^1 \pair{\betadot_t}{\beta_t \cdot \psi_t + \beta_t \cdot K_{\mu_t}[\psi_t]}_{\Hil_c} \diff t 
        = \int_0^1 \pair{\betadot_t}{ \phi_t + H_c[\phi_t / \beta_t^2 \cdot \mu_t]}_{\Hil_c} \diff t .
    \end{align*}
    By Lemma~\ref{lemma:H1_weakdiff_H} we know that 
    \(
        \int_0^1 \pair{\betadot_t}{\phi_t}_{\Hil_c} \diff t = - \int_0^1 \pair{\beta_t}{\partial_t \phi_t}_{\Hil_c} \diff t.     
    \)
    Further, \(\pair{\beta_t}{\partial_t \phi_t}_{\Hil_c} = \pair{\alpha_t}{\partial_t \phi_t} = \pair{\mu_t}{(\partial_t \phi_t) / \beta_t}\).
    We compute the remaining part using 
    \(
        \pair{\betadot_t}{H_c[\phi_t/\beta_t^2 \cdot \mu_t]}_{\Hil_c} 
        = \pair{\mu_t}{\phi_t \cdot \betadot_t / \beta_t^2}.
    \)
    Together this gives
    \begin{align*}
        \int_0^1 \pair{\mudot_t}{\psi_t} 
        &= - \int_0^1 \pair{\mu_t}{(\partial_t \phi_t) / \beta_t - \phi_t \cdot \betadot_t / \beta_t^2} \diff t
        = - \int_0^1 \pair{\mu_t}{\partial_t \psi_t} \diff t. \qedhere
    \end{align*}
\end{proof}

The previous proposition gives a full characterization of the set of admissible paths and requires a light notion of differentiability of the path \((\mu_t)_t\) together with a robust regularity assumption on the path of self-transport potentials \((f_{t,t})_t\) in \(\Cont(X)\).

In the differentiable setting, Proposition~\ref{prop:Cm_bound_dS} below gives a simpler way to show that a path is admissible. In it we utilize a stronger regularity assumption on \((\mu_t)_t\) by requiring \(\mudot_t \in \Cont^m(X)^*\) and avoid the extra assumption on the potentials.

\begin{proposition} \label{prop:Cm_bound_dS}
    For $m \geq 0$, in the setting of Assumption~\ref{asp:diff_m}, there exists $C > 0$ depending on $c,\varepsilon$ and $m$ such that for all $\mu_0, \mu_1 \in \prm(X)$,
    \begin{equation} \label{eq:dS_bound_Cm}
        \dS(\mu_0,\mu_1) \leq C \cdot \norm{\mu_1 - \mu_0}_{\Cont^{m,*}}.
    \end{equation}
    In particular, \((\mu_t)_t \in \AC^2((0,1); (\Cont^m(X)^*, \| \cdot \|_{\Cont^{m,*}} ))\) yields \((\mu_t)_t \in \AC^2((0,1);(\prm(X),\dS))\) and any such path is admissible for \(\dS\).
\end{proposition}

\begin{proof}
We clearly only need to prove~\eqref{eq:dS_bound_Cm}.
The vertical interpolation \(\mu_t = (1-t) \mu_0 + t \mu_1\) is an admissible path with \(\mudot_t = \mu_1 - \mu_0\). We have, using successively the estimates~\eqref{eq:control_metric_tensor_Hmu} and~\eqref{eq:comp_Hmu_Cm},
\begin{equation*}
    \g_{\mu_t}(\mudot_t,\mudot_t) \leq C_1 \norm{\mu_1-\mu_0}^2_{\Hil_{\mu,0}^*} \leq C_2 \norm{\mu_1-\mu_0}^2_{\Cont^{m,*}}, 
\end{equation*}
for some \(C_1, C_2 \geq 0\), where the constant $C_2$ may depend on $m$ but neither one depends on $\mu$. Integrating this between $0$ and $1$ yields the estimate as the left-hand side becomes an upper bound for $ \dS(\mu_0,\mu_1)^2$ by definition. 
\end{proof}

We can now show that paths which essentially follow a \(\Cont^m\)-perturbation for time \([0,1]\) are admissible paths for \(\dS\). For these paths the metric tensor is explicitly given by the Hessian of the Sinkhorn divergence. This provides the link between Section~\ref{section:sinkhorn_hessian} and Section~\ref{section:sub:path_metric_intro}.

\begin{corollary}
\label{cor:Cm_admissible}
    In the setting of Assumption~\ref{asp:diff_m} for some \(m \geq 0\), let \((\mu_t)_t\) be a path that is continuously differentiable from \((0,1)\) into the space \(\Cont^m(X)^*\) endowed with the \textweakstar~topology, such that the derivative extends \textweakstar~continuously to the end points. 
    Then \(\sup_{t \in [0,1]} \, \norm{\mudot_t}_{\Cont^{m,*}} < \infty\), and we have \((\mu_t)_t \in \AC^2((0,1);(\Cont^m(X)^*,\norm{\vdot}_{\Cont^{m,*}})\). In particular, \((\mu_t)_t\) is an admissible path for \(\dS\).
\end{corollary}

\begin{proof}
    There holds \(\sup_{t \in [0,1]} \, \norm{\mudot_t}_{\Cont^{m,*}} < \infty \), as \textweakstar-continuity gives boundedness of $\norm{\mudot_{t_n}}_{\Cont^{m,*}}$ along any convergent sequence \(t_n \to t\) and \([0,1]\) is compact. Then by \cite[Prop.~3.8,~(iv)$\Rightarrow$(ii)]{Kreuter2015} with the choice of separating subset \(\Cont^m(X) \subseteq (\Cont^m(X)^*)^*\), we obtain that \((\mu_t)_t\) is of class \(\AC^2((0,1);(\Cont^m(X)^*,\norm{\vdot}_{\Cont^{m,*}}))\) if we show that for any \(\phi \in \Cont^m(X)\), the real valued function \(t \mapsto \pair{\mu_t}{\phi}\) is continuously differentiable with derivative \(\pair{\mudot_t}{\phi}\) and \((\mudot_t)_t \in L^2((0,1);\Cont^m(X)^*)\). All of these are already established.
\end{proof}

\begin{remark}[{Comparison between $\dS$, Wasserstein, and Hellinger--Kantorovich}] \label{remark:comparison_W1_HK}
    In the setting of Assumption~\ref{asp:diff_m} with \(m=1\), Proposition~\ref{prop:Cm_bound_dS} yields a \(C > 0\) such that
    \begin{equation} \label{eq:W1_bound_dS}
        \dS(\mu_0,\mu_1) \leq C \cdot \norm{\mu_1 - \mu_0}_{\Cont^{1,*}} \leq C \cdot W_1(\mu_0,\mu_1)
    \end{equation}
    for any \(\mu_0, \mu_1 \in \prm(X)\). Here $W_p$ denotes the $p$-Wasserstein distance \cite[Chapter 5]{Santambrogio2015} for $p \geq 1$. 
    Thus 2-absolutely continuous paths in \(W_1\) (hence \(W_p\) for any \(p \in [1,\infty)\)) are admissible paths for our metric. 
    The converse does not hold as shown in Example~\ref{example:quadratic_split} below.
    
    In the Euclidean setting the formal Riemannian tensor of the Hellinger--Kantorovich distance \cite{KMV-OTFisherRao-2015,Chizat:UOT,Liero2018} is given by
    \begin{equation}
        \label{eq:gHK}
        \g^{\tn{HK}}_\mu(\mudot, \mudot) = \int_X \left( |\nabla \psi(x)|^2 + \tfrac14 |\zeta(x)|^2 \right) \diff \mu (x), \qquad \mudot = - \ddiv (\mu \nabla \psi) + \mu \cdot \zeta
    \end{equation}
    where $\mudot$ is explicitly decomposed into a horizontal and a vertical component.
    This induces a Riemannian distance on non-negative Radon measures with variable mass. The geodesic restriction to probability measures, the so-called spherical Hellinger--Kantorovich distance \(\mathsf{SHK}\), is studied in \cite{LaMi2017}.
    In particular, $\g^{\tn{HK}}_\mu(\mudot, \mudot)$ is finite, even when $\mudot$ contains a vertical component, which enables the ``teleportation'' of mass from one point to another without moving it through the base space.
    This is similar to the metric tensor $\g_\mu(\mudot, \mudot)$ of $\dS$ and different from the tensor for the standard Wasserstein distance, \eqref{eq:intro_gmu_OT}. The same \(C\) as in~\eqref{eq:W1_bound_dS} gives \(\g_\mu(\mudot,\mudot) \leq C^2 \cdot \norm{\mudot}_{\Cont^{1,*}}^2 \leq 16 C^2 \cdot (\norm{\nabla \phi}_{L^2(X,\mu)} + \frac{1}{4} \norm{\zeta}_{L^2(X,\mu)})^2 \leq 32 C^2 \cdot \g_\mu^{\tn{HK}}(\mudot,\mudot)\) for all perturbations \(\mudot\) as in~\eqref{eq:gHK}. Hence we have the generalization of~\eqref{eq:W1_bound_dS}: 
    \begin{equation*} 
    %\label{eq:SHK_bound_dS}
        \dS(\mu_0,\mu_1) \leq 4 \sqrt{2} C \cdot \mathsf{SHK}(\mu_0,\mu_1).
    \end{equation*}
    However, beyond this similarity, there are several significant differences: $\g^{\tn{HK}}_\mu$ is spatially local in the functions $\nabla \psi$ and $\zeta$, which parametrize $\mudot$, whereas $\g_\mu$ from Definition~\ref{def:metric_tensor_measure} is non-local due to the operators $K_\mu$ and $H_\mu$ and does not rely on an explicit decomposition of $\mudot$ into vertical and horizontal terms. For $\g^{\tn{HK}}_\mu$ the cost of teleportation is independent of the distance between the start and end point in base space. In fact, when the supports of $\mu_0$ and $\mu_1$ are separated by more than $\pi/2$, their geodesic with respect to $\g^{\tn{HK}}_\mu$ solely relies on the vertical teleportation term and becomes equal to that of the Hellinger distance. In contrast, for $c(x,y)=\|x-y\|^2$ we find below (Section \ref{section:sub:vertical_two_points}) that in $\g_\mu$ teleportation can become exponentially expensive at large distances whereas at small distances horizontal and vertical perturbations have comparable cost. 
\end{remark}

\begin{example} \label{example:quadratic_split}
    The path given by \(\mu_t = \frac{1}{2} (\delta_{-\sqrt{t}} + \delta_{\sqrt{t}})\) on \(X = [-1,1] \subseteq \R\) is continuously differentiable in \(\Cont^2(X)^*\) endowed with the norm topology for \(t \in (0,1]\) and \textweakstar-differentiable in \(\Cont^2(X)^*\) at \(t=0\). Indeed, we have \(\mudot_t = \frac{1}{2\sqrt{t}} ( \delta'_{\sqrt{t}} - \delta'_{-\sqrt{t}})\) thus the fundamental theorem of calculus yields \(\norm{\mudot_t}_{\Cont^{2,*}} \leq 1\) for any \(t \in (0,1]\). At \(t=0\), we can actually compute $\pair{\mudot_0}{\phi} = \phi''(0)$. 
    Proposition~\ref{prop:Cm_bound_dS} yields that $(\mu_t)_t$ is an admissible path for $\mathsf{d}_S$ when $c \in \Cont^{(2,2)}(X \times X)$, such as for the quadratic cost.  
    
    On the other hand as $\mudot_t = - \ddiv(\mu_t v_t)$ with $|v_t(x)| = 1 / \sqrt{t}$, we have  $\g^0_{\mu_t}(\mudot_t,\mudot_t) = \g^{\tn{HK}}_{\mu_t}(\mudot_t,\mudot_t) = 1/t$. It shows that $(\mu_t)_t$ has infinite action on $[0,1]$ both for the classical optimal transport geometry and the Hellinger--Kantorovich geometry. 
\end{example}

\subsection{Mean estimate for the squared Euclidean distance} \label{section:sub:mean_estimate}

In this section $c(x,y) = \norm{x-y}^2$ is the squared Euclidean distance and we work with measures in \(\R^d\) with bounded support. For such a measure $\mu$, we write $m = \pair{\mu}{\id} \in \R^d$ for its mean, and $\bar{\mu} = (\id - m)_\# \mu$ for the centered measure. As recalled in the introduction, it is well understood that for classical optimal transport the cost $\OT_0$ between $\mu_0$ and $\mu_1$ decomposes as
\begin{equation*}
    \OT_0(\mu_0,\mu_1) = \norm{m_1 - m_0}^2 + \OT_0(\bar{\mu}_0,\bar{\mu}_1).
\end{equation*}
The same proof carries through for the Sinkhorn divergence (Proposition~\ref{prop:Seps_bound_mean}). We then extend this orthogonal decomposition result to the metric tensor in Proposition~\ref{prop:gmu_mean} and to the distance $\dS$ in Theorem~\ref{theorem:Pythagoras_distance}, which in particular implies that translations \(\mu_t = (\id + t \cdot v)_\# \mu_0\) for $v \in \R^d$ are geodesics for $\dS$.

\begin{proposition} \label{prop:Seps_bound_mean}
    On \(\R^d\) let \(c(x,y) = \norm{x-y}^2\). Let \(\mu_0,\mu_1 \in \prm(\R^d)\) with bounded support and means \(m_0, m_1 \in \R^d\). Their centered versions $\bar{\mu}_0, \bar{\mu}_1$ satisfy 
    \begin{equation*}
        S_\varepsilon(\mu_0,\mu_1) = \norm{m_1 - m_0}^2 + S_\varepsilon(\bar{\mu}_0,\bar{\mu}_1).
    \end{equation*}
\end{proposition}

\begin{proof}
    The proof is basically the same as for plain optimal transport, where $\varepsilon = 0$. For the squared Euclidean distance, optimal couplings for \(\OT_\varepsilon\) remain optimal under translations in the following sense: for all \(\pi \in \Pi(\mu_0,\mu_1)\) it holds \((\id-m_0,\id - m_1)_\# \pi \in \Pi(\bar{\mu}_0, \bar{\mu}_1)\) and 
    \begin{align*}
        \pair{(\id - m_0,\id - m_1)_\# \pi}{c} 
        &= \int_{X \times X} \norm{x - m_0 - (y - m_1)}^2 \diff \pi(x,y) \\
        &= \int_{X \times X} \norm{x - y}^2 - 2 \pair{x - y}{m_0 - m_1} + \norm{m_1 - m_0}^2 \diff \pi(x,y) \\ 
        &= \pair{\pi}{c} -  \norm{m_1 - m_0}^2.
    \end{align*}
    The \(\KL\)-divergence is invariant under such shifts, i.e
    \begin{equation*}
        \KL((\id - m_0, \id - m_1)_\# \pi \, | \, (\id - m_0)_\# \mu_0 \otimes (\id - m_1)_\# \mu_1) = \KL(\pi \, | \, \mu_0 \otimes \mu_1).
    \end{equation*}
    Therefore, minimizing over \(\pi\) yields $\OT_\varepsilon(\mu_0, \mu_1) =   \norm{m_1 - m_0}^2 + \OT_\varepsilon(\bar{\mu}_0,\bar{\mu}_1)$. Analogously one sees that \(\OT_\varepsilon( \mu_i, \mu_i) = \OT_\varepsilon(\bar{\mu}_i,\bar{\mu}_i)\) for $i=0,1$. From the definition of \(S_\varepsilon\) in~\eqref{eq:intro:def_sinkhorn_div} we obtain the result. 
\end{proof}

We now extend this decomposition to the level of the metric tensor. 
If $v \in \R^d$, the uniform translation of a measure $\mu$ by $v$ corresponds to the derivative $\mudot = - \ddiv(v \mu) \in \Hil_{\mu,0}^*$: this is the continuity equation. For \(\mu \in \prm(X)\) let \(\calS_\mu \assign \setgiven{-\ddiv(v\mu)}{v \in \R^d} \subseteq \Hil_{\mu,0}^*\) be the set of constant shift perturbations of \(\mu\). Denote the orthogonal projection onto \(\calS_\mu\) with respect to the inner product \(\g_\mu\) by 
\begin{equation*}
    P_{\calS_\mu} \colon \Hil_{\mu,0}^* \to \calS_\mu. 
\end{equation*}

\begin{proposition} \label{prop:gmu_mean}
    Let \(X \subseteq \R^d\) be compact, \(c(x,y) = \norm{x-y}^2\). 
    Let \(\mu \in \prm(X)\) and let \(\mudot \in \Hil_{\mu,0}^*\).
    Then 
    \begin{enumerate}
        \item the mean \(\dot{m} = \pair{\mudot}{\id}\) exists in \(\R^d\) and is characterized by \(P_{\calS_{\mu}}(\mudot) = - \ddiv(\dot{m} \mu)\),
        \item there holds $\g_\mu(\mudot,\mudot) = \norm{\dot{m}}^2 + \g_\mu(\mudot + \ddiv(\dot{m} \mu), \mudot + \ddiv(\dot{m} \mu))$.
    \end{enumerate}
\end{proposition}

\begin{proof}
    For \(v \in \R^d\) denote \(\mudot^v = - \ddiv(v\mu)\). As a distribution of order one it belongs to $\Hil_{\mu}^*$ (see Remark~\ref{rmk:comp_Hmu_Cm}). 
    We compute
    \begin{align}
    \begin{split} \label{eq:proof:gmu_mean:Hmu}
        H_\mu [\mudot^v] (x) & = \int_X \pair{v}{\nabla_y k_\mu(x,y)} \diff \mu(y) \\
        & = \frac{1}{\varepsilon} \int_X \pair{v}{\nabla \fmumu(y) - 2y + 2x} k_\mu(x,y) \diff \mu(y) \\
        & = \frac{1}{\varepsilon} K_\mu[\pair{v}{\nabla \fmumu - 2 \id}](x) + \frac{2}{\varepsilon} \pair{v}{x}.
    \end{split}
    \end{align}
    Differentiating the relation \(\fmumu(y) = T_\varepsilon(\fmumu,\mu)(y)\) in \(y\) yields 
    \begin{equation} \label{eq:gradient_fmumu}
        2y - \nabla \fmumu(y) = \int_X k_\mu(y,z) 2 z \diff \mu(z) = 2 K_\mu[\id](y). 
    \end{equation}
    Thus $\pair{v}{\nabla \fmumu - 2 \id} = - 2 K_\mu[\pair{v}{\id}]$. Applying this in~\eqref{eq:proof:gmu_mean:Hmu}, we deduce 
    the identity $\frac{\varepsilon}{2} (\id - K_\mu^2)\inv H_\mu[\mudot^v] = \pair{v}{\id}$ in \(\Hil_{\mu} / \R\).
    In particular, the function $\pair{v}{\id}$ belongs to $\Hil_{\mu} / \R$ for any $v \in \R^d$. Since every coordinate function $x \mapsto x_i$ belongs to $\Hil_\mu / \R$, the mean $\dot{m} = \pair{\mudot}{\id}$ is well defined as an element of $\R^d$. 
    
    Given the expression of the metric tensor (Definition~\ref{def:metric_tensor_measure}), we see that 
    \begin{equation*}
        \g_\mu(\mudot,\mudot^v) = \pair{\mudot}{\pair{v}{\id}} = \pair{v}{\dot{m}} = \g_\mu(\mudot^{\dot{m}},\mudot^v)
    \end{equation*}
    Therefore \(\mudot^{\dot{m}}\) is the orthogonal projection of \(\mudot\) to $\calS_\mu$ with respect to the metric tensor \(\g_\mu\). The second point of the proposition is a direct consequence of the Pythagorean identity for the inner product $\g_\mu$. 
\end{proof}

\begin{remark}
    For $\Cont^m$-perturbations (Definition~\ref{def:Cm_perturbations}) for which the metric tensor is linked to the Sinkhorn divergence via Theorem~\ref{theorem:hessian_sinkhorn}, the proof of Proposition~\ref{prop:gmu_mean} can be done directly as an application of Proposition~\ref{prop:Seps_bound_mean}.    
\end{remark}

The orthogonal decomposition at the level of the metric tensor from Proposition~\ref{prop:gmu_mean} can be integrated and extended to the level of the distance \(\dS\).

\begin{theorem}
\label{theorem:Pythagoras_distance}
    On \(\R^d\) let \(c(x,y) = \norm{x-y}^2\). Let \(\mu_0,\mu_1 \in \prm(\R^d)\) with bounded support and means \(m_0, m_1 \in \R^d\) and centered measures $\bar{\mu}_0$, $\bar{\mu}_1$. Then 
    \begin{equation*}
        \dS(\mu_0,\mu_1)^2 = \norm{m_1 - m_0}^2 + \dS(\bar{\mu}_0,\bar{\mu}_1)^2.
    \end{equation*}
    In particular $\dS(\mu_0,\mu_1) \geq \norm{m_1 - m_0}$ with equality if and only if \(\mu_1\) is a translation of \(\mu_0\). In this case \(\mu_t = (\id + t(m_1 - m_0))_\# \mu_0\) is the unique geodesic between them.
\end{theorem}

\noindent Before diving into the proof of the theorem, we need an additional technical lemma. Whereas Proposition~\ref{prop:gmu_mean} is concerned with what is happening at the level of the metric tensor, we need to check additionally that the map $\mu \mapsto \bar{\mu} = (\id - m)_\# \mu$ with $m = \pair{\mu}{\id}$ preserves the class of admissible curves.

\begin{lemma} \label{lemma:mt_barmu_t_smooth}
    Let $(\mu_t)_t \in \Pall$ with support in some bounded set and write $m_t$ and $\bar{\mu}_t$ for the mean and centered part of $\mu_t$. Then $(\bar{\mu}_t) \in \Pall$ (with support in some possibly larger bounded set) and $(m_t) \in \Hil^1((0,1);\R^d)$.
\end{lemma}

\begin{proof}
    We start with a preliminary remark. We proved in Proposition~\ref{prop:gmu_mean} that for any $v \in \R^d$ and any $\mu$ we have $\pair{v}{\id} = \frac{\varepsilon}{2} (\id - K_{\mu}^2)\inv H_{\mu}[- \ddiv(\mu v)]$.
    Thus with the help of Theorem~\ref{theorem:inverse_operators_Hmu} and the estimate~\eqref{eq:comp_Hmu_Cm} we have $\pair{v}{\id} \in \Hil_\mu$ and 
    \begin{equation} \label{eq:aux_lemma_H1}
        \| \pair{v}{\id} \|_{\Hil_\mu} \leq  C_1\| \ddiv(\mu v) \|_{\Hil_\mu^*}  \leq C_2 \| \ddiv(\mu v) \|_{\Cont^{1,*}}  \leq C_2 \norm{v}
    \end{equation}
    for constants $C_1,C_2$ which only depends on the radius of the support of $\mu \in \prm(X)$.

    We first prove that $m_t \in \Hil^1([0,1],\R^d)$. It is enough to prove that $\pair{v}{m_t} \in \Hil^1([0,1],\R)$ for any \(v \in \R^d\). Recall from reformulating~\eqref{eq:proof:mudot_by_betadot:mu_difference_quotient} that for any \(\phi \in \Hil_{\mu_s}\) and any $t,s$,
    \begin{equation*}
        \pair{\mu_{t} - \mu_s}{\phi} = \bpair{\beta_t - \beta_s}{\ef_{\mu_s}\inv \phi}_{\Hil_c} + \bpair{\beta_t - \beta_s}{H_c[\alpha_t \phi]}_{\Hil_c}.    
    \end{equation*}
    For \(\phi = \pair{v}{\id}\), with the help of the Cauchy--Schwarz inequality and using that multiplication by $\ef_{\mu_s}\inv$ is an isometry from $\Hil_{\mu_s}$ to \(\Hil_c\), this gives 
    \begin{equation*}
        |\pair{v}{m_t - m_s}| \leq \norm{\beta_t - \beta_s}_{\Hil_c} \left( \norm{ \pair{v}{\id} }_{\Hil_{\mu_s}} + \norm{H_c[\alpha_t \pair{v}{\id}]}_{\Hil_c} \right).
    \end{equation*}
    Given that $(\beta_t)_t \in \Hil([0,1]; \Hil_c)$, we need only to prove that $\norm{ \pair{v}{\id} }_{\Hil_{\mu_s}} + \norm{H_c[\alpha_t \phi]}_{\Hil_c}$ can be bounded independently on $t,s$ to conclude with Theorem~\ref{theorem:H1_reformulation}. This is a consequence of~\eqref{eq:aux_lemma_H1} for the first term while for the second term it is enough to notice that $\alpha_t$ is bounded in total variation uniformly in $t$.
    
    We now proceed to prove that $\bar{\mu}_t \in \Pall$. By uniqueness of the Schrödinger potentials we see that $f_{\bar{\mu}_t, \bar{\mu}_t}(x) = f_{\mu_t, \mu_t}(x + m_t)$. Writing $\bar{\beta}_t = B(\bar{\mu}_t)$, this gives $\bar{\beta}_t(x) = \beta_t(x + m_t)$. Thus for any $t,s$, 
    \begin{align*}
        \norm{\bar{\beta}_t - \bar{\beta}_s}_{\Hil_c} & = \norm{\beta_t(\vdot + m_t) - \beta_s(\vdot + m_s)}_{\Hil_c} \\
        & \leq \norm{\beta_t(\vdot + m_t) - \beta_s(\vdot + m_t)}_{\Hil_c} + \norm{\beta_s(\vdot + m_t) - \beta_s(\vdot + m_s)}_{\Hil_c} .
    \end{align*}
    Using that the norm on $\Hil_c$ is invariant by translation because so is the kernel, the first term coincides with $\norm{\beta_t - \beta_s}_{\Hil_c}$.
    For the second term, we can write 
    \begin{equation*}
        \beta_s(x - m_t) - \beta_s(x - m_s) = \int_{s}^t \dot{m}_r \cdot \nabla \beta_s(x - m_r) \diff r,
    \end{equation*}
    and so we only need to check that $\nabla \beta_s(\cdot - m_r) \in \Hil_c$ with a controlled norm to justify that the equality above holds as equality of functions in $\Hil_c$. Recalling $\beta = \exp(- \fmumu / \varepsilon)$ and recalling that we have already computed $\nabla \fmumu$ in~\eqref{eq:gradient_fmumu}, we find   
    \begin{equation*}
        \dot{m}_r \cdot \nabla \beta_s = \frac{2}{\varepsilon} \left( K_{\mu_s}[\pair{\dot{m}_r}{\id}] -  \pair{\dot{m}_r}{\id}  \right) \beta_s.
    \end{equation*}
    Using invariance by translation of the norm in $\Hil_c$, and that multiplication by $\beta_s = \ef_{\mu_s}\inv$ is an isometry between $\Hil_{\mu_s}$ and $\Hil_c$,
    \begin{equation*}
        \norm{\dot{m}_r \cdot \nabla \beta_s(\vdot - m_s)}_{\Hil_c} = \norm{\dot{m}_r \cdot \nabla \beta_s}_{\Hil_c} \leq \frac{2}{\varepsilon} \left( \norm{ K_{\mu_s}[\pair{\dot{m}_r}{\id}]}_{\Hil_{\mu_s}} + \norm{\pair{\dot{m}_r}{\id}}_{\Hil_{\mu_s}} \right).
    \end{equation*}
    We use that $K_{\mu_s}$ is a contraction (see Proposition~\ref{prop:Hmu_compact_Hmu}), and again~\eqref{eq:aux_lemma_H1} to conclude that the right-hand side can be bounded by $C \cdot \norm{\dot{m}_r}$ for some \(C > 0\) independently on $s$. Thus
    \begin{equation*}
        \norm{\bar{\beta}_t - \bar{\beta}_s}_{\Hil_c} \leq  \norm{\beta_t - \beta_s}_{\Hil_c} +  C \int_t^s \norm{\dot{m}_r} \diff r.   
    \end{equation*}
    From Theorem~\ref{theorem:H1_reformulation} we conclude that $(\bar{\beta}_t)_t \in \Hil^1([0,1];\Hil_c)$ so that $(\bar{\mu}_t)_t \in \Pall$.
\end{proof}

\begin{proof}[\textbf{Proof of Theorem~\ref{theorem:Pythagoras_distance}}]
    If \((\mu_t)_t \in \Pall(\mu_0, \mu_1)\), then by Lemma~\ref{lemma:mt_barmu_t_smooth} $(B(\bar{\mu}_t))_t \in \Pall$ and $(m_t)_t$ exists in $\Hil^1((0,1);\R^d)$.
    Writing $\dot{\bar{\mu}}_t$ for the derivative of the curve $t \mapsto \bar{\mu}_t$ and $\bar{\dot{\mu}}_t$ for the translation of the perturbation $\dot{\mu}_t$ by $-m_t$, it easy to check that $\dot{\bar{\mu}}_t = \bar{\mudot}_t + \ddiv(\dot{m}_t \bar{\mu}_t)$ holds in $\Hil^*_{\mu_t,0}$ for any $t$ when all three derivatives exists. In particular, translation invariance of the metric tensor yields 
    \begin{equation*}
    \g_{\mu_t}(\mudot + \ddiv(\dot{m} \mu), \mudot_t + \ddiv(\dot{m}_t \mu_t)) = \g_{\bar{\mu}_t}(\bar{\mudot}_t + \ddiv(\dot{m}_t \bar{\mu}_t), \bar{\mudot}_t + \ddiv(\dot{m}_t \bar{\mu}_t)) =  \g_{\bar{\mu}_t}(\dot{\bar{\mu}}_t, \dot{\bar{\mu}}_t).
    \end{equation*}
    From Proposition~\ref{prop:gmu_mean} we deduce that 
    \begin{align*}
        \int_0^1 \g_{\mu_t}(\mudot_t,\mudot_t) \diff t 
        & = \int_0^1 \norm{\dot{m}_t}^2 \diff t + \int_0^1 \g_{\mu_t}(\mudot + \ddiv(\dot{m} \mu), \mudot_t + \ddiv(\dot{m}_t \mu_t)) \diff t \\
        & =  \int_0^1 \norm{\dot{m}_t}^2 \diff t + \int_0^1 \g_{\bar{\mu}_t}(\dot{\bar{\mu}}_t, \dot{\bar{\mu}}_t) \diff t.
    \end{align*}
    Conversely, from a curve of centered measures $(\bar{\mu}_t)_t$ and a path $(m_t)_t \in \Hil^1([0,1],\R^d)$ of means we can reconstruct $\mu_t = (\id + m_t)_\# \bar{\mu}_t$. We get the decomposition of the distance by minimizing over $(\mu_t)_t$. 
    
    The inequality follows, and if we have equality, then necessarily $(\bar{\mu}_t)_t$ is a constant curve and $(m_t)_t$ is a geodesic in $\R^d$. 
\end{proof}

\begin{corollary}[Energy of a travelling Dirac measure]\label{cor:example:traveling_dirac}
Let \(c(x,y) = \norm{x - y}^2\). 
Let \((x_t)_{t} \in \Hil^1((0,1);\R^d)\) and set \(\mu_t = \delta_{x_t}\). 
Then 
\begin{equation*}
    E((\mu_t)_t) = \int_0^1 \norm{\dot{x}_t}^2 \diff t.
\end{equation*}
In particular, 
\(\dS(\delta_{x_0},\delta_{x_1}) = \norm{x_0 - x_1}\) agrees with the metric on \(\R^d\) and the unique geodesic is \((\delta_{y_t})_{t \in [0,1]}\) with \(y_t = (1-t) x_0 + t x_1\). 
\end{corollary}

\section{Examples in the metric\texorpdfstring{ $\dS$}{}} \label{section:examples}

\subsection{Explicit formula for Gaussian measures} \label{section:sub:gaussians}

\begin{figure}[!ht]
\begin{center}

\includegraphics[scale=0.92]{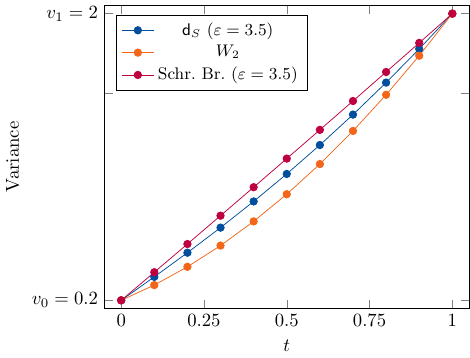}
\includegraphics[scale=0.9]{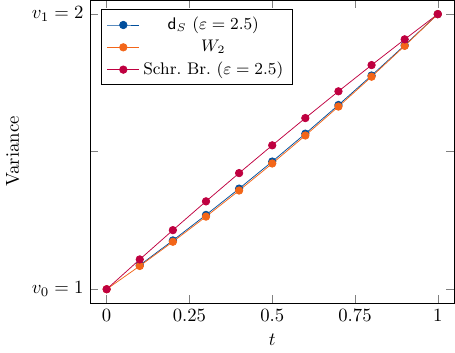}

\caption{\label{fig: Gauss}Representation of the Wasserstein and the (Gaussian-restricted) Sinkhorn-based geodesics for one-dimensional centered Gaussian distributions parametrized by their variance $v_t$,  $t \in [0,1]$, for different choices of $v_0$, $v_1$ and $\varepsilon$. For comparison we further show the $\varepsilon$-Schr\"odinger bridge.}
\end{center}
\end{figure}

Throughout this section we fix \(c(x,y) = \abs{x-y}^2\) as the squared Euclidean distance on \(\R\). We derive the metric tensor as well as an explicit formula for the geodesic restriction of the distance $\dS$ to the family of Gaussian measures in dimension one. By virtue of Theorem~\ref{theorem:Pythagoras_distance} we can restrict ourselves to centered measures. Strictly speaking, our theory a priori only applies to compactly supported measures, but we can still evaluate the metric tensor on Gaussians. We parametrize Gaussian measures by their variance $v$ rather than their standard deviation $\sigma = \sqrt{v}$. Writing $\mathcal{N}(0,v)$ for the centered Gaussian measure with variance $v$, we recall from~\cite{Gelbrich1990,Takatsu2011} that
\begin{equation*}
    \OT_0(\mathcal{N}(0,v_0),\mathcal{N}(0,v_1)) = \left( \sqrt{v_1} - \sqrt{v_0} \right)^2,
\end{equation*}
which corresponds to the unregularized metric tensor $\g^0_{\mathcal{N}(0,v)} (\dot{v},\dot{v}) = \frac{1}{4v} \dot{v}^2$.
As discussed in the introduction, page \pageref{paragraph:difference_bridges}, geodesics with respect to $\dS$ differ from a $\varepsilon$-Schr\"odinger bridge, for which an explicit formula for Gaussian measures can be found in~\cite{Bunne2022}: the result is displayed in Figure~\ref{fig: Gauss}.
With our metric we obtain the following:

\begin{theorem} \label{theorem:gaussians_gmu}
    Consider a path $(v_t)_t$ in \((0,\infty)\) that is differentiable around \(t=0\) and denote by  $\dot v_t$ its derivative at $t$. Identifying the Gaussian distribution $\mathcal{N}(0,v_t)$ with its variance \(v_t\) and writing \(v = v_0\), \(\dot{v} = \dot{v}_0\), the metric tensor reads
    \begin{equation}
    \label{eq:met-tens-gaus}
        \g_{v}(\dot{v}, \dot{v}) = g(v) \dot{v}^2 \quad \text{with} \quad g(v) = \frac{1}{4\sqrt{ \frac{\varepsilon^2}{16} + v^2 }}.
    \end{equation}
\end{theorem}

\noindent
Clearly as $\varepsilon \to 0$ the metric tensor $\g_{v}(\dot{v}, \dot{v})$ converges to $\g^0_{\mathcal{N}(0,v)} (\dot{v},\dot{v})$. For $\varepsilon \to \infty$ it degenerates and no longer penalizes changes in the variance: this is consistent with the claims of Section~\ref{sec:sub:epsilon_limits}.

\begin{proof}
    We identify the Gaussian \(\mu_t=\mathcal{N}(0,v_t)\) with its variance \(v_t\). As shown in~\cite[Cor.~1]{Mallasto2021}, the entropic optimal transport cost between 1-dimensional Gaussian distributions with variance $v_t, v_s$ is explicitly given by
    \begin{equation}\label{eq:gaussian-OTeps}
        \OT_\varepsilon(v_t,v_s) = v_t + v_s - \frac{\varepsilon}{2} \left( \kappa_{ts} - \log \kappa_{ts} + \log 2 - 2 \right), \qquad \kappa_{ts} \assign 1 + \sqrt{1+ \frac{16}{\varepsilon^2} v_t v_s}.
    \end{equation}
    Let \(v_0 = v\), \(\dot{v}_0 = \dot{v}\).
    Let us introduce $f(x) = 1+\sqrt{1+ 16 \varepsilon^{-2}x}$, so that $\kappa_{ts} = f(v_t v_s)$. We have
    \[
        \partial_t \kappa_{ts} = \dot{v}_t v_s f'(v_t v_s), \qquad \partial_s\partial_t \kappa_{ts} = \dot{v}_s \dot{v}_t  f'(v_t v_s) + \dot{v}_s \dot{v}_t v_s v_t f''(v_t v_s).
    \]
    As in the proof of Theorem~\ref{theorem:hessian_sinkhorn}, we now compute $\g_v(\dot{v},\dot{v}) = - \frac{1}{2} \partial_t|_{t=0} \partial_s|_{s=0} \OT_\varepsilon(v_t,v_s)$. To this end, observe that
    \begin{equation*} %\label{eq:proof:gaussian_gmu:del_OT}
        \partial_t \partial_s \OT_\varepsilon(v_t,v_s) = \frac{\varepsilon}{2} \left( - \partial_t \partial_s \kappa_{ts} + \frac{\partial_t \partial_s \kappa_{ts}}{\kappa_{ts}} - \frac{\partial_t \kappa_{ts} \partial_s \kappa_{ts}}{\kappa_{ts}^2} \right).    
    \end{equation*}
    In the new notation this gives
    \begin{equation*}
        \partial_t |_{t=0} \partial_s |_{s=0} \OT_\varepsilon(v_t,v_s) = \frac{\varepsilon}{2} \dot{v}^2 \left( ( f'(v^2) + v^2 f''(v^2) ) \left( \frac{1}{f(v^2)} -1 \right) - v^2 \frac{f'(v^2)^2}{f(v^2)^2} \right).
    \end{equation*}
    The metric tensor can thus be rewritten as
    \[
        \mathbf{g}_v(\dot{v}, \dot{v}) = g(v) \dot{v}^2, \qquad
        g(v) =  - \frac{\varepsilon}{4} \left\{ v^2 \left( f''(v^2) \frac{1-f(v^2)}{f(v^2)} - \frac{f'(v^2)^2}{f(v^2)^2} \right) + f'(v^2) \frac{1-f(v^2)}{f(v^2)} \right\}.
    \]
    It remains to simplify the expression for $g$. To this end, let us compute the first and second derivatives of $f$:
    \[
        f'(x) = \frac{8}{\varepsilon^2} \left( 1 + \frac{16}{\varepsilon^2} x \right)^{-1/2}, \qquad f''(x) = - \frac{64}{\varepsilon^4} \left( 1 + \frac{16}{\varepsilon^2} x \right)^{-3/2}.
    \]
    To simplify notations introduce the variable $u = 1+16 \varepsilon^{-2} v^2$ in such a way that 
    \[
        f(v^2) = 1 + \sqrt{u}, \quad f'(v^2) = \frac{8}{\varepsilon^2 \sqrt{u}}, \quad f''(v^2) = - \frac{64}{\varepsilon^4 u^{3/2}},
    \]
    and in particular
    \[
        \frac{1-f(v^2)}{f(v^2)} = - \frac{\sqrt{u}}{1+\sqrt{u}}.
    \]
    Plugging this into the expression for $g$ yields 
    \[
    \begin{split}
        g(v) & = \frac{\varepsilon}{4} \left\{ - v^2 \left( \frac{64}{\varepsilon^4} \frac{1}{u(1+\sqrt{u})} - \frac{64}{\varepsilon^4} \frac{1}{u(1+ \sqrt{u})^2} \right) + \frac{8}{\varepsilon^2 } \frac{1}{1+\sqrt{u}} \right\}  \\
        & = - \frac{16 v^2}{\varepsilon^3} \frac{1}{\sqrt{u}(1+ \sqrt{u})^2}  + \frac{2}{\varepsilon} \frac{1}{1+\sqrt{u}}.
    \end{split}
    \]
    To lighten the notation further, use $\frac{16}{\varepsilon^2} v^2 = u - 1$ to get
    \begin{equation*}
        g(v) = - \frac{1}{\varepsilon} \frac{u - 1}{\sqrt{u} ( 1 + \sqrt{u} )^2} + \frac{2}{\varepsilon} \frac{1}{1 + \sqrt{u}}
        = \frac{1}{\varepsilon} \frac{1 + 2\sqrt{u} + u}{\sqrt{u} (1 + \sqrt{u})^2}
        = \frac{1}{\varepsilon \sqrt{u}}.
    \end{equation*}
    The claim follows when unpacking $u =  1 + \frac{16}{\varepsilon^2} v^2$. 
\end{proof}

We now turn to the metric on $(0,\infty)$ induced by the tensor $\g$ of Theorem~\ref{theorem:gaussians_gmu}. Said differently, we look at the geodesic restriction of $\dS$ to Gaussian measures. At this point we do not know whether the set of Gaussian measures is in fact geodesically convex.

\begin{theorem} \label{theorem:gaussians_dE}
    Let $F$ be  the antiderivative of the function $x \mapsto (1+x^2)^{-1/4}$. Let $v_0, v_1 \geqslant 0$. Then the geodesic restriction $\widehat{\dS}$ of \(\dS\) to Gaussians can be expressed as
    \[
        \widehat{\dS}(v_0, v_1) = \frac{\sqrt{\varepsilon}}{4} \left| F \left( \frac{4}{\varepsilon} v_1 \right) - F \left( \frac{4}{\varepsilon} v_0 \right) \right|.
    \]
    The geodesic $(v_t)_{t \in [0,1]}$ joining $v_0$ to $v_1$ is given by
    \[
        v_t = \frac{\varepsilon}{4} F^{-1} \left[ (1-t) F \left( \frac{4}{\varepsilon} v_0 \right) + t F \left( \frac{4}{\varepsilon} v_1 \right)  \right]
    \]
    and the map $t \mapsto v_t$ is convex.
\end{theorem}

\begin{proof}
    Given the metric tensor~\eqref{eq:met-tens-gaus} on \((0,\infty)\), geodesics correspond to minimizers of the energy 
    \[
        E((v_t)_t) = \int_0^1  g(v_t) \dot{v}_t^2 \diff t.
    \]
    It is easy to see that geodesics must be monotone and since $g$ is smooth on the interval $[v_0,v_1]$ we can use standard tools from smooth calculus of variations on one-dimensional curves to obtain the minimizer.
    Let us call $q_t = \dot{v}_t$, and $L(v,q) = g(v) q^2$ the Lagrangian. Recall that the Euler--Lagrange equation is given by 
    \[
        \dd{}{t} \left( \frac{\partial L}{\partial q}(v_t, \dot{v_t}) \right) = \frac{\partial L}{\partial v}(v_t, \dot{v}_t).
    \]
    Expanding it yields
    \[
        g(v_t) \ddot{v}_t = - \frac{g'(v_t)}{2} \dot{v}_t^2.
    \]
    A simple inspection reveals that $g$ is non-negative and decreasing. Thus $g' \leqslant 0$ and the equation tells us that $\ddot{v} \geqslant 0$, which is equivalent to the fact that $(v_t)_{t} $ is convex in time.
    Note that if $v_0 \neq v_1$, then $\dot{v}$ cannot vanish: if it were the case, by the unique solvability of the above ODE, $v$ should be a constant function. Thus $(v_t)_t$ is monotone, and without loss of generality we focus on $v_0 \leqslant v_1$. In this case $(v_t)_t$ is increasing in time.
    For solutions of the above ODE one has that
    \(
    C =  g(v_t) \dot{v}_t^2 
    \)
    is constant in time (as usual for geodesics) and non-negative and thus $v$ satisfies the first-order equation 
    \[
        \dot{v}_t = \sqrt{ \frac{C}{g(v_t)} } = 2 \sqrt{C} \left( \frac{\varepsilon^2}{16} + v_t^2 \right)^{1/4} =\sqrt{C\varepsilon} \left(1 + \left(\frac{4v_t}{\varepsilon}\right)^2\right)^{1/4},  
    \]
    This equation is separable and thus can be integrated. Recalling the definition of $F$ above, the ODE yields
    \begin{equation} \label{eq:proof:gaussians_dE:ddtF}
        \dd{}{t} F \left( \frac{4}{\varepsilon} v_t \right) = 4 \sqrt{\frac{C}{\varepsilon}}, 
    \end{equation}
    and thus integrating between $0$ and $1$ gives
    \[
        C = \frac{\varepsilon}{16} \left[ F \left( \frac{4}{\varepsilon} v_1 \right) - F \left( \frac{4}{\varepsilon} v_0 \right) \right]^2. 
    \]
    Now \(\widehat{\dS}(v_0,v_1) = \sqrt{C}\) and the formula for \(v_t\) follows from~\eqref{eq:proof:gaussians_dE:ddtF}.
\end{proof}

\subsection{The two-point space} \label{section:sub:vertical_two_points}

In this section, we analyze the Riemannian tensor \(\g_\mu\) and the metric \(\dS\) for a two-point space $X=\{x_1,x_2\}$ where probability measures are fully specified by the mass $m \in [0,1]$ located at $x_1$ and the metric is fully characterized by the distance $d(x_1,x_2)$.
The precise result is stated in Proposition \ref{prop:example:two_point} below.
When $d(x_1,x_2)$ is much smaller than the blur scale $\sqrt{\varepsilon}$ we find that $\dS(\delta_{x_1},\delta_{x_2}) \approx d(x_1,x_2)$, which is the same as one would get on an interval $[x_1,x_2] \subset \R$ of length $d(x_1,x_2)$. Thus, it seems that the metric $\dS$ may `gloss over' the discreteness of $X$ at very small length scales.
Conversely, for $d(x_1,x_2) \gg \sqrt{\varepsilon}$ the diameter of the space $(\prm(X),\dS)$ grows at least like $\exp(d(x_1,x_2)^2/(2\varepsilon))$. As $d(x_1,x_2)$ increases, the blur in the entropic optimal transport underlying the construction of $\dS$ decreases. It is well known that $\prm(X)$ equipped with the quadratic Monge--Kantorovich distance is not a length space~\cite[Remark 2.1]{MaasDiscreteMCFlow2011}. Thus this exponential divergence of the diameter is consistent with the conjecture that $\dS$ converges to the unregularized quadratic Monge--Kantorovich distance as $\varepsilon \to 0$.
See Section \ref{sec:Outlook} for a slightly longer discussion and some related references.

\begin{proposition} \label{prop:example:two_point}
    Let \(X = \{ x_1, x_2 \}\). For a parameter \(r \in (0, \infty)\) let \(d(x_1,x_2) = r\), \(c = d^2\). 
    \begin{enumerate}
        \item \label{item:prop:two_point} For \(m \in (0,1)\), \(\mu = m \delta_{x_1} + (1-m) \delta_{x_2}\), \(\dot{m} \in \R\), \(\mudot = \dot{m} (\delta_{x_1} - \delta_{x_2})\), denote by \(p\) the amount of mass moving from \(x_1\) to \(x_2\) in the optimal self-transport plan from \(\mu\) to itself (see equation~\eqref{eq:proof:two_point:p} below).  Then the metric tensor is given by
        \begin{equation*} %\label{eq:two_point:gmu}
            \g_{\mu}(\mudot,\mudot) = \frac{\varepsilon}{2} \dot{m}^2 \frac{m(1-m) - p}{p ( 2 m(1-m) - p )}.
        \end{equation*} 
        \item For \(m \colon [0,1] \to [0,1]\) continuously differentiable set \(\mu_t = m_t \delta_{x_1} + (1 - m_t) \delta_{x_2}\) for \(t \in [0,1]\). Then, as $r \to 0$,
        \begin{equation} \label{eq:two_point:near}
            E((\mu_t)_t) = r^2 \cdot \int_0^1 \dot{m}_t^2 \diff t + \Oh \left( \frac{1}{\varepsilon} \norm{\dot{m}}_\infty^2 r^4 \right).
        \end{equation}
        \item Let \(C > 0\). In the setting of point \ref{item:prop:two_point}, when \(\sqrt{m (1-m)} \geq C \cdot \exp(-2r^2/\varepsilon)\) it holds
        \begin{equation} \label{eq:two_point:far}
            \g_{\mu}(\mudot,\mudot) = \frac{\varepsilon}{2} \dot{m}^2 \cdot \left( \frac{\exp( r^2 / \varepsilon )}{2 \sqrt{m (1-m)}} + \Oh( ( m (1 - m) )\inv ) ) \right)
        \end{equation}
        with the constant in the \(\Oh\)-term only depending on \(C\).
    \end{enumerate}
\end{proposition}

In particular, for small distances, \(\dot{m} = - 1\) in~\eqref{eq:two_point:near} leads to \(E(((1 - t) \delta_{x_1} + t \delta_{x_2})_t) = r^2 + \Oh \left( \frac{1}{\varepsilon} r^4 \right)\).

\begin{proof}
    Let \(m \in (0,1)\), \(\mu = m \delta_{x_1} + (1-m) \delta_{x_2}\), \(M = \sqrt{m(1-m)}\). We compute the optimal self-transport plan \(\pi\) and obtain \(k_\mu\) as the density of \(\pi\) with respect to \(\mu \otimes \mu\). Since \(X\) only contains two points, we may use the matrix vector notation of the discrete setting. To avoid ambiguity, we denote the entry-wise multiplication of two matrices \(A,B\) by \(A \odot B\). Denoting \(a_\mu = \exp(\fmumu/\varepsilon)\) as a vector \(a=(a_1,a_2)^\Transpose\), we see that $\pi = k_\mu \odot \mu \otimes \mu = (a \cdot a^\Transpose) \odot k_c \odot \mu \otimes \mu$. The set of admissible transport plans is described entirely by one parameter \(p\) describing the exchange of mass between the two points. With $\exp( - 2 d(x_1, x_2)^2 / \varepsilon ) \assignRe b \in (0,1)$, 
    \begin{equation*} 
        \pi = 
        \begin{pmatrix}
            m - p & p \\
            p & (1 - m) - p 
        \end{pmatrix}, 
        \qquad
        k_c \odot \mu \otimes \mu = 
        \begin{pmatrix}
            m^2 & M^2 \sqrt{b} \\
            M^2 \sqrt{b} & (1-m)^2 
        \end{pmatrix}.
    \end{equation*} 
    We have \(p^2 / ((m-p)(1-m-p)) = (a_1 a_2 M^2 \sqrt{b})^2 / ( a_1^2 m^2 a_2^2 (1-m)^2 ) = b.\)
    This is a quadratic equation whose solution is 
    \begin{equation}
    \label{eq:proof:two_point:p}
        p = \frac{1}{2(1-b)} \left[ - b + \sqrt{b^2 + 4 M^2 b (1 - b)} \right]
    \end{equation}
    We can express the matrix representing $K_\mu$ as a scaling of $\pi$: 
    \begin{align*} %\label{eq:proof:two_point:Kmu_pi}
        K_\mu & = \begin{pmatrix} m^{-1} & 0 \\ 0 & (1-m)^{-1} \end{pmatrix} \cdot \pi = 
        \begin{pmatrix}
            \frac{m - p}{m} & \frac{p}{m} \\
            \frac{p}{1-m} & \frac{1-m-p}{1-m} 
        \end{pmatrix}. 
    \end{align*}
    Without suprise $(1,1)^\Transpose$ is an eigenvector associated to the eigenvalue $\lambda_1 = 1$. We know that the other eigenvector of \(K_\mu\) should be orthogonal to $(1,1)^\Transpose$ in the inner product of $L^2(X,\mu)$, so it must be $(1-m,-m)^\Transpose$. The corresponding eigenvalue is computed via
    \begin{equation} \label{eq:proof:two_point:second_eigenvalue}
        K_\mu \begin{pmatrix} 1-m \\ - m \end{pmatrix} = ( 1 - {p}/{M^2} ) \begin{pmatrix} 1-m \\ - m \end{pmatrix}. 
    \end{equation}
    In particular, on this eigenspace $(\id - K_\mu^2)\inv$ acts as multiplication by $(1-\lambda_2^2)\inv$ with \(\lambda_2 = 1 - p / M^2\). 
    To compute the metric tensor \(\g_\mu(\mudot, \mudot) = \pair{\mudot}{(\id - K_\mu^2)\inv H_\mu[\mudot]}\) for some vertical perturbation \(\mudot = \dot{m} \cdot (1, -1)^\Transpose\), we need to compute $H_\mu[\mudot] = K_\mu \left[ {\mathrm{d}\mudot}/{\mathrm{d}\mu} \right]$. As ${\mathrm{d}\mudot}/{\mathrm{d}\mu} = \dot{m}/M^2 (1-m, -m)^\Transpose$, by the previous computation
    \begin{equation} \label{eq:proof:two_point:Hmu}
    H_\mu \begin{pmatrix} 1 \\ -1 \end{pmatrix} = \frac{1 - {p}/{M^2}}{M^2} \begin{pmatrix}  1-m \\ -  m \end{pmatrix}. 
    \end{equation}
    The formulas \eqref{eq:proof:two_point:second_eigenvalue} and \eqref{eq:proof:two_point:Hmu} lead to
    \begin{equation} \label{eq:proof:two_point:gmu}
        \g_\mu(\mudot,\mudot) = \frac{\varepsilon}{2} \frac{\dot{m}^2 \lambda_2}{M^2 (1 - \lambda_2^2)} 
        = \frac{\varepsilon}{2} \dot{m}^2 \frac{M^2 - p}{p ( 2 M^2 - p )}.
    \end{equation}
    
    \underline{Small $r$ asymptotic:}
    The small parameter in the expansion~\eqref{eq:proof:two_point:p} is \(\tau = 2 d(x_1,x_2)^2 / \varepsilon\), so that we have \((1 - b) = \tau + \Oh(\tau^2)\). A second-order Taylor expansion of the square root leads to \(p = M^2 - \tau M^4 + \Oh(\tau^2)\), which gives \(\lambda_2 = \tau M^2 + \Oh(\tau^2)\).
    Thus \((1 - \lambda_2^2)\inv = 1 + \Oh( \tau^2 )\) and \eqref{eq:proof:two_point:gmu} gives
    \begin{equation*}
        \g_\mu(\mudot, \mudot) 
        = \frac{\varepsilon}{2} ( \tau \dot{m}^2 + \Oh(\tau^2 \dot{m}^2) )
        = \dot{m}^2 \cdot d(x_1, x_2)^2 + \Oh \left( \frac{1}{\varepsilon} \dot{m}^2 d(x_1, x_2)^4 \right).
    \end{equation*}
    
    \underline{Big $r$ asymptotic:}
    We set \(\varrho \assign \sqrt{b} = \exp( - d(x_1,x_2)^2 / \varepsilon )\). Provided that \(M \geq C \varrho\) for some \(C > 0\), rewriting~\eqref{eq:proof:two_point:p} gives
    \begin{equation*} %\label{eq:proof:two_point:p_bigM}
        p = \frac{1}{2 ( 1 - \varrho^2 )} \left( - \varrho^2 + 2 M \varrho \sqrt{1 + \varrho^2 ( 1 - 4 M^2) / (4 M ^2)} \right) 
        = M \varrho + \Oh( \varrho^2 )
    \end{equation*}
    with the constant in the \(\Oh\)-term only depending on \(C\). 
    This gives
    \begin{align*}
        \lambda_2 = 1 - \frac{p}{M^2} = 1 - \frac{\varrho}{M} + \Oh((\varrho / M)^2
    \end{align*}
    and using~\eqref{eq:proof:two_point:gmu} it leads to
    \begin{align*}
        \g_\mu(\mudot,\mudot) 
        &= \frac{\varepsilon}{2} \dot{m}^2 \cdot \left( \frac{1}{2 M \varrho} + \Oh( 1 / M^2 ) \right)
        = \frac{\varepsilon}{2} \dot{m}^2 \cdot \left( \frac{\exp( r^2 / \varepsilon )}{2 M} + \Oh( 1 / M^2) ) \right). \qedhere
    \end{align*}
\end{proof}

For small \(M / \varrho\) one can use analogous methods to find \(p = M^2 + \Oh(M^4 / \varrho^2)\), which leads to \(\lambda_2 = \Oh(1/\varrho^2)\) and \(\g_\mu(\mudot,\mudot) = \Oh(\varepsilon \dot{m}^2 / \varrho^2)\).

\section{Negative results on the Sinkhorn divergence and the metric tensor} \label{section:negative_examples}

\subsection{\texorpdfstring{$\sqrt{S_\varepsilon}$}{The square root of the Sinkhorn divergence} violates the triangle inequality} \label{subsection:Seps_triangle}

One of the main goals of this paper is to find a ``Riemannian'' distance associated with the Sinkhorn divergence $S_\varepsilon$. The reader may wonder why one could not directly use $S_\varepsilon$ as a distance, since it enjoys several nice features, as recalled in Theorem~\ref{theorem:FeydyEtAl}. However, despite being positive definite, already for the squared Euclidean distance as the cost function the Sinkhorn divergence, or rather $\sqrt{S_\varepsilon}$ (since $S_\varepsilon$ converges to the square of the Monge-Kantorovich distance), is not itself a distance! Although this fact is well known in the community, we have not been able to find an explicit statement of this property. For this reason, we provide an elementary proof, which will actually exclude any convex power of $\sqrt{S_\varepsilon}$.

\begin{theorem} \label{theorem:Seps_triangle}
    On \(\R\) let \(c(x,y) = \abs{x-y}^2\). Then \(S_\varepsilon^\alpha\) fails to satisfy the triangle inequality for any \(\alpha \geq \frac{1}{2}\).
\end{theorem}

The case \(0 < \alpha < \frac{1}{2}\) is unsuitable for our purposes too, as \(S_\varepsilon^\alpha (\delta_x, \delta_y) = \norm{x-y}^{2\alpha}\) is a concave function of the Euclidean distance. Hence if $S_\varepsilon^\alpha$ were a metric at all, it would induce a Finslerian rather than Riemannian geometry.

\begin{proof}
    It is readily verified that $S_\varepsilon^\alpha$ cannot be a metric for any $\alpha > 1/2$, as it is sufficient to consider $\mu_i = \delta_{m_i}$ and observe that
    \(S_\varepsilon^\alpha(\mu_0,\mu_1) = \|m_0-m_1\|^{2\alpha}\)
    does not satisfy the triangle inequality. 
    
    Finally, $\sqrt{S_\varepsilon}$ is not a distance either.
    The idea is to look at the Sinkhorn divergence between 1-dimensional Gaussian distributions, for which an explicit formula is available as a direct consequence of~\eqref{eq:gaussian-OTeps} (see~\cite[Cor.~1]{Mallasto2021}). Indeed, given $\mu_i = \mathcal{N}(m_i,v_i)$ with $m_i \in \mathbb{R}$, $v_i \geq 0$ (for $i=0,1$), the Sinkhorn divergence between them is explicitly given by
    \begin{equation}\label{eq:gaussian-sinkhorn}
        S_\varepsilon(\mu_0,\mu_1) = \|m_0-m_1\|^2 + \frac{\varepsilon}{4}\bigg(\kappa_{00} - 2\kappa_{01} + \kappa_{11} + \log\Big(\frac{\kappa_{01}^2}{\kappa_{00} \kappa_{11}} \Big) \bigg),
    \end{equation}
    where $\kappa_{ij} \assign 1+(1+16\varepsilon^{-2}v_i v_j)^{1/2}$, for $i,j \in \{0,1\}$, as in~\eqref{eq:gaussian-OTeps}.
    Choose $\mu_0 = \delta_0$, $\mu_1 = \mathcal{N}(0,\frac{\varepsilon}{4}x)$ and $\mu_2 = \mathcal{N}(0,\frac{\varepsilon}{4}v)$, for a fixed $v>0$: we claim that
    \begin{equation}\label{eq:triangle-false}
        \sqrt{S_\varepsilon(\mu_0,\mu_2)} > \sqrt{S_\varepsilon(\mu_0,\mu_1)} + \sqrt{S_\varepsilon(\mu_1,\mu_2)}
    \end{equation}
    when $x$ is sufficiently small and $v$ sufficiently large. To prove the claim, observe that
    \[
        4\varepsilon^{-1}S_\varepsilon(\mu_0,\mu_2) = \sqrt{1+v^2}-\log(1+\sqrt{1+v^2})+\log 2-1.
    \]
    To lighten the notation, let us introduce
    \[
    \begin{split}
        \zeta(x) \assign 4\varepsilon^{-1}S_\varepsilon(\mu_0,\mu_1) & = \sqrt{1+x^2}-\log(1+\sqrt{1+x^2})+\log 2-1 \\
        \xi(x) \assign 4\varepsilon^{-1}S_\varepsilon(\mu_1,\mu_2) & = \sqrt{1+v^2} - 2\sqrt{1+v x} + 2\log(1+\sqrt{1+v x}) \\
        & \qquad -\log(1+\sqrt{1+v^2})-\log 2+1+\zeta(x).
    \end{split}
    \]
    Squaring both sides, and after algebraic manipulations,~\eqref{eq:triangle-false} is thus equivalent to
    \[
        \Psi(x) \assign 2\log(1+\sqrt{1+v x}) - 2\sqrt{1+v x} - 2\log 2 + 2 + 2\zeta(x) + 2\sqrt{\zeta(x)\xi(x)} < 0
    \]
    for $x$ sufficiently close to 0. To prove this inequality, let us compute
    \[
    \begin{split}
        \Psi'(x) = & -\frac{v}{1+\sqrt{1+v x}} + \frac{2x}{1+\sqrt{1+x^2}} + \left(\frac{\xi(x)}{\zeta(x)}\right)^{1/2}\frac{x}{1+\sqrt{1+x^2}} \\
        & + \left(\frac{\zeta(x)}{\xi(x)}\right)^{1/2}\left(-\frac{v}{1+\sqrt{1+v x}} + \frac{x}{1+\sqrt{1+x^2}}\right)
    \end{split}
    \]
    and note that $\zeta(0)=0$, so that $\Psi(0)=0$ and
    \[
    \lim_{x \to 0}\Psi'(x) = -\frac{v}{2} + \frac{\sqrt{\xi(0)}}{2}\lim_{x \to 0}\frac{x}{\sqrt{\zeta(x)}}
    \]
    with $\xi(0) = \sqrt{1 + v^2} - 1 + \log 2 - \log(1 + \sqrt{1 + v^2}) > 0$. 
    It is now sufficient to observe that $\zeta'(0)=0$ and $\zeta''(0)=1/2$: by a Taylor expansion of $\zeta(x)$ up to order 2 we finally get
    \[
    \Psi'(0) = -\frac{v}{2} + (\sqrt{1 + v^2} - 1 + \log 2 - \log(1 + \sqrt{1 + v^2}))^{1/2},
    \]
    which is negative provided $v$ is sufficiently large. We have thus shown that $\Psi(0)=0$ and $\Psi'(0) < 0$, hence the conclusion.
\end{proof}

\begin{remark}
\label{rmk:triangle_inequality_dirac}
    One can also use the measures \(\mu_r = \frac{1}{2} \delta_r + \frac{1}{2} \delta_{-r}\) from Section~\ref{section:sub:nonconvexity} to show that the square root of the Sinkhorn divergence does not satisfy the triangle inequality:
    Choose \(r = \sqrt{\varepsilon}\). Then Lemma~\ref{lemma:explicit_OT-S} gives
    \begin{align*}
        &\frac{1}{\varepsilon} S_\varepsilon (\mu_0,\mu_r) 
        = 1 - \frac{1}{2} \log 2 + \frac{1}{2} \log \left( 1 + \frac{1}{e^4} \right), \\
        &\frac{1}{\varepsilon} S_\varepsilon (\mu_r,\mu_{2r}) 
        = 1 - \log \left( 1 + \frac{1}{e^8} \right) + \frac{1}{2} \log \left( 1 + \frac{1}{e^4} \right) + \frac{1}{2} \log \left( 1 + \frac{1}{e^{16}} \right), \\
        &\frac{1}{\varepsilon} S_\varepsilon (\mu_0,\mu_{2r}) 
        = 4 - \frac{1}{2} \log 2 + \frac{1}{2} \log \left( 1 + \frac{1}{e^{16}} \right).
    \end{align*}
    This yields \(\frac{1}{\sqrt{\varepsilon}} \cdot \left( \sqrt{S_\varepsilon(\mu_0,\mu_{2r})} - \sqrt{S_\varepsilon(\mu_0,\mu_{r})} - \sqrt{S_\varepsilon(\mu_r,\mu_{2r})} \right) \approx 0.093 > 0\), which breaks the triangle inequality.
\end{remark}

\begin{figure}
    \centering
    \includegraphics[scale=0.54]{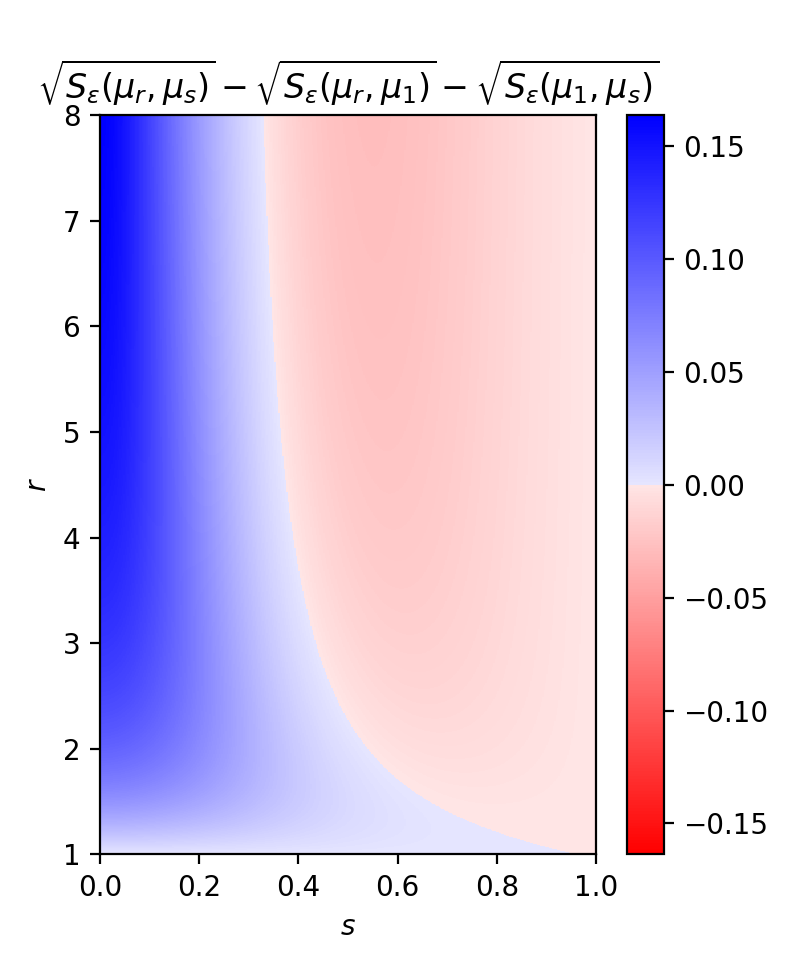}
    \includegraphics[scale=0.54]{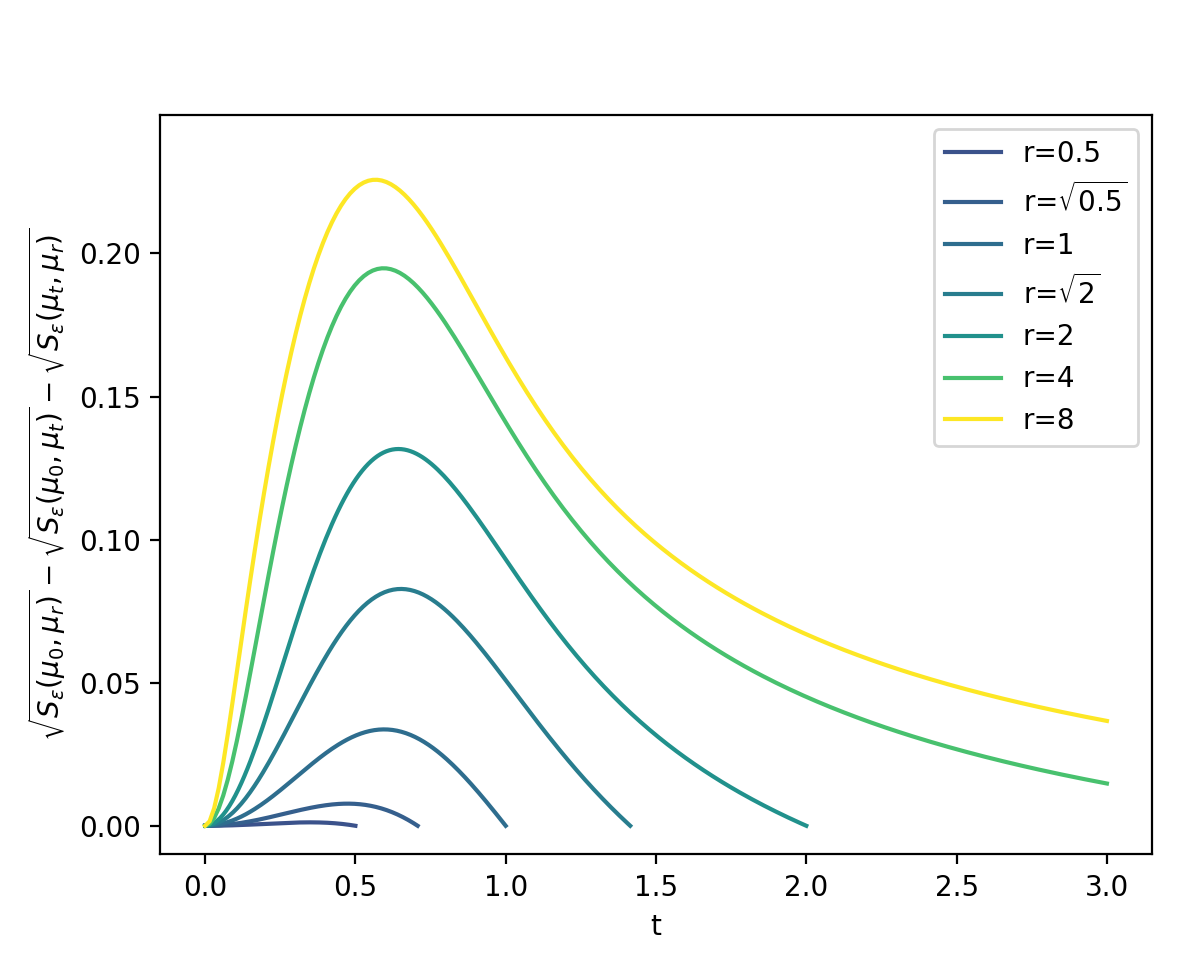}
    \caption{Gap in the triangle inequality for \(\sqrt{S_\varepsilon}\) on the measures \(\mu_r = \frac{1}{2} \delta_r + \frac{1}{2} \delta_{-r}\), discussed in Remark~\ref{rmk:triangle_inequality_dirac}, for \(\varepsilon=1\). Left: The gap in the triangle inequality when going from \(\mu_s\) to \(\mu_r\) via \(\mu_1\). The triangle inequality breaks in the blue region, where \(s \ll 1 = \sqrt{\varepsilon} \ll r\). Right: Gap for \(s=0\) and varying intermediate position \(0 \leq t \leq r\) for different values of \(r\).}
    \label{figure:triangle_Seps}
\end{figure}

\begin{remark}
    We observe that in the proof of Theorem~\ref{theorem:Seps_triangle} and in Figure~\ref{figure:triangle_Seps}, the triangle inequality for \(\sqrt{S_\varepsilon}\) breaks when the initial and intermediate measures are concentrated below the blur scale and the target measure is more spread out.
\end{remark}

\subsection{The Sinkhorn divergence is not jointly convex} \label{section:sub:nonconvexity}

The second negative result about the Sinkhorn divergence deals with joint convexity. The Sinkhorn divergence $S_\varepsilon$ interpolates between squared MMD and Wasserstein distances which are both jointly convex functions of their inputs, and moreover $S_\varepsilon$ is coordinate-wise convex as recalled in Theorem~\ref{theorem:FeydyEtAl}. Thus it is natural to investigate the convexity properties of $S_\varepsilon$ as a function on $\prm(X) \times \prm(X)$, and we show that the answer is negative.

Throughout this section let \(c(x,y) = \abs{x-y}^2\) be the quadratic cost.

\begin{theorem}
\label{theorem:Sinkhorn_nonconvex}
    Let $X \subseteq \R$ be the closure of an open set with \(\diam(X) > \sqrt{2\varepsilon}\). Then the Sinkhorn divergence $S_\varepsilon$ is not jointly convex over $\prm(X) \times \prm(X)$. 
\end{theorem}

We first examine the measures that will form the counterexample.

\begin{lemma}\label{lemma:explicit_OT-S}
    For \(r \in [0,\infty)\) set \(\mu_r \assign \frac{1}{2}\delta_r + \frac{1}{2}\delta_{-r}\). Then for \(r,s \in [0,\infty)\) it holds
    \begin{align*}
        &\frac{1}{\varepsilon} \OT_\varepsilon(\mu_r,\mu_s) = \log 2 - \log(k_c(r,s) + k_c(r,-s)), \\
        &\frac{1}{\varepsilon} S_\varepsilon(\mu_r,\mu_s) = - \log(k_c(r,s)+k_c(r,-s)) + \frac{1}{2} \log(1+k_c(r,-r)) + \frac{1}{2} \log(1+k_c(s,-s)).
    \end{align*}
\end{lemma}

\begin{proof}
    We first compute the Schrödinger potentials for $(\mu_r,\mu_s)$. As \(\mu_r\) and \(\mu_s\) are symmetric around \(0\), we have \(f_{\mu_r,\mu_s}(r) = f_{\mu_r,\mu_s}(-r) \assignRe f\), \(g_{\mu_r,\mu_s}(s) = g_{\mu_r,\mu_s}(-s) \assignRe g\). The Schrödinger system~\eqref{eq:Schroedinger_system} then reduces to 
    \begin{equation*}
        \exp(-f/\varepsilon) = \frac{1}{2} k_c(r,s) \exp(g/\varepsilon) + \frac{1}{2} k_c(r,-s) \exp(g/\varepsilon),
    \end{equation*}    
    i.e.~\(\frac{1}{\varepsilon} (f + g) = \log 2 - \log(k_c(r,s) + k_c(r,-s))\).
    Now the claim follows from~\eqref{eq:entropic_dual_max}.
\end{proof}

With these explicit expressions at hand, the proof of Theorem~\ref{theorem:Sinkhorn_nonconvex} comes down to an appropriate choice of parameters.

\begin{proof}[\textbf{Proof of Theorem~\ref{theorem:Sinkhorn_nonconvex}}]
    Use \(\mu_r \assign \frac{1}{2}\delta_r + \frac{1}{2}\delta_{-r}\) as in Lemma~\ref{lemma:explicit_OT-S}. For lighter notation, we assume that \(\pm r \in X\) and we set \(\kappa = k_c(r,-r)\). We compute the Taylor expansion of \(S_\varepsilon(\mu_r,\mu_{r+t})\) for \(r \in[0,\infty)\).
    With the explicit formula for the Sinkhorn divergence from Lemma~\ref{lemma:explicit_OT-S} and the second-order Taylor expansions
    \[
    \begin{split}
    & k_c(r,r+x) = 1 - \frac{1}{\varepsilon}x^2 + \Oh(x^3), \\
    & k_c(r,-(r+x)) = \kappa - \kappa \frac{4r}{\varepsilon} x - \kappa \left( \frac{1}{\varepsilon} - \frac{8 r^2}{\varepsilon^2} \right) x^2 + \Oh(x^3), \\
    & \log(1+\kappa + x) = \log(1+\kappa) + \frac{1}{1+\kappa} x - \frac{1}{2(1+\kappa)^2} x^2 + \Oh(x^3)
    \end{split}
    \]
    we obtain
    \begin{align}
    \begin{split} \label{eq:proof:nonconvexity_example}
        \frac{1}{\varepsilon} S_\varepsilon(\mu_r,\mu_{r+t}) 
        &= - \log \left( 1 + \kappa - \frac{1}{\varepsilon} t^2 - \kappa \left( \frac{4r}{\varepsilon} t + \left(\frac{1}{\varepsilon} - \frac{8r^2}{\varepsilon^2} \right) t^2 \right) + \Oh(t^3) \right) + \frac{1}{2} \log(1+\kappa) \\
        &\quad + \frac{1}{2} \log \left( 1+\kappa - \kappa \left( \frac{4r}{\varepsilon} \cdot 2t + \left( \frac{1}{\varepsilon} - \frac{8r^2}{\varepsilon^2} \right) \cdot 4t^2 \right) + \Oh(t^3) \right) \\
        &= \frac{1}{\varepsilon} t^2 + t^2 \frac{\kappa}{1+\kappa} \left( \frac{1}{1+\kappa}\frac{8r^2}{\varepsilon^2} - \frac{2}{\varepsilon} \right) + \Oh(t^3).
    \end{split}
    \end{align}
    As $\kappa \leq 1$, when \(r > \sqrt{\varepsilon / 2}\) this is greater than \(\frac{1}{\varepsilon} S_\varepsilon(\delta_{\pm r}, \delta_{\pm(r+t)}) = \frac{1}{\varepsilon} t^2\) for small \(t\). Thus in this case $S_\varepsilon(\mu_r,\mu_{r+t}) > \frac{1}{2} S_\varepsilon(\delta_r,\delta_{r+t}) + \frac{1}{2} S_\varepsilon(\delta_{-r},\delta_{-(r+t)})$, disproving convexity.
    
    When \(\diam(X) > \sqrt{2\varepsilon}\) we can always translate \(X\) to contain points \(\pm r\) with \(r > \sqrt{\varepsilon/2}\), hence disproving joint convexity of \(S_\varepsilon\).
\end{proof}

The non-convexity of $S_\varepsilon$ descends to the level of the metric tensor.

\begin{corollary} \label{cor:nonconvexity_tensor}
    Let $X \subseteq \R$ be the closure of an open set with \(\diam(X) > \sqrt{2\varepsilon}\). Then the map \((\mu,\mudot) \mapsto \g_\mu(\mudot,\mudot)\) is not jointly convex on $\prm(X) \times \Cont^1(X)^*_0$.
\end{corollary}

\begin{proof}
    We follow the proof of Theorem~\ref{theorem:Sinkhorn_nonconvex}.  
    Denote \(\mu = \mu_r\), \(\mudot = \frac{1}{2}\delta_r' - \frac{1}{2}\delta_{-r}'\). Theorem~\ref{theorem:hessian_sinkhorn} and~\eqref{eq:proof:nonconvexity_example} give
    \begin{align} \label{eq:proof:nonconvexity_tensor}
        \g_{\mu}(\mudot,\mudot)  
        = \lim_{t \to 0} \frac{S_\varepsilon(\mu_r,\mu_{r+t})}{t^2} 
        = 1 + \frac{\kappa}{1+\kappa} \left( \frac{1}{1+\kappa}\frac{8r^2}{\varepsilon} - 2 \right),
    \end{align}
    which is greater than \(\g_{\delta_{\pm r}}(\pm \delta_{\pm r}',\pm \delta_{\pm r}') = 1\) when \(r > \sqrt{\varepsilon / 2}\). 
\end{proof}

As discussed in the introduction, if the metric tensor were jointly convex, then existence of geodesic (Theorem~\ref{theorem:dE_minimizers_existence}) could be obtained with arguments of convex analysis. Corollary~\ref{cor:nonconvexity_tensor} shows that some more involved arguments  such as the ones  given in Section~\ref{section:sub:geodesics} are necessary. 
Joint convexity of the metric tensor \(\tg_\mu(\betadot,\betadot)\) in \(\mu\) and \(\betadot\) is ill defined, because the compatibility assumption \(\pair{\B(\mu)}{\betadot} = 0\) can fail at a convex combination. 

\begin{figure}[!ht]
\begin{center}
    \includegraphics[scale=0.6]{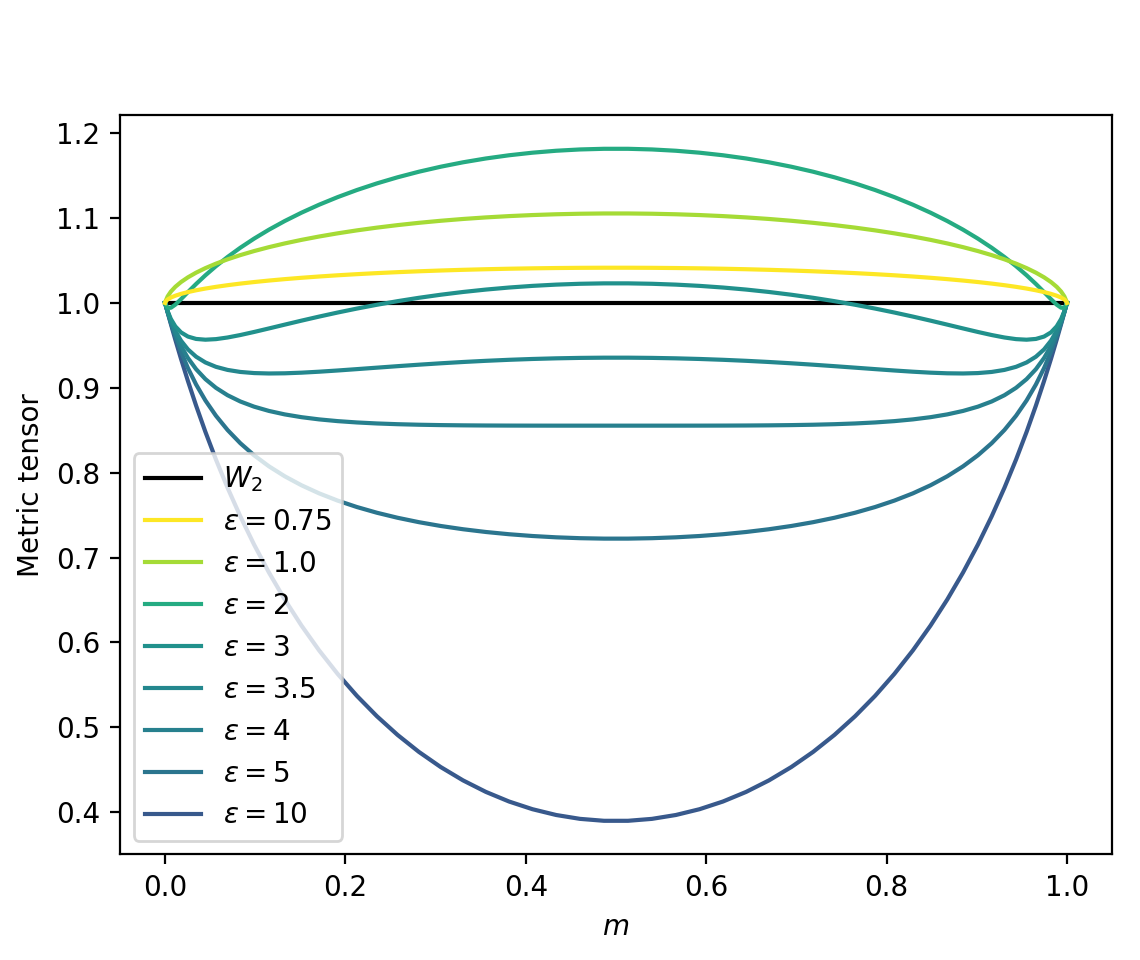}
    \includegraphics[scale=0.6]{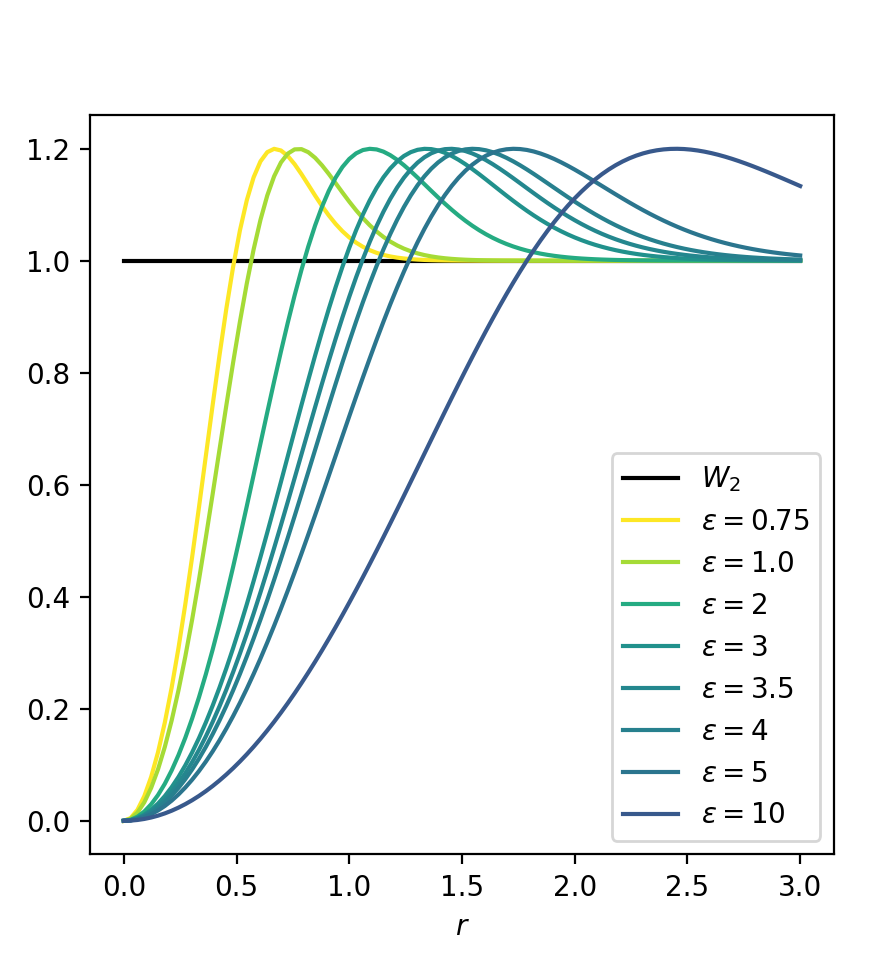}
    \caption{Numerically computed values of the metric tensor of two point masses moving apart with \(\mu = m \delta_{-r} + (1-m) \delta_r\), \(\mudot = - m \delta_{-r}' + (1-m) \delta_r'\) for different values of \(\varepsilon\). This is the linear interpolation between a point mass at \(-r\) moving to the left and a point mass at \(r\) moving to the right ($m=1$ and $m=0$ respectively). Left: fixed \(r=1\), varying mass \(m\). Joint convexity of the metric tensor fails, as the graphs sometimes exceed the value one. Right: fixed \(m=\frac{1}{2}\), varying position \(r\). This is the term~\eqref{eq:proof:nonconvexity_tensor} computed in the proof of Corollary~\ref{cor:nonconvexity_tensor}, which only depends on \({r}/{\sqrt{\varepsilon}}\). In both plots, the black line shows \(\norm{v}^2_{L^2(X,\mu)} = \g^0_\mu(\mudot,\mudot) = 1\) for the metric tensor of classical optimal transport~\eqref{eq:intro_gmu_OT}.} 
    \label{figure:nonconvexity_gmu}
\end{center}
\end{figure}

\begin{remark}
    Although very explicit, the previous computations do not shed light on the reason why $S_\varepsilon$ fails to be jointly convex. Instead, we present also a more abstract argument, which is originally how we proved Theorem~\ref{theorem:Sinkhorn_nonconvex}. It proceeds by contradiction, by combining three results already present in the literature: 
    \begin{enumerate}[(a)]
        \item A very insightful remark of Savaré and Sodini~\cite{Savare2022}, which says that $\OT \assign \OT_0$ (i.e.\ the unregularized transport cost induced by $c$ from \eqref{eq:intro_def_OT}) is the largest convex and \textweakstar-lower semicontinuous function which coincides with $c$ on Dirac masses. Since the set of couplings between two Dirac masses contains only one element, with $c(x,y) = |x-y|^2$ we have $S_\varepsilon(\delta_x, \delta_y) = c(x,y)$ for any $x,y \in X$. As a consequence, if $S_\varepsilon$ were jointly convex, as it is in addition \textweakstar-continuous, then \(S_\varepsilon (\mu_0, \mu_1) \leq \OT (\mu_0, \mu_1)\) for all \(\mu_0, \mu_1 \in \prm_2(\R)\). So rather than negating the convexity of $S_\varepsilon$, it is enough to find measures $\mu_0, \mu_1$ such that $S_\varepsilon(\mu_0, \mu_1) > \OT(\mu_0,\mu_1)$.  
        
        \item We then use a small-\(\varepsilon\) asymptotics of \(S_\varepsilon\)~\cite{Conforti2019, Chizat2020}. Introducing the Fisher information of a measure $\mu$ w.r.t.\ the Lebesgue measure as
        \begin{equation*}
            I(\mu) \coloneqq 4 \int_{\R} \left| \nabla \sqrt{\rho} \right|^2 \diff x, \qquad \textrm{if } \mu = \rho\diff x \textrm{ and } \sqrt{\rho} \in H^1(\R, \diff x),
        \end{equation*}
        and $+ \infty$ otherwise, this result reads as follows: Let $\mu_0, \mu_1 \in \prm(\R)$ have bounded densities and supports and denote by $(\mu_t)_{t \in [0,1]}$ the Wasserstein-2 geodesic connecting them. If $I(\mu_0)$, $I(\mu_1)$, and $\int_0^1 I(\mu_t)\diff t$ are finite, then it holds
        \begin{equation}
        \label{eq:lower-bound-sinkhorn}
            S_\varepsilon(\mu_0, \mu_1) - \OT_0(\mu_0, \mu_1) = \frac{\varepsilon^2}{8} \left( \int_0^1 I(\mu_t) \diff t - \frac{I(\mu_0) + I(\mu_1)}{2} \right) + o(\varepsilon^2).
        \end{equation}
        Thus, given (a), it is enough to find measures $\mu_0$, $\mu_1$ such that 
        \begin{equation}
        \label{eq:opposite-convexity}    
            \int_0^1 I(\mu_t) \diff t  > \frac{I(\mu_0) + I(\mu_1)}{2},
        \end{equation}
        and this will disprove joint convexity of $S_\varepsilon$ for $\varepsilon$ small enough.

        \item  Eventually, we use the property that the Fisher information is not geodesically convex~\cite{Carrillo2009}. Specifically, by~\cite[Ex.~2 in Sec.~3.2]{Carrillo2009} there exists a Wasserstein geodesic $(\tilde{\mu}_t)_{t \in [0,1]}$ on $\R$ such that the function $t \mapsto I(\tilde{\mu}_t)$ is not convex. This implies that we can find $t_0, t_1$ such that 
        \begin{equation*}
            \frac{1}{t_1 - t_0}\int_{t_0}^{t_1} I(\tilde{\mu}_t) \diff t  > \frac{I(\tilde{\mu}_{t_0}) + I(\tilde{\mu}_{t_1})}{2},
        \end{equation*}
        because if the equality were in the other direction for any $t_0, t_1$, it would imply convexity of $t \mapsto I(\tilde{\mu}_t)$. Defining $\mu_t \assign \tilde\mu_{t_0 + t(t_1-t_0)}$ for $t \in [0,1]$, we obtain~\eqref{eq:opposite-convexity} and conclude.
    \end{enumerate}
    
    A scaling argument then extends the lack of convexity to any $\varepsilon > 0$. Indeed, it suffices to define $L^\varepsilon(x) \coloneqq \varepsilon^{-1/2} x$ and observe that $S_\varepsilon(\mu,\nu) = \varepsilon S_1(L^\varepsilon_\# \mu, L^\varepsilon_\# \nu)$, as already mentioned in~\eqref{eq:S_rescaled}, so that non-convexity of $S_\varepsilon$ for one $\varepsilon > 0$ is equivalent to non-convexity for any $\varepsilon > 0$.
    
    \medskip
    
    At the technical level, however, there is a slight gap in the argument we have just outlined, since \eqref{eq:lower-bound-sinkhorn} requires $\mu_0,\mu_1$ to have bounded support, while the counterexample built by Carrillo and Slep\v cev \cite[Ex.~2 in Sec.~3.2]{Carrillo2009} has unbounded support. This difficulty can be overcome by a standard truncation procedure, but the resulting argument would be much heavier to read than the straightforward computation of Proposition \ref{lemma:explicit_OT-S}. 
\end{remark}

\section{Outlook}
\label{sec:Outlook}
In this article we have introduced a Riemannian distance associated with entropic optimal transport and established its most fundamental properties, such as its connection to RKHS spaces, the existence of geodesics, its metrization of the \textweakstar~topology, as well as some instructive examples, such as translations, Gaussians, and the two-point space.
This Riemannian perspective provides further theoretical underpinning for the widespread application of the Sinkhorn divergence in machine learning.
After this foundation has been established, there are various potential applications and more advanced intriguing directions for further study that we comment now.

\medskip

\emph{Convergence when $\varepsilon \to 0$}.
One immediate question is, as one sends $\varepsilon$ to $0$, whether the limit of $\dS$ is the quadratic Monge--Kantorovich distance, and if there is a corresponding metric convergence in the sense of Gromov--Hausdorff.
In Section~\ref{sec:sub:epsilon_limits} we claim that we should have at least convergence of the metric tensor, but a rigorous proof is missing.

\emph{Homogeneization when the metric space $X$ is refined.}
The new Sinkhorn distance is a geodesic distance on any compact metric space $(X,d)$ when $c=d^2$ induces a positive definite universal kernel, in particular also in cases where $(X,d)$ is not a length space itself, such as the two-point space. This is in stark contrast to the quadratic Monge--Kantorovich distance, for which there are no geodesics when $X$ is discrete. A distance on probability measures over discrete graphs that resembles the Monge--Kantorovich distance has been introduced in \cite{MaasDiscreteMCFlow2011} with non-trivial homogenization results given, for instance, in \cite{GigliMaas-GromovHausdorff2013,GlKoMa2018}. For the distance $\dS$ for measures on the two-point space, a consequence of Proposition \ref{prop:example:two_point} is that $\dS(\delta_{x_1},\delta_{x_2}) \approx d(x_1,x_2)$ when $d(x_1,x_2) \ll \sqrt{\varepsilon}$. Based on this one may wonder, whether the metric space $(\prm(X),\dS)$ can be approximated by $(\prm(X_n),\dS)$ where $(X_n)_{n \in \N}$ is a sequence of increasingly fine discrete approximations of $X$ with respect to the ground metric $d$. In other words, does $\dS$ over $\prm(X)$ arise as the homogenization of itself?
In combination with the previous question one might even ask, whether it is possible to decrease $\varepsilon$ to zero at a suitable rate, while $X_n$ is refined, such that the joint limit becomes the quadratic Monge--Kantorovich distance over $\prm(X)$.
 This provides a novel way to define geodesic distances on probability measures over discrete spaces with a potentially consistent continuum limit.

\emph{Numerical approximation.}
A more practical concern is the efficient numerical approximation of $\dS$ and its geodesics. As the Riemannian tensor of $\dS$ was chosen to be the Hessian of the Sinkhorn divergence $S_\varepsilon$, one might conjecture that
\begin{equation*}
    \dS(\mu,\nu)^2 = \lim_{N \to \infty} \inf \setgiven{ N \cdot \sum_{n=0}^{N-1} S_\varepsilon(\rho_n,\rho_{n+1}) }{ \rho_0,\ldots,\rho_N \in \prm(X), \, \rho_0=\mu,\, \rho_N = \nu }
\end{equation*}
in the spirit of time-discrete variational geodesic calculus \cite{RumpfWirth2015}. For this it seems that one requires a more uniform version of the expansion of Theorem \ref{theorem:hessian_sinkhorn}.

\emph{Sample complexity.}
From a statistical perspective an interesting question is whether $\dS$ inherits the favourable sample complexity of the Sinkhorn divergence \cite{Genevay:Samples,MenaWeed}.

\emph{Linearization.}
For the quadratic Monge--Kantorovich distance local approximation via linearization has been established as a useful tool for data analysis applications \cite{Wang2013} which reduces the computational complexity and allows to combine the geometry of optimal transport with standard data analysis tools on Hilbert spaces. 
If this can be adapted to the new metric $\dS$ it may provide a practically relevant application for geometric data analysis.

\emph{Gradient flows.}
Gradient flow with respect to the Monge-Kantorovich distance is a well-studied topic with important consequences at the theoretical and practical level. It would be important to understand how the gradient flows with respect to the metric $\dS$ behave. They should be recovered as a scaling limit if ones tries to use the minimizing movemement scheme for gradient flows, but with the Sinkhorn divergence $S_\varepsilon$ in place of the more classical squared Monge-Kantorovich distance, for a fixed $\varepsilon > 0$.  This may lead to a new class of statistically robust interacting particle methods in high dimensions with potential applications in machine learning, similar to \cite{Sander22a,agarwal2025langevin}, see also \cite{Agarwal2024Schrodinger}.

\emph{Extension to non-compact base spaces.}
In this article we have restricted the analysis to compact base spaces $(X,d)$ and the extension to the non-compact setting will be an interesting question. For this, we expect several challenges. The uniform contraction estimate of Proposition \ref{prop:Kmu_contraction} will degenerate and so bounded invertibility of $\id-K_\mu^2$ is no longer implied. The self-transport potentials $f_{\mu,\mu}$ may become unbounded and thus the measures $A(\mu)$ may no longer be finite. This poses challenges in the RKHS analysis, including the regularity of the maps $A$, $B$, the relation between the tensors $\g_\mu$ and $\tilde{\g}_\mu$, and their equivalence with the norm.

\medskip

Of course this list is far from complete and in particular more questions will emerge over time, as our understanding of the new distance deepens.

\section*{Acknowledgements}

The research leading to this work started during a visit of BS to HL and LT in the Department of Decision Sciences of Bocconi University, which is warmly acknowledged for the hospitality.
We also thank the two anonymous referees for their many suggestions which improved the quality and scope of this article.
HL and LT acknowledge the support of the MUR-Prin 2022-202244A7YL “Gradient Flows and
Non-Smooth Geometric Structures with Applications to Optimization and Machine Learning”, funded by the European Union - Next Generation EU. LT is a also member of INdAM and the GNAMPA group.
GM gratefully acknowledges funding of the DFG within CRC1456-A04.
BS gratefully acknowledges funding of the DFG within the Emmy Noether programme and in project CRC1456-A03.

\appendix

\section{Some results of functional analysis} \label{section:appendix_convergence}

\begin{lemma}\label{lemma:abstract_weak_unif}
    Let \(Y\) be a Banach space, X a compact topological space. For \(n\in\N\) let \(\omega_n\in Y^*\) with \(\omega_n \to \omega \in Y^*\) in the \textweakstar~topology. For \(x \in X\) let \(k_x \in Y\), such that \(x \mapsto k_x\) is continuous with respect to the norm-topology on \(Y\). Then 
    \[\sup_{x \in X} \abs{\pair{\omega_n}{k_x} - \pair{\omega}{k_x}} \to 0,\]
    so \(x \mapsto \pair{\omega_n}{k_x}\) converges uniformly on \(X\) as \(n \to \infty\).
\end{lemma}

\noindent We usually apply this with \(k_x = k(x,\vdot)\) for some suitable \(k \colon X \times X \to \R\) and $Y = \calM(X)$.

\begin{proof}
    Let \(\phi_n(x) = \pair{\omega_n}{k_x}\), \(\phi(x) = \pair{\omega}{k_x}\). By assumption \(\phi, \phi_n \in \Cont(X)\), \(\phi_n \to \phi\) pointwise. The functions \(\phi_n\) are equi-continuous, as for any \(n \in \N\)
    \begin{equation*}
        \abs{ \phi_n(x) - \phi_n(y) } \leq \norm{\omega_n}_{Y^*} \cdot \norm{k_x - k_y}_Y,
    \end{equation*}
    the \textweakstar-convergent sequence \((\omega_n)_n\) is bounded, and the map \(x \mapsto k_x\) is uniformly continuous. Uniform convergence now follows from the Arzelà--Ascoli Theorem.
\end{proof}

\begin{lemma}[{\cite[Proposition 3.5 (iv)]{Brezis2011}}]
    \label{lemma:pair_strong_weak}
    Let \(Y\) be a Banach space. For \(n\in\N\) let \(\omega_n\in Y^*\) with \(\omega_n\to\omega\in Y^*\) in the \textweakstar~topology. Let \(y_n \in Y\) with \(y_n \to y \in Y\) in norm. Then 
    \[\lim_{n\to\infty} \pair{\omega_n}{y_n}=\pair{\omega}{y}.\]
\end{lemma}

We also state the following easy lemma which is used in Section~\ref{section:sinkhorn_hessian}, while refinements with stronger assumptions on $k$ can be found in Appendix~\ref{section:appendix_RKHS} below.

\begin{lemma}
    \label{lemma:H_k_Cm}
    If $k \in \Cont^{(m,m)}(X \times X)$ and $H_k[\nu](y) = \langle \nu, k(\cdot,y) \rangle$, then $H_k$ is defined on $\Cont^m(X)^*$ and valued in $\Cont^m(X)$. It is \textweakstar{}-to-norm continuous, and norm-to-norm continuous and compact.
\end{lemma}

\begin{proof}
    We can prove easily $\partial^\alpha_y H_k[\nu](y) = \langle \nu, \partial^\alpha_y k(\cdot,y) \rangle$ for any multi-index $\alpha$ with $|\alpha| \leq m$ as the function $y \mapsto k(\cdot,y)$ is of class $\Cont^m$ seen as valued in the Banach space $\Cont^m(X)$. From this, if $\nu_n \to \nu$ \textweakstar{} and $y_n \to y$ then $\partial^\alpha_y H_k[\nu_n](y_n) \to \partial^\alpha_y H_k[\nu](y)$, and the convergence can be made uniform with Lemma~\ref{lemma:abstract_weak_unif}. Thus $H_k$ is \textweakstar{}-to-norm continuous, and norm-to-norm compactness follows as the unit ball of $\Cont^m(X)$ is \textweakstar{} compact.
\end{proof}

\section{Additional results on Reproducing Kernel Hilbert Spaces} \label{section:appendix_RKHS}

Recall the construction of the RKHS \(\Hil_k\) corresponding to a positive definite universal kernel \(k \in \Cont(X \times X)\) from Section~\ref{section:sub:RKHS_introduction} on a compact Hausdorff space \(X\). A detailed analysis of universal kernels can be found in~\cite{Sriperumbudur2011} and~\cite{Micchelli2006}. 
For all \(\phi \in \Hil_k\) and all \(x \in X\) it holds \(\phi(x) = \pair{\phi}{k(x,\vdot)}_{\Hil_k}\). Using the Cauchy--Schwarz inequality we find
\begin{equation} \label{eq:norm_infty_bound_RKHS}
    \norm{\phi}_\infty = \sup_{x \in X} \, \abs{\pair{\phi}{k(x,\vdot)}_{\Hil_k}} \leq \norm{\phi}_{\Hil_k} \cdot \sup_{x \in X} \, k(x,x) 
\end{equation}
for any \(\phi \in \Hil_k\).

We now fix $m \geq 0$ an integer. If \(m > 0\), we further impose the analogous version of Assumption~\ref{asp:diff_m} for \(k\). That is, we assume that \(X\) is the closure of a bounded open set in \(\R^d\) and \(k \in \Cont^{(m,m)}(X \times X)\). The case $m=0$ stands for $X$ compact metric space and $k \in \Cont^{(0,0)}(X \times X) = \Cont(X \times X)$ as above. In this case we have \(\Hil_k \hookrightarrow \Cont^m(X)\) and \(\Cont^m(X)^* \hookrightarrow \Hil^*_k\) and the embeddings are continuous \cite[Cor.~4]{SimonGabriel2018}. In this work we need a quantitative estimate on these embeddings, which we present now.

\begin{lemma}
\label{lemma:injection_Hk_Cm_quantitative}
    We have $\partial^\alpha_x k(x,\vdot) \in \Hil_k$ for any \(x \in X\) and any \(\alpha \in \N_0^d\) with \(\abs{\alpha} \leq m\). Moreover, 
    if \(\phi \in \Hil_k\) and \(\nu \in \Cont^m(X)^*\), then 
    \begin{equation*} 
        \norm{\phi}_{\Cont^m} \leq \norm{\phi}_{\Hil_k} \cdot \norm{k}_{\Cont^{(m,m)}}^\frac{1}{2}, \qquad \norm{\nu}_{\Hil^*_k} \leq \norm{\nu}_{\Cont^{m,*}} \cdot \norm{k}_{\Cont^{(m,m)}}^\frac{1}{2}.
    \end{equation*}
\end{lemma}

\begin{proof}
    Using that $H_k$ is defined on $\Hil_k^*$ and valued in $\Hil_k$ as well as \(\Cont^m(X)^* \hookrightarrow \Hil^*_k\), we have \(H_k[\partial^\alpha|_x](y) = \pair{\partial^\alpha|_x}{k(\vdot,y)} = \partial^\alpha_x k(x,y) \in \Hil_k\) with \(\norm{H_k[\partial^\alpha |_x]}_{\Hil_k}^2 = \partial^{(\alpha,\alpha)} k(x,x) \leq \norm{k}_{\Cont^{(m,m)}}\). We deduce the norm bounds from $\abs{\partial^\alpha \phi(x)} = \abs{\langle\partial^\alpha |_x, \phi \rangle} = \abs{\langle H_k[\partial^\alpha |_x], \phi \rangle_{\Hil_k}} \leq \norm{H_k[\partial^\alpha |_x]}_{\Hil_k} \norm{\phi}_{\Hil_k}$, and the one on $\norm{\nu}_{\Hil^*_k}$ by duality. 
\end{proof}

\begin{lemma}
\label{lemma:h_k_Cm_compactness}
    We have the following weak-to-norm injections.
    \begin{enumerate}
        \item The injection of $\Hil_k \hookrightarrow \Cont^m(X)$ is weak-to-norm continuous, hence norm-to-norm compact. 
        \item The operator \(H_k \colon \Cont^m(X)^* \to \Hil_k\) is \textweakstar-to-norm continuous, hence norm-to-norm compact.
    \end{enumerate}
\end{lemma}

\begin{proof}
    For the first point, let \(\phi_n \in \Hil_k\) converge weakly to \(\phi \in \Hil_k\). Then from the proof of Lemma~\ref{lemma:injection_Hk_Cm_quantitative} \(\partial^\alpha \phi_n(x) = \pair{\phi_n}{\partial_x^\alpha k(x,\vdot)}_{\Hil_k} \to \pair{\phi}{\partial_x^\alpha k(x,\vdot)}_{\Hil_k} = \partial^\alpha \phi(x)\) for all \(x \in X\). Uniform convergence for any allowed \(\alpha\) follows from Lemma~\ref{lemma:abstract_weak_unif}.
    
    For the second point let \(\nu_n \to \nu\) \textweakstar~in \(\Cont^m(X)^*\). From the continuity of the injection \( \Cont^m(X)^* \hookrightarrow \Hil^*_k\) we obtain \(H_k[\nu_n] \to H_k[\nu]\) weakly in \(\Hil_k\), which by the first point implies norm convergence of \(H_k[\nu_n]\) to \(H_k[\nu]\) in \(\Cont^m(X)\). Thus Lemma~\ref{lemma:pair_strong_weak} applied to the duality pairing between $\Cont^m(X)$ and $\Cont^m(X)^*$ shows that $\norm{H_k[\nu_n] - H_k[\nu]}_{\Hil_k}^2 = \pair{\nu_n - \nu}{H_k[\nu_n] - H_k[\nu]} \to 0$
    
    Eventually the claims of compactness follow from the Banach--Alaoglu theorem.
\end{proof}

To finish this section, we briefly show that the universality of \(k_c\) implies universality of \(k_\mu\) for all \(\mu \in \prm(X)\): 
Note that \(\exp(\fmumu/\varepsilon)\) is bounded from above and away from zero. Now for any \(\phi \in \Cont(X)\) an approximation \(\norm{\sum_{i=1}^n a_i k_c(x_i,.) - \phi/\exp(\fmumu/\varepsilon)}_\infty \leq \delta\) yields \(\norm{\sum_{i=1}^n \frac{a_i}{\exp(\fmumu(x_i)/\varepsilon)} k_\mu(x_i,.) - \phi}_\infty \leq \delta \cdot \norm{\exp(\fmumu/\varepsilon)}_\infty\).

\section{The Sobolev space of Hilbert space valued functions} \label{section:appendix_sobolev}

We briefly summarize a few results from~\cite[Sec.~3.2]{Kreuter2015}. Let \(\Hil\) be a Hilbert space. We write \(L^2((0,1);\Hil)\) for the space of measurable functions \(u \colon (0,1) \to \Hil\) for which \(\norm{u(\vdot)}_\Hil\) is 2-integrable, up to equality almost everywhere. It carries the obvious inner product. We say that a function \(u \in L^2((0,1);\Hil)\) has the distributional derivative \(v \in L^2((0,1);\Hil)\) and write \(v = u'\) if for all \(\phi \in \Contc^\infty((0,1);\R)\) it holds 
\begin{equation} \label{eq:def_H1_weakdiff}
    \int_0^1 u(t) \phi'(t) \diff t = - \int_0^1 v(t) \phi(t) \diff t
\end{equation}
as Bochner integrals. The first Sobolev space is given by 
\begin{equation*}
    \Hil^1((0,1);\Hil) = \setgiven{u \colon (0,1) \to \Hil}{u \textup{ has a distr.~derivative } u' \textup{ and } u, u' \in L^2((0,1);\Hil)}.
\end{equation*}
It is a Hilbert space with the inner product \(\pair{u_1}{u_2}_{\Hil^1} = \pair{u_1}{u_2}_{L^2} + \pair{u_1'}{u_2'}_{L^2}\).

\begin{theorem}[{\cite[Thm.~2.5]{Kreuter2015}}] \label{theorem:H1_norm_estimate}
    A measurable function \(u \colon (0,1) \to \Hil\) is Bochner integrable if and only if \(\norm{u(\vdot)}_\Hil\) is Lebesgue-integrable. In this case \(\int_0^1 u(t) \diff t \in \Hil\) and it holds
    \begin{equation*}
        \bnorm{\int_0^1 u(t) \diff t}_\Hil \leq \int_0^1 \norm{u(t)}_\Hil \diff t.
    \end{equation*}
\end{theorem}

\begin{theorem}
\label{theorem:H1_reformulation} 
    Let \(u \in L^2((0,1);\Hil)\). 
    Then the following are equivalent:
    \begin{enumerate}
        \item \label{item:H1_reformulation:H1} 
        \(u \in \Hil^1((0,1);\Hil)\).
        \item 
        \(u \in \AC^2((0,1);\Hil)\) in the sense of Definition~\ref{def:AC_curves}, 
        i.e.~there exists a representation of \(u\) and a function \(m \in L^2((0,1);\R)\), such that \(\norm{u(t) - u(s)}_\Hil \leq \abs{\int_s^t m(\tau) \diff \tau}\) for all \(t,s \in (0,1)\). 
        \item \label{item:H1_reformulation:u'_L2_compatible} 
        \(u\) is differentiable a.e., \(u' \in L^2((0,1);\Hil)\) and there is a \(t_0 \in (0,1)\), such that for almost all \(t\) it holds
            \begin{equation*}
               u(t) = u(t_0) + \int_{t_0}^t u'(\tau) \diff \tau.
            \end{equation*}
    \end{enumerate}
    If \(u \in \Hil^1((0,1);\Hil)\) then $m(\tau) = \norm{u'(\tau)}_\Hil$ is a possible choice for point 2. 
\end{theorem}

\begin{proof}
\underline{1 $\Rightarrow$ 3:} 
\cite[Thm.~3.7,~Prop.~3.8]{Kreuter2015}. 
\underline{3 $\Rightarrow$ 2:} 
Theorem~\ref{theorem:H1_norm_estimate}. 
\underline{2 $\Rightarrow$ 1:} 
\cite[Thm.~3.13]{Kreuter2015}
\end{proof}

As a consequence of absolute continuity, every \(u \in \Hil^1((0,1);\Hil)\) continuously extends to the boundary points \(\{0,1\}\) of the interval \((0,1)\).

\begin{lemma} \label{lemma:H1_weakdiff_H}
    Let \(\Hil\) be separable, \(u, u' \in L^2((0,1);\Hil)\). 
    Then \(u \in \Hil^1((0,1);\Hil)\) with distributional derivative \(u'\) in the sense of~\eqref{eq:def_H1_weakdiff} if and only if
    \begin{equation} \label{eq:lemma:H1_weakdiff_H}
        \int_0^1 \pair{u(t)}{\phi'(t)}_\Hil \diff t = - \int_0^1 \pair{u'(t)}{\phi(t)}_\Hil \diff t
    \end{equation}
    for all \(\phi \in \Contc^\infty((0,1);\Hil)\).
\end{lemma}

\begin{proof}
    Let \(\psi_1, \psi_2, \ldots \in \Hil\) be an orthonormal basis.
    By~\cite[Prop.~3.8]{Kreuter2015}, we have \(u \in \Hil^1((0,1);\Hil)\) with derivative \(u'\) if and only if \(\pair{u}{\psi_i}_{\Hil} \in \AC^2((0,1);\R) = \Hil^1((0,1);\R)\) with \(\partial_t \pair{u(t)}{\psi_i}_\Hil = \pair{u'(t)}{\psi_i}_\Hil\) for all \(i \in \N\). We first show that~\eqref{eq:lemma:H1_weakdiff_H} holds under the original definition. For any \(\omega_i \in \Contc^\infty((0,1);\R)\) we get
    \begin{align}
    \begin{split} \label{eq:proof:lemma:H1_weakdiff}
        \int_0^1 \pair{u(t)}{\partial_t(\omega_i(t) \psi_i)}_\Hil \diff t 
        &= \int_0^1 \omega_i'(t) \pair{u(t)}{\psi_i}_\Hil \diff t \\
        &= - \int_0^1 \omega_i(t) \pair{u'(t)}{\psi_i}_\Hil \diff t
        = - \int_0^1 \pair{u'(t)}{\omega_i(t) \psi_i}_\Hil \diff t.
    \end{split}
    \end{align}
    Thus~\eqref{eq:lemma:H1_weakdiff_H} holds for functions of the form \(\phi(t) = \sum_{i=1}^N \omega_i(t) \psi_i\). Such functions are dense in \(\Contc^\infty((0,1);\Hil)\) with respect to the norm on the larger space \(\Hil^1((0,1);\Hil)\).
    As the left and right-hand side of~\eqref{eq:lemma:H1_weakdiff_H} are bounded by \(\norm{u}_{L^2((0,1);\Hil)} \cdot \norm{\phi}_{\Hil^1((0,1);\Hil)}\) and \(\norm{u'}_{L^2((0,1);\Hil)} \cdot \norm{\phi}_{\Cont((0,1);\Hil)}\) respectively, the claim follows from an approximation argument. 
    
    The converse statement is obtained by the reverse argument, where we obtain equality of the middle terms in~\eqref{eq:proof:lemma:H1_weakdiff} from equality of the outer terms.
\end{proof}

\bibliographystyle{abbrv}
{\small
\bibliography{references_geometry}}

\end{document}